\documentclass[10pt,letterpaper,twoside]{article}

\usepackage[letterpaper,left=1in,right=1in,top=1.5in,bottom=1.5in]{geometry}

\title{Prismatic cohomology relative to $\delta$-rings}
\author{Benjamin Antieau, Achim Krause, and Thomas Nikolaus}
\date{\today}

\usepackage[pdfstartview=FitH,
            pdfauthor={},
            pdftitle={},
            colorlinks,
            linkcolor=reference,
            citecolor=citation,
            urlcolor=e-mail,
            bookmarks=false,
            ]{hyperref}

\usepackage{amsmath}
\usepackage{amscd}
\usepackage{amsbsy}
\usepackage{amssymb}
\usepackage{verbatim}
\usepackage{eucal}
\usepackage{microtype}
\usepackage{mathrsfs}
\usepackage{amsthm}
\usepackage{stmaryrd}
\usepackage{bbm}
\usepackage{xy}
\input xy
\xyoption{all}
\usepackage{tikz}
\usetikzlibrary{matrix,arrows}
\usepackage{tikz-cd}
\usepackage{dsfont}
\usepackage[T1]{fontenc}
\usepackage{enumitem}
\setlist{noitemsep}
\usepackage{slashed}

\usepackage{fancyhdr}
\usepackage[bbgreekl]{mathbbol}
\usepackage{amsfonts}
\DeclareSymbolFontAlphabet{\mathbb}{AMSb} 
\DeclareSymbolFontAlphabet{\mathbbl}{bbold}
\newcommand{\Prism}{{\mathbbl{\Delta}}}
\newcommand{\Prismbar}{{\overline{\mathbbl{\Delta}}}}
\newcommand{\Prismpackage}{\underline{\,\Prism\,}}

\usepackage{yhmath}

\newcommand{\WCart}{\mathrm{WCart}}
\newcommand{\FWCart}{\mathrm{FWCart}}

\newcommand{\HT}{\mathrm{HT}}
\newcommand{\prism}{{ \mathbbl{\Delta}}}

\newcommand{\wc}{\mathrm{wc}}
\newcommand{\rel}{\mathrm{rel}}
\newcommand{\perf}{\mathrm{perf}}

\newcommand{\DiffractedHodge}{\Omega^{\slashed{D}}}

\usepackage{color}
\definecolor{todo}{rgb}{1,0,0}
\definecolor{conditional}{rgb}{0,1,0}
\definecolor{e-mail}{rgb}{0,.40,.80}
\definecolor{reference}{rgb}{.20,.60,.22}
\definecolor{mrnumber}{rgb}{.80,.40,0}
\definecolor{citation}{rgb}{0,.40,.80}

\renewcommand{\bf}{\bfseries}

\let\oldmarginpar\marginpar
\renewcommand\marginpar[1]{\-\oldmarginpar[\raggedleft\footnotesize #1]%
{\raggedright\footnotesize #1}}

\newcommand{\Ascr}{\mathcal{A}}

\newcommand{\Cscr}{\mathcal{C}}
\newcommand{\Dscr}{\mathcal{D}}
\newcommand{\Escr}{\mathcal{E}}
\newcommand{\Fscr}{\mathcal{F}}
\newcommand{\Gscr}{\mathcal{G}}
\newcommand{\Hscr}{\mathcal{H}}
\newcommand{\Iscr}{\mathcal{I}}

\newcommand{\Mscr}{\mathcal{M}}

\newcommand{\Oscr}{\mathcal{O}}

\newcommand{\Qscr}{\mathcal{Q}}
\newcommand{\Rscr}{\mathcal{R}}

\newcommand{\Tscr}{\mathcal{T}}

\newcommand{\Xscr}{\mathcal{X}}
\newcommand{\Yscr}{\mathcal{Y}}

\newcommand{\A}{\mathrm{A}}

\newcommand{\B}{\mathrm{B}}

\renewcommand{\d}{\mathrm{d}}
\newcommand{\D}{\mathrm{D}}
\newcommand{\E}{\mathrm{E}}
\newcommand{\F}{\mathrm{F}}

\renewcommand{\H}{\mathrm{H}}

\newcommand{\K}{\mathrm{K}}
\renewcommand{\L}{\mathrm{L}}

\newcommand{\N}{\mathrm{N}}

\newcommand{\R}{\mathrm{R}}

\newcommand{\s}{\mathrm{site}}
\newcommand{\T}{\mathrm{T}}

\newcommand{\Z}{\mathrm{Z}}

\newcommand{\bA}{\mathbf{A}}

\newcommand{\bD}{\mathbf{D}}
\newcommand{\bE}{\mathbf{E}}
\newcommand{\bF}{\mathbf{F}}
\newcommand{\bG}{\mathbf{G}}

\newcommand{\bN}{\mathbf{N}}

\newcommand{\bQ}{\mathbf{Q}}

\newcommand{\bS}{\mathbf{S}}

\newcommand{\bZ}{\mathbf{Z}}

\renewcommand{\mathds}{\mathbbl}

\renewcommand{\SS}{\mathds{S}}

\newcommand{\op}{\mathrm{op}}
\newcommand{\cofib}{\mathrm{cofib}}
\newcommand{\fib}{\mathrm{fib}}

\newcommand{\Mod}{\mathrm{Mod}}

\newcommand{\Perf}{\mathrm{Perf}}

\newcommand{\CAlg}{\mathrm{CAlg}}
\newcommand{\DAlg}{\mathrm{DAlg}}

\newcommand{\Shv}{\mathrm{Shv}}

\newcommand{\Gr}{\mathrm{Gr}}
\newcommand{\gr}{\mathrm{gr}}

\newcommand{\st}{\mathrm{st}}
\newcommand{\an}{\mathrm{an}}
\newcommand{\doubleslash}{\mathord{\sslash}}
\newcommand{\Rees}{\mathrm{Rees}}
\newcommand{\Kos}{\mathrm{Kos}}
\newcommand{\Pairs}{\mathrm{Pairs}}
\newcommand{\tensorhat}{\widehat{\otimes}}
\newcommand{\tensor}{\otimes}

\newcommand{\site}{\mathrm{site}}

\newcommand{\can}{\mathrm{can}}
\newcommand{\xto}{\xrightarrow}

\newcommand{\id}{\mathrm{id}}

\newcommand{\im}{\mathrm{im}}
\renewcommand{\geq}{\geqslant}
\renewcommand{\leq}{\leqslant}

\newcommand{\THH}{\mathrm{THH}}

\newcommand{\TC}{\mathrm{TC}}

\newcommand{\RQRSPerfdscr}{\Rscr\Qscr\mathrm{RSPerfd}}
\newcommand{\QSynscr}{\Qscr\mathrm{Syn}}
\newcommand{\QSyn}{\mathrm{QSyn}}

\newcommand{\QRSPerfdscr}{\Qscr\mathrm{RSPerfd}}
\newcommand{\dR}{\mathrm{dR}}
\newcommand{\conj}{\mathrm{conj}}

\newcommand{\Prismhat}{\widehat{\Prism}}

\newcommand{\Map}{\mathrm{Map}}

\newcommand{\Fun}{\mathrm{Fun}}
\newcommand{\End}{\mathrm{End}}

\newcommand{\bPic}{\mathbf{Pic}}

\newcommand{\Gmhat}{\widehat{\bG}_m}
\newcommand{\Ahat}{\widehat{\bA}^1}
\newcommand{\Gm}{\bG_{m}}

\DeclareMathOperator*{\colim}{colim}

\DeclareMathOperator*{\Tot}{Tot}

\DeclareMathOperator{\Spec}{Spec}
\DeclareMathOperator{\FSpec}{FSpec}
\DeclareMathOperator{\FSpf}{FSpf}
\DeclareMathOperator{\FSSpf}{F\mathds{S}pf}
\DeclareMathOperator{\GrSpec}{GrSpec}
\DeclareMathOperator{\GrSpf}{GrSpf}
\DeclareMathOperator{\Spf}{Spf}
\DeclareMathOperator{\SSpf}{\SS pf}

\newcommand{\we}{\simeq}
\newcommand{\iso}{\cong}

\theoremstyle{plain}
\newtheorem{theorem}{Theorem}[section]
\newtheorem*{theorem*}{Theorem}
\newtheorem{lemma}[theorem]{Lemma}

\newtheorem{proposition}[theorem]{Proposition}

\newtheorem{corollary}[theorem]{Corollary}
\newtheorem*{corollary*}{Corollary}

\theoremstyle{plain}

\theoremstyle{definition}

\newtheoremstyle{named}{}{}{\itshape}{}{\bfseries}{.}{.5em}{#1 \thmnote{#3}}
\theoremstyle{named}

\theoremstyle{definition}
\newtheorem{definition}[theorem]{Definition}
\newtheorem{warning}[theorem]{Warning}
\newtheorem{variant}[theorem]{Variant}
\newtheorem{notation}[theorem]{Notation}

\newtheorem{example}[theorem]{Example}
\newtheorem*{example*}{Example}

\newtheorem*{question*}{Question}
\newtheorem{construction}[theorem]{Construction}

\newtheorem{remark}[theorem]{Remark}
\newtheorem{convention}[theorem]{Convention}

\AtBeginDocument{%
   \def\MR#1{}
}

\begin{document}

\maketitle

\begin{abstract}
    \noindent
    We develop prismatic and syntomic cohomology relative to a $\delta$-ring.
    This simultaneously generalizes Bhatt and Scholze's absolute and relative prismatic
    cohomology and shows that the latter, which was defined relative to a prism, is in fact
    independent of the prism structure and only depends on the underlying $\delta$-ring.
    We give several possible definitions of our new version of prismatic cohomology:
    a site-theoretic definition, one using prismatic
    crystals, and a stack-theoretic definition. These are equivalent under mild syntomicity hypotheses.
    As an application, we note how the theory of prismatic cohomology of filtered rings arises
    naturally in this context.


    \medskip\noindent
    {\bf Cohomologie prismatique relative aux anneaux delta.}
    Nous développons la cohomologie prismatique et syntomique relative à un anneau-$\delta$.
    Cela généralise simultanément la cohomologie prismatique absolue et relative de Bhatt et Scholze,
    et montre que cette dernière, qui était définie relativement à un prisme, est en fait
    indépendante de la structure de prisme et ne dépend que de l'anneau-$\delta$ sous-jacent.
    Nous proposons plusieurs définitions possibles de notre nouvelle version de la cohomologie
    prismatique~:
    une définition basée sur la théorie des sites, une autre utilisant les cristaux prismatiques,
    et une définition fondée sur la théorie des champs. Ces définitions sont équivalentes sous de
    légères hypothèses de syntomicité. 
    Comme application, nous montrons comment la théorie de la cohomologie prismatique des anneaux filtrés
    émerge naturellement dans ce contexte.


\end{abstract}

\setcounter{tocdepth}{1}
\tableofcontents

\section{Introduction}\label{sec:intro}

The $p$-adic syntomic complexes $\bZ_p(i)(R)$ for $i\in\bZ_{\geq 0}$ and a quasisyntomic ring $R$ are objects
of the $p$-complete derived $\infty$-category $\D(\bZ_p)_p^\wedge$ introduced by Bhatt, Morrow, and Scholze~\cite{bms2} in order to study topological cyclic homology and
algebraic $\K$-theory.
They are a form of \'etale-sheafified motivic cohomology at the `bad' prime~\cite{geisser-hesselholt}
and were studied in a different guise by Fontaine and Messing~\cite{fontaine-messing}; see~\cite{ammn} for a comparison.
Our goal in this paper is to introduce and study a relative form of the $p$-adic
syntomic complexes $\bZ_p(i)(R/A)$ when $A$ is a $\delta$-ring and $R$ is a commutative
$A$-algebra.

Syntomic cohomology is built from the theory of prismatic cohomology due to Bhatt,
Morrow, and Scholze~\cite{bms2, prisms} and we will similarly generalize prismatic cohomology to a relative setting.
We assume that the reader either knows what a prism is or
is willing to take this notion as a black box. For now we only need that (among other things) a prism
is a commutative ring $A$ with an endomorphism $\varphi_A$ and an invertible ideal $I \subseteq A$.
For a prism $(A,I)$ and a commutative $\overline{A}$-algebra $R$, where $\overline{A}=A/I$,
Bhatt and Scholze naturally associate the following objects:
\begin{enumerate}
    \setlength\itemsep{.75em}
    \item[(a)] an $\bE_\infty$-$A$-algebra $\Prism_{R/A}$, the (derived) prismatic cohomology of $R$
        relative to $A$;
    \item[(b)] $\Prism_{R/A}$-modules $\Prism_{R/A}\{i\}$ for every integer $i\in\bZ$ called the
        Breuil--Kisin twists (with $\Prism_{R/A}\{0\} = \Prism_{R/A}$) which are graded multiplicative with respect to $i$;
    \item[(c)] complete Hodge--Tate towers $\Prism_{R/A}^{[\star]}\{i\}$
        $$\cdots\rightarrow\Prism_{R/A}^{[n+1]}\{i\}\rightarrow\Prism_{R/A}^{[n]}\{i\}\rightarrow\Prism_{R/A}^{[n-1]}\{i\}\rightarrow\cdots$$
        for $i,n\in\bZ$, multiplicative in $i$ and $n$, with weight $n=0$ part
        $\Prism_{R/A}\{i\}$, underlying object $\Prism_{R/A}\{i\}[1/I]$, and associated graded
        pieces
        $\gr^n\Prism_{R/A}^{[\star]}\{i\}\we\Prismbar_{R/A}\{i+n\}$, the $(i+n)$-Breuil--Kisin twisted
        Hodge--Tate cohomology of $R$ over $A$;
    \item[(d)] complete $\bZ$-filtered $A$-modules 
    \[
    \cdots\rightarrow \N^{\geq n+1}\Prismhat_{R/A}^{(1)}\{i\} \to \N^{\geq
        n}\Prismhat_{R/A}^{(1)}\{i\} \to \N^{\geq n-1}\Prismhat_{R/A}^{(1)}\{i\} \rightarrow\cdots
    \]
    for $i,n\in\bZ$, multiplicative in $i$ and $n$, called the 
    Nygaard filtration on the Nygaard-completed, Frobenius-twisted, Breuil--Kisin twisted prismatic
        cohomology $\Prismhat_{R/A}^{(1)}\{i\}$, which have the property that
        $\N^{\geq\star}\Prismhat_{R/A}^{(1)}\{i\}$ is constant for $\star \leq 0$;
    \item[(e)]  a $\varphi_A$-semilinear morphism $c\colon \Prism_{R/A}\{i\} \to \Prismhat_{R/A}^{(1)}\{i\}$ for every $i \in \bZ$;
    \item[(f)] a filtered relative Frobenius map
        $\N^{\geq\star}\Prismhat_{R/A}^{(1)}\{i\}\xrightarrow{\varphi_{/A}}
        \Prism_{R/A}^{[\star-i]}\{i\}$ for every  $i \in \bZ$ (the variable $\star$ is the filtration).
\end{enumerate}
We will not give a detailed review of the definition of these objects here and just note that the functor $\Prism_{R/A}$ can be constructed by taking the cohomology of the prismatic site of $R$
relative to $A$ when $R$ is $p$-adically formally smooth over $\overline{A}$, and then left Kan
extending in the $\infty$-category of $(p,I)$-complete $\bE_\infty$-$A$-algebras from formally smooth $\overline{A}$-algebras. 

    Let $\Prismhat_{R/A}^{(1)}=\Prismhat_{R/A}^{(1)}\{0\}$. In this case, the morphisms $c$ and
$\varphi_{/A}$ are maps of $\bE_\infty$-rings and the composition 
\[
\Prism_{R/A} \xto{c} \Prismhat_{R/A}^{(1)} \xto{\varphi_{/A}} \Prism_{R/A}
\]
induces a $\varphi_A$-semilinear endomorphism $\varphi$ that we call the (absolute)
Frobenius map.

We want to explain in which sense $\Prismhat_{R/A}^{(1)}$ is a completion to explain the
terminology: using the map $c$ one can `pull-back' the Nygaard filtration to get a non-complete
filtration on the Frobenius twisted prismatic cohomology
$\Prism_{R/A}^{(1)}\{i\}:=\Prism_{R/A}\{i\}\tensor_{A,\varphi_A}A$. 
Concretely this new filtration is given by the formula
\[
\N^{\geq\star}\Prism_{R/A}^{(1)}\{i\} := \N^{\geq\star}\Prismhat_{R/A}^{(1)}\{i\}
\times_{\Prismhat_{R/A}^{(1)}\{i\},\tilde c} \Prism_{R/A}^{(1)}\{i\},
\]
where $\tilde c$ is the $A$-linear map $\Prism_{R/A}^{(1)}\{i\} \to \Prismhat_{R/A}^{(1)}\{i\}$ induced from $c$. 
By construction the completion of $\Prism_{R/A}^{(1)}\{i\}$ with respect to
$\N^{\geq\star}\Prism_{R/A}^{(1)}\{i\}$ is Nygaard-completed prismatic
cohomology $\Prismhat_{R/A}^{(1)}\{i\}$.

\begin{definition} The objects $(a) - (f)$ are naturally objects of the $\infty$-category $\mathcal{C}_A$ consisting of quadruples $(H,N,c,\varphi)$ where 
    \begin{itemize}
    \item
    $H$ is a complete filtered, graded $\bE_\infty$-algebra over $A$, i.e., a lax symmetric monoidal functor 
    \[
    H\colon (\bZ, \geq ) \times \bZ \to \D(A),
    \]
    where $(\bZ,\geq)$ denotes the poset with the additive symmetric monoidal structure and $\bZ$ denotes the set with the additive symmetric monoidal
            structure. Completeness means that for fixed grading the inverse limit in the poset direction vanishes;
    \item
    $N$ is another complete filtered, graded  $\bE_\infty$-algebra over $A$ which is constant in
            non-positive filtration weights;

    \item $c$ is a  $\varphi_A$-semilinear map of graded $A$ algebras $H^{\geq 0} \to N^{\geq 0}$
        with certain extra structure which we will suppress for simplicity here;
    %
    \item $\varphi$ is a map of graded filtered algebras $N \to \mathrm{sh}(H)$ over $A$ where $\mathrm{sh}$ is the shearing of $H$ defined as  
    changing the filtration degree by subtracting the grading degree. 
    \end{itemize}
\end{definition}

There is an absolute form of the theory also introduced in \cite{prisms} and studied
in \cite{bhatt-lurie-apc}, yielding for commutative rings $R$ the following
objects:
\begin{enumerate}
    \setlength\itemsep{.75em}
    \item[(i)] an $\bE_\infty$-algebra $\Prism_{R}$ over $\bZ_p$;
    \item[(ii)] Breuil--Kisin twists $\Prism_{R}\{i\}$ for every integer $i\in\bZ$;
    \item[(iii)] complete Hodge--Tate towers $\Prism_{R}^{[\star]}\{i\}$ with weight $0$ part
        $\Prism^{[0]}_{R}\{i\}\we\Prism_R\{i\}$ and associated graded pieces
        $\gr^n\Prism_R^{[\star]}\{i\}\we\Prismbar_R\{i+n\}$, the $(i+n)$-Breuil--Kisin twisted
        absolute Hodge--Tate cohomology of $R$;
    \item[(iv)] Nygaard filtrations $\N^{\geq\star}\Prismhat_{R}\{i\}$;
    \item[(v)] maps $c\colon \Prism_{R}\{i\} \to \Prismhat_{R}\{i\}$;
    \item[(vi)] filtered Frobenius maps
        $\N^{\geq\star}\Prismhat_{R}\{i\}\xrightarrow{\varphi}
        \Prism_{R}^{[\star-i]}\{i\}$.
\end{enumerate}
These objects are naturally objects of the $\infty$-category $\mathcal{C}_{\bZ_p}$ of the previous
definition, where $\varphi_{\bZ_p} = \id$.  
The $p$-adic syntomic complexes are defined for $i\in\bZ$ as
\begin{enumerate}
    \item[(vii)] $\bZ_p(i)(R)=\fib\left(\N^{\geq
        i}\Prismhat_{R}\{i\}\xrightarrow{\can-c\varphi}\Prismhat_{R}\{i\}\right)$.
\end{enumerate}
For example, $\bZ_p(0)\we\lim\bZ/p^m$ as a flat sheaf and $\bZ_p(1)(R)\we\T_p\Gm$, the $p$-adic Tate
module of $\Gm$, while $\bZ_p(i)\we 0$ for $i<0$; see~\cite[Props.~6.16 and~6.17]{bms2}.

\subsection{Results}

To generalize syntomic cohomology to the case of a commutative algebra $R$ with bounded $p$-primary
torsion over a $\delta$-ring
$A$ with bounded $p$-primary torsion (called a bounded $\delta$-pair $(A,R)$), we extend prismatic cohomology and the structures
(a)-(f) above to the setting of commutative $R$-algebras over general $\delta$-rings $A$. 
Say that a $\delta$-pair $(A,R)$ is finitely presented free if $A$ is a finitely generated free
$\delta$-ring and $R$ is a finitely generated polynomial ring over $A$.
The main result of this paper is the following.

\begin{theorem}\label{thm:intro}
    There is a extension of prismatic cohomology to bounded $\delta$-pairs. More precisely we construct objects 
    $$
    \underline{\,\Prism\,}_{R/A} = 
    \left(\Prism_{R/A}^{[\star]}\{\star\},\N^{\geq\star}\Prismhat_{R/A}^{(1)}\{\star\},c,\varphi\right) \in \mathcal{C}_A
    $$
   depending functorially on $(R,A)$ which possess the following properties.
    \begin{enumerate}
        \setlength\itemsep{.75em}
        \item[{\em (1)}]\label{item:eins} {\em Insensitivity to localization and completion
            (Definition~\ref{def:crystal},
            Corollary~\ref{cor:inv_loc_comp}, and Corollary~\ref{cor:nygaardproperties}):} 
         $\underline{\,\Prism\,}_{R/A}$
            depends only on the $S$-localization of $A$ where $S$ is any set of elements of $A$
            that become invertible in $R$; similarly, $\underline{\,\Prism\,}_{R/A}$
            depends only on the derived $(p,K)$-adic completion of $A$ and the derived
            $p$-completion of $R$ where $K$ is any finitely generated ideal in the kernel of $A \to R$.
        \item[{\em (2)}]\label{item:zwei} {\em Relative prismatic comparison (Propositions~\ref{prop:comparison},
           \ref{prop:push}, \ref{prop:prismaticcomparison} and~\ref{prop:nygaardcomparison}, ~\cite[Thm.~7.17]{bhatt-lurie-apc}):} if $A$ is a prism and $R$ is an
            $\overline{A}$-algebra, then
            $\underline{\,\Prism\,}_{R/A}$
            agrees naturally with the derived prismatic cohomology of Bhatt and Scholze~\cite{prisms}.
        \item[{\em (3)}] {\em Absolute prismatic comparison (Example~\ref{exs:wc} and Remark \ref{rem:absolute_nygaard_comparison}):} if $A=\bZ_p$ or if $A=W(k)$ for a perfect
            $\bF_p$-algebra $k$ and if $R$ is an $A$-algebra, then
            $\underline{\,\Prism\,}_{R/A}$
            agrees naturally with the absolute prismatic cohomology of Bhatt and Lurie~\cite{bhatt-lurie-apc};
            note in this case that the Frobenius twist is equivalent to $\Prism_{R/A}$.
        \item[{\em (4)}] {\em Preservation of sifted colimits (Corollary~\ref{cor:kanextended} and \ref{cor:nygaardproperties}):} for varying $A$ and $R$,
            $\underline{\,\Prism\,}_{R/A}$
            is left Kan extended from finitely presented free $\delta$-pairs (as functors to
            $\mathcal{C}_\bZ$, i.e., we get colimits in complete filtrations).
        \item[{\em (5)}] {\em Quasisyntomic descent (Proposition~\ref{prop:cwdescent},
            Definition~\ref{def:relqsyn}, and
            Corollary~\ref{cor:nygaardproperties}):} for fixed $A$,
            $\underline{\,\Prism\,}_{R/A}$
            satisfies $p$-completely quasisyntomic descent in relatively quasisyntomic
            $\delta$-pairs $(A,R)$, see Definition \ref{def:relqsyn} .
        \item[{\em (6)}] {\em Quasismooth descent in $A$ (Proposition~\ref{prop:cwdescent}
            and Corollary~\ref{cor:nygaardproperties}):} for fixed $R$,
            $\underline{\,\Prism\,}_{R/A}$ satisfies $p$-completely quasismooth descent in
            relatively quasisyntomic $\delta$-pairs $(A,R)$.
        \item[{\em (7)}]
            {\em Base change (Corollary~\ref{cor:kanextended} and Corollary \ref{cor:nygaardproperties}):}
            $\underline{\,\Prism\,}_{R/A}$
            satisfies base change in $A$: for any map $A \to A'$ of $\delta$-rings the natural map $\Prism_{R/A}\tensor_A
            A'\rightarrow\Prism_{R\otimes_AA'/A'}$ is an equivalence after completing the left-hand
            side with respect to the Hodge--Tate filtration $\Prism_{R/A}^{[\star]}\tensor_AA'$ and
            similarly for the Nygaard-completed term.
        \item[{\em (8)}] {\em Restriction along relatively perfect maps
          (Proposition~\ref{prop:quasietalecrystal} and Corollary~\ref{cor:nygaardproperties}):} if $A\rightarrow A'$ is a $p$-adically relatively perfect map of
            bounded $\delta$-rings, then the map
            $\underline{\,\Prism\,}_{R/A} \to \underline{\,\Prism\,}_{R/A'}$
            is an equivalence for any bounded commutative $A'$-algebra $R$.
    \end{enumerate}
   The functor $(A,R) \mapsto \underline{\,\Prism\,}_{R/A}$ is uniquely determined by these properties (Proposition \ref{prop_unique}). 
\end{theorem}

Note that the previous theorem also implies that Bhatt--Scholze's relative prismatic cohomology only depends on the $\delta$-ring $A$ and not the prism structure, which is slightly surprising. 
As a consequence of Theorem~\ref{thm:intro}, there is an extension of the relative $p$-adic syntomic
complexes $\bZ_p(i)(R/A)$ to $\delta$-pairs. If $A\rightarrow R$ is a $\delta$-pair, we
let 
\[
\bZ_p(i)(R/A)=\fib\left(\N^{\geq
i}\Prismhat_{R/A}^{(1)}\{i\}\xrightarrow{\can-c\varphi}\Prismhat_{R/A}^{(1)}\{i\}\right)
\]
for $i\in\bZ$, the $i$th relative syntomic complex of $R$ over $A$.

It turns out that Nygaard-completed Frobenius twisted cohomology $\N^{\geq\star}\Prismhat_{R/A}$
satisfies a stronger form of descent for morphisms of $\delta$-pairs $(A,R)\rightarrow(B,S)$, requiring only
that $R\rightarrow S$ be of universal descent, meaning that for every map $R\rightarrow T$ of
animated commutative rings the limit of the 
\v{C}ech complex of $T\rightarrow T\otimes_RS$ computes $T$. However, note that this generality
requires prismatic cohomology relative to animated $\delta$-rings, which we will return to in
future work.
To avoid this, in Corollary~\ref{cor:intro}, one can assume that $(A,R)\rightarrow (B,S)$ be a map
of bounded $\delta$-pairs where $A\rightarrow B$ is flat and $R\rightarrow S$ is faithfully flat.

\begin{corollary}[Descent for syntomic cohomology (Corollary~\ref{cor:syntomicdescent})]\label{cor:intro}
    For each $i\in\bZ$,
    the relative $p$-adic syntomic complexes $\bZ_p(i)(R/A)$ satisfy descent for maps of pairs
    $(A,R)\rightarrow (B,S)$ such that $R\rightarrow S$ is
    a universal descent morphism (with no condition on $A \to B$).
\end{corollary}

Adapting an argument from~\cite[Lem.~7.22]{bms2} (see also~\cite[Thm.~5.1(2)]{ammn}), we also find that the relative $p$-adic syntomic complexes
$\bZ_p(i)(R/A)$ preserve sifted colimits and hence are left Kan extended from their values on
finitely presented free $\delta$-pairs.

\begin{corollary}[Left Kan extension of syntomic cohomology
    (Corollary~\ref{cor:syntomickanextension})]\label{cor:intro_kan}
    For each $i\in\bZ$,
    the relative $p$-adic syntomic complexes $\bZ_p(i)(R/A)$ are left Kan extended from finitely
    presented free $\delta$-pairs (as a functor to the $p$-complete derived category
    $\D(\bZ)_p^\wedge$).
\end{corollary}

\begin{example}
    If $A=\bZ$, the initial $\delta$-ring, then by Theorem~\ref{thm:intro}(3), $\Prism_{R/\bZ}$ recovers absolute prismatic
    cohomology $\Prism_R$ and $\N^{\geq\star}\Prismhat^{(1)}_{R/\bZ}$ agrees with
    $\N^{\geq\star}\Prismhat_R$. It follows that $\bZ_p(i)(R/\bZ)$
    agrees with the $p$-adic syntomic complexes defined in~\cite{bms2}. The Frobenius endomorphism of $\bZ$ is the identity,
    which explains why no Frobenius twists appear for either the Nygaard filtration or the $p$-adic
    syntomic complexes in the absolute case.
\end{example}


\subsection{Construction}

One can use the descent properties of Theorem~\ref{thm:intro} to reduce to the derived prismatic
cohomology of~\cite{prisms}.
However, this makes it very hard to verify all the  properties stated in Theorem \ref{thm:intro} and also makes it hard to compare it to different definitions. Therefore we use the theory of prismatic crystals to give a direct definition. Prismatic crystals are equivalent to quasi-coherent sheaves on  the
Cartier--Witt stack $\WCart$ of Drinfeld~\cite{drinfeld-prismatization} and Bhatt--Lurie~\cite{bhatt-lurie-apc}. For every bounded $\delta$-pair $(R,A)$ we introduce prismatic crystals $\Hscr^{[n]}_\Prism(R/A)\{i\}$ and set
\begin{align*}
    \Prism_{R/A}^{[n]}\{i\}&=\R\Gamma(\WCart,\Hscr_\Prism^{[n]}(R/A)\{i\}).
\end{align*}
Concretely the value of $\Hscr^{[n]}_\Prism(R/A)\{i\}$ on a transversal prism $(B,J)$ is given by
derived relative prismatic cohomology
$\Prism^{\rel,[n]}_{R\tensorhat\overline{B}/A\tensorhat B}\{i\}$ as constructed in~\cite{prisms}
(with notation from~\cite{bhatt-lurie-apc}).
Here $A \tensorhat B$ denotes the prism obtained by $(p,J)$-adically completing the $\delta$-ring
$A \otimes B$ and giving it the prism structure induced from $J$, while $R\tensorhat\overline{B}$ is
the $p$-completion of $R \otimes \overline{B} = R \otimes B/J$.

Using this description many of the properties of relative prismatic cohomology, such as the
following Hodge--Tate comparison, follow immediately from
properties of prismatic cohomology relative to prisms. For fixed $i$ we let $\Prismbar_{R/A}\{i\}$ be the zero'th graded of the Hodge-Tate tower $\Prism_{R/A}^{[\star]}\{i\}$, i.e. $\Prism_{R/A}^{[0]}\{i\} / \Prism_{R/A}^{[1]}\{i\}$.

\begin{proposition} {\em Hodge--Tate comparison (Proposition~\ref{prop:conjugate}):}
    There is a natural exhaustive increasing
    filtration $\F_{\leq\star}^{\delta\conj}\Prismbar_{R/A}\{i\}$ on $\Prismbar_{R/A}\{i\}$, called the $\delta$-conjugate
    filtration, whose graded pieces are given as
    \[
    \gr_u^{\delta\conj}\Prismbar_{R/A}\{i\}\we\L\Omega^u_{R/A}\tensor\fib(\bZ_p\xrightarrow{i-u}\bZ_p)[-u]
    \]
    for
    all $u\in\bZ$.
 \end{proposition}      

Our conjugate filtration does not agree with the conjugate filtration
constructed in~\cite{prisms} in the prismatic case; see Warning~\ref{warn:conjugate}.
Rather it is an analog of the conjugate filtration on absolute prismatic cohomology
studied in~\cite{bhatt-lurie-apc}.

The construction of the Nygaard-completed prismatic cohomology and the Nygaard filtration will also
be given using the language of prismatic crystals as is done in~\cite[Sec.~5.5]{bhatt-lurie-apc}.

Finally we note that one can also give other descriptions of relative prismatic cohomology that are
often more accessible or easier to describe.

\begin{theorem}
    \begin{enumerate}
        \item[{\em (1)}] {\em Site-theoretic comparison (Theorem~\ref{thm:equivalence}):} when $\L_{R/A}$ has $p$-complete
            Tor-amplitude in $[0,1]$, $\Prism_{R/A}$ has a site-theoretic description, that is it agrees with the cohomology of the relative prismatic site
            $(R/A)_\Prism$ as defined in Section \ref{sec:site}.
        \item[{\em (2)}] {\em Comparison to prismatization (Proposition~\ref{prop:push}):} for each bounded $\delta$-pair
            $(R/A)$, there is a formal stack $\WCart_{R/A}$ over $\WCart$. The pushforward of the
            structure sheaf $\Oscr_{\WCart_{R/A}}$ along $\WCart_{R/A} \to \WCart$ agrees with $\Hscr_\Prism(R/A)$ when $\L_{R/A}$
            has $p$-complete Tor-amplitude in $[0,1]$. In particular in this situation the global
            sections of the structure sheaf of $\WCart_{R/A}$ is equivalent to $\Prism_{R/A}$.
    \end{enumerate}
\end{theorem}

\subsection{Applications}

The idea that there should be relative $p$-adic syntomic complexes in the generality of this paper is motivated
by topological considerations such as applying $\TC$ to $\THH(R/\bS[z])$ where $z$ maps to some
element of $R$. In this case, the complexes could have been defined in~\cite[Sec.~11]{bms2}. The
idea that
the relative syntomic complexes should satisfy descent results as in Corollary~\ref{cor:intro} is
heavily influenced by the work of Liu and Wang~\cite{liu-wang} on $\TC_*(\Oscr_K;\bF_p)$, in which the
authors recover Hesselholt and Madsen's verification~\cite{hesselholt-madsen} of the Quillen--Lichtenbaum conjecture for
local fields, and
by~\cite{krause-nikolaus} on $\THH$ of (quotients of) discrete valuation rings. These papers both
take a topological approach to calculations using descent, so it was natural to look for a purely
prismatic approach. The explicit connection between the two approaches will be
studied in future work in the context of the modern approach to cyclotomic spectra
developed in~\cite{nikolaus-scholze}.

In~\cite{akn-kzpn} (see~\cite{akn-announcement} for a survey), we use Corollary~\ref{cor:intro}(d) for $R=\Oscr_K/\varpi^n$ when $K$ is
a finite extension of $\bQ_p$ with uniformizer $\varpi$ and residue field $k$ to show that the absolute $p$-adic syntomic
complexes $\bZ_p(i)(\Oscr_K/\varpi^n)$
can be computed from descent along $W(k)\rightarrow W(k)\llbracket z\rrbracket$ using the relative
$p$-adic syntomic syntomic complexes. The point is that $\bZ_p(i)(\Oscr_K/\varpi^n/W(k)\llbracket
z\rrbracket)$ admits a purely algebraic description in terms of the prismatic envelopes
introduced in~\cite{prisms}.

In order to make this approach amenable to computer
calculations, we use the $\varpi$-adic filtration and argue that one has to compute
$\bZ_p(i)(\Oscr_K/\varpi^n)$ only up to finite filtration level. To make this precise, we introduce
filtered prismatic cohomology below, following a suggestion of Bhatt. It turns out that this is most
naturally viewed as a specific form of prismatic cohomology relative to a $\delta$-stack,
$\Ahat/\Gmhat$, so it fits naturally into the context of the present paper.
We use filtered prismatic cohomology to prove the following result.

\begin{theorem}[Crystalline degeneration]
    If $\F^{\geq\star} R$ is a filtered $p$-complete commutative ring with $\F^0 R=R$ and which is
    constant for $\star\leq 0$, then the syntomic complexes of $R$ and $\gr^\star R$
    admit natural filtrations $\F^{\geq\star}\bZ_p(i)(R)$ and $\F^{\geq\star}\bZ_p(i)(\gr^\star R)$
    and there are natural identifications
    $\gr^\star\bZ_p(i)(R)\we\gr^\star\bZ_p(i)(\gr^\star R)$.
\end{theorem}

We call this phenomenon crystalline degeneration
as in the cases of interest to us in~\cite{akn-kzpn}, which is the ring $\bZ/p^n$ with the $p$-adic filtration, we have
$p\in\F^{\geq 1} R$ so that the associated graded is an $\bF_p$-algebra.

\paragraph{Outline.}
We give three different constructions of prismatic cohomology relative to a $\delta$-ring $A$ in
Sections~\ref{sec:site},~\ref{sec:crystal}, and~\ref{sec:prismatization}: one
is site-theoretic and extends the prismatic site
of~\cite{prisms}; one is by constructing prismatic crystals $\Hscr_\Prism(R/A)$ on the stack
$\WCart$ of Bhatt--Lurie~\cite{bhatt-lurie-apc} and Drinfeld~\cite{drinfeld-prismatization}. The
final method is via a prismatization and gives
relative Cartier--Witt stacks $\WCart_{R/A}$, extending the definition of the relative
prismatization from~\cite{bhatt-lurie-prism}. We compare these approaches under mild hypotheses in
Section~\ref{sec:comparisons} and we introduce the Frobenius twisted variants and the Nygaard
filtration in Section~\ref{sec:nygaard}. We discuss syntomic
cohomology in Section~\ref{sec:syntomic} and we wrap it all up and discuss the prismatic package in
Section~\ref{sec:package}. We discuss the use of relative quasisyntomic descent to compute prismatic
cohomology relative to $\delta$-rings in Section~\ref{sec:descent} and prove the uniqueness statement
of Theorem~\ref{thm:intro} there.
In Section~\ref{sec:filtered}, we explain how to work over the $p$-completion of $\bA^1/\Gm$ to construct filtered variants.
Finally, we explain how the theory of prismatic cohomology relative
to $\delta$-rings gives a natural way to extend the theory of prismatic cohomology to
filtered and graded rings by working over the stack $\WCart_{\bA^1/\Gm}$.
In Appendix~\ref{app}, we give background on the theory
of quasi-coherent sheaves on formal stacks and prove some results on base change for quasi-coherent
cohomology needed for the computation of prismatic crystals.

\paragraph{Background.} We will freely use the theory of prismatic cohomology as developed by
Bhatt, Morrow, and Scholze~\cite{bms2}, Bhatt and Scholze~\cite{prisms}, and Bhatt and
Lurie~\cite{bhatt-lurie-apc,bhatt-lurie-prism}. In particular, we use the notions of quasisyntomic
rings from~\cite{bms2}, $\delta$-rings from~\cite{prisms} (and going back
to~\cite{joyal-delta,joyal-witt}), animated $\delta$-rings and animated prisms
from~\cite{bhatt-lurie-prism}, generalized Cartier--Witt divisors
from~\cite{bhatt-lurie-apc,bhatt-lurie-prism}, and the prismatic site from~\cite{prisms} as well as
its animated analogue from~\cite{bhatt-lurie-prism}. We also use animated commutative and derived
commutative rings (as defined by Mathew), see \cite{raksit} for the latter.


\paragraph{Notation.} Throughout this paper, we fix a prime number $p$. All $\delta$-ring theoretic
notions are taken with respect to $p$ and quasisyntomic will be shorthand for quasisyntomic with
respect to the prime $p$. Thus, for example, a map $A\rightarrow A'$ of commutative rings is
quasisyntomic if $A'$ is $p$-completely flat over $A$ and $\L_{A'/A}$ has $p$-complete Tor-amplitude
in $[0,1]$ (where we index throughout using the homological convention). Note that our
definition of quasisyntomic is not the same as that found in~\cite{bms2}, which requires the rings
themselves to be derived $p$-complete, but agrees with the notion of $p$-quasisyntomic morphisms
found in~\cite[Def.~C.9]{bhatt-lurie-apc}.
A commutative ring is bounded if it has bounded $p$-power torsion.
A map $A\rightarrow A'$ of commutative rings is $p$-completely quasismooth if $\L_{A'/A}$ has
$p$-complete Tor-amplitude in $[0,0]$.
A map $A\rightarrow A'$ of commutative rings is $p$-completely quasi-\'etale if $\L_{A'/A}$ vanishes $p$-adically.

\paragraph{Acknowledgments.} We thank Bhargav Bhatt, Adam Holeman, Deven Manam, and Akhil Mathew for helpful
conversations about this material and Johannes Ansch\"utz and Noah Riggenbach for extensive comments on a
draft. We also enjoyed two productive visits to Oberwolfach when many of
the details in this paper were worked out. We thank MFO and its staff for providing such an ideal
setting for research. We also thank the referees for extensive comments which have helped
us substantially improve the paper.

This paper also benefited from the opportunity of two of its authors to give a masterclass on the
topic in Copenhagen in early 2023. We thank the organizers (Shachar Carmeli, Lars Hesselholt, Ryomei
Iwasa, and Mikala Jansen) and the
Copenhagen Centre for Geometry and Topology for the opportunity.

The first author was supported by NSF grants DMS-2120005, DMS-2102010, and DMS-2152235 and by Simons Fellowship 666565;
he would like to thank Universit\"at M\"unster for its hospitality during a visit in 2020. 
The second and third author were funded by the Deutsche Forschungsgemeinschaft
(DFG, German Research Foundation) – Project-ID 427320536 – SFB 1442, as well as
under Germany’s Excellence Strategy EXC 2044 390685587, Mathematics Münster:
Dynamics–Geometry–Structure. They would also like to thank the Mittag--Leffler
Institute for its hospitality while working on this project.

\section{The prismatic site}\label{sec:site}

In this first section we define a relative prismatic site of a $\delta$-pair $(A,R)$. The
cohomology of this site with respect to
the structure sheaf $\Oscr_\Prism$ is a version of relative prismatic cohomology which we call the
\emph{site-theoretic (relative) prismatic cohomology} $\Prism_{R/A}^{\s}$. We will see later that
the site-theoretic prismatic cohomology agrees on a large class of rings with the `correct'
prismatic cohomology. We begin this paper with the site-theoretic variant since it is the easiest
one to understand conceptually.

\begin{definition}[$\delta$-pairs]
    Let $\Pairs^\delta$ be the category consisting of $\delta$-pairs $(A,R)$ (or $A\rightarrow R$), meaning a $\delta$-ring $A$
    and a commutative $A$-algebra $R$. The morphisms $(A,R)\rightarrow(A',R')$ consist of
    commutative diagrams
    $$\xymatrix{
        A\ar[r]\ar[d]&A'\ar[d]\\
        R\ar[r]&R',
    }$$ where $A\rightarrow A'$ is a $\delta$-ring map.
    In particular, we emphasize that
    $R$ is not equipped with a $\delta$-ring structure, in fact it typically does not even admit a (natural) $\delta$-ring structure.
    In general, we say a $\delta$-pair is bounded, (derived) $p$-complete, or quasisyntomic if $A$ and $R$ are
    bounded, (derived) $p$-complete, or quasisyntomic as commutative rings.
\end{definition}

\begin{definition}[Pre-prismatic $\delta$-pairs]\label{prismatic}
    Say that a $\delta$-pair $(A,R)$ is
    \emph{pre-prismatic} if the kernel of $A\rightarrow R$ contains an invertible ideal $I$ such that
    Zariski locally on $\Spec A$ any generator $d$ of $I$ has the property that $\delta(d)$ maps to a
    unit in $R^\wedge_p$. A pre-prismatic $\delta$-pair is prismatic if there is such an ideal $I$ such that $(A,I)$ is a prism. 
\end{definition}

\begin{remark}
    Note that if $(A,R)$ is a pre-prismatic bounded $\delta$-pair, as exhibited by a Cartier divisor $I$, then
    $A[\delta(I)^{-1}]_{(I,p)}^\wedge$ inherits a $\delta$-ring
    structure by~\cite[Rem.~2.16, Lem.~2.18]{prisms} and becomes a prism with respect to the completion of $I$.
    By property (1) of Theorem~\ref{thm:intro}, the prismatic cohomology will not
    see the difference between $(A,R)$ and $(A[\delta(I)^{-1}]^\wedge_{(I,p)}, R^\wedge_p)$.
\end{remark}


\begin{definition}[The prismatic site]
    Let $A\rightarrow R$ be a $\delta$-pair.
    The prismatic site $(R/A)_\Prism$ of $R$ relative to the $\delta$-ring $A$ is the opposite of the
    category of commutative squares
    \begin{equation}\label{eq:site}\begin{gathered}\xymatrix{
        A\ar[r]\ar[d]&B\ar[d]\\
        R\ar[r]&\overline{B}
    }\end{gathered}\end{equation}
    where $(B,I)$ is a bounded prism, $\overline{B}=B/I$, and $A\rightarrow B$ is a map of
    $\delta$-rings; we equip this category with the $(p,I)$-completely faithfully flat topology in $B$.
    The proof that $(R/A)_\Prism$ is indeed a site is the same as the first paragraph of the proof
    of~\cite[Cor.~3.12]{prisms}.
\end{definition}

\begin{remark}[Comparison to other prismatic sites]\label{rem:sitecomparison}
    \begin{enumerate}
        \item[(a)]
            If $(A,J)$ is itself a bounded prism and the map $A\rightarrow R$ factors through
            $\overline{A}=A/J$, then
            $(R/A)_\Prism$ agrees by definition with the relative prismatic site of $R$ over $A$ defined
            in~\cite[Def.~4.1]{prisms}. This follows since $A\rightarrow B$
            has to take $J$ to the ideal $I = \ker(B \to \overline{B})$ by
            commutativity of~\eqref{eq:site}.
        \item[(b)] If $A=\bZ_p$, then $(R/\bZ_p)_\Prism$ agrees with the absolute prismatic site $(R)_\Prism$ of $R$ defined
            in~\cite[Def.~4.4.27]{bhatt-lurie-apc} since $\bZ_p$ is initial as a $p$-complete $\delta$-ring. In particular, in the special case of
            $(\bZ_p/\bZ_p)_\Prism$, we recover the site of all bounded prisms.
    \end{enumerate}
\end{remark}

\begin{definition}[Breuil--Kisin twists]
    Each prism $B$ admits a natural invertible module $B\{1\}$ constructed
    in~\cite[Sec.~2]{bhatt-lurie-apc} called the first Breuil--Kisin twist of $B$. It is the home
    for the prismatic logarithm map and is the prismatic analogue of the Tate twist
    $\bZ_p(1)$ in the \'etale setting. It is also compatible with base change: if $B\rightarrow B'$
    is a map of prisms, then there is a natural isomorphism $B'\otimes_B B\{1\}\iso
    B'\{1\}$. In the language of~\cite{bhatt-lurie-apc}, the assignment $B\mapsto B\{1\}$
    defines a prismatic crystal, or a quasi-coherent sheaf on the Cartier--Witt stack $\WCart$.
    The $i$th Breuil--Kisin twist is the tensor power $B\{i\}=B\{1\}^{\otimes_B i}$, which is
    defined for $i \in \mathbb{Z}$ and not just for $i \in \mathbb{N}$ since $B\{1\}$ is invertible. 
    These  define presheaves $\Oscr_\Prism\{i\}$ on $(R/A)_\Prism$ by sending a
    square~\eqref{eq:site} to $B\{i\}$. The presheaf $\Oscr_\Prism=\Oscr_\Prism\{0\}$ is a sheaf of
    $\delta$-rings and is called the prismatic structure sheaf, terminology which will be
    justified below. Given any sheaf of $\Oscr_\Prism$-modules $\Mscr$ on $(R/A)_\Prism$, we denote by $\Mscr\{i\}$
    the tensor product $\Mscr\otimes_{\Oscr_\Prism}\Oscr_\Prism\{i\}$ in sheaves of abelian groups
    on $(R/A)_\Prism$.
\end{definition}

\begin{definition}[The Hodge--Tate tower]
    The assignment to~\eqref{eq:site} of the prismatic ideal $I$ defines an invertible sheaf of ideals
    $\Iscr_\Prism$ in $\Oscr_\Prism$ called the Hodge--Tate ideal. We let $\Oscr_\Prism^{[n]}\{i\}=\Iscr_\Prism^{\otimes
    n}\otimes_{\Oscr_\Prism}\Oscr_\Prism\{i\}$ for $n\in\bZ$. These assemble into a tower
    $$\cdots\rightarrow\Oscr_\Prism^{[2]}\{i\}\rightarrow\Oscr_\Prism^{[1]}\{i\}\rightarrow\Oscr_\Prism^{[0]}\{i\}\rightarrow\Oscr_\Prism^{[-1]}\{i\}\rightarrow\cdots$$
    with weight $0$ term the $i$th Breuil--Kisin twist. We also have
    $\Oscr_{\Prismbar}=\Oscr_\Prism/\Oscr_\Prism^{[1]}$,
    which takes $(B,I)\in(R/A)_\Prism$ to $\overline{B}=B/I$ and defines a sheaf of commutative
    $R$-algebras.
\end{definition}

Note that the Hodge--Tate tower for $\Oscr_\Prism\{i\}$ is tensored up from the Hodge--Tate tower
for $\Oscr_\Prism$: we have
$\Oscr_\Prism^{[\star]}\{i\}\iso\Oscr_\Prism\{i\}\otimes_{\Oscr_\Prism}\Oscr_\Prism^{[\star]}$.

\begin{lemma}\label{lem:sheaf}
    For each $n,i\in\bZ$, $\Oscr_\Prism^{[n]}\{i\}$ and $\Oscr_{\Prismbar}\{i\}$ are sheaves on $(R/A)_\Prism$ with vanishing
    higher cohomology for any object $(B,I)\in(R/A)_\Prism$.
\end{lemma}

The condition that the higher cohomology groups vanish means that these objects are not only
sheaves of abelian groups but are in fact sheaves when viewed as having values in $\D(\bZ)$, the
derived $\infty$-category of $\bZ$.

\begin{proof}[Proof of Lemma~\ref{lem:sheaf}]
    For $\Oscr_\Prism$ and $\Oscr_\Prismbar$, this follows from~\cite[Cor.~3.12]{prisms}, which says
    that these define sheaves with vanishing higher cohomology in the case of the site of all bounded prisms, i.e.,
    $(\bZ_p/\bZ_p)_\Prism$. We use that the value of
    $\Oscr_\Prism$ and $\Oscr_\Prismbar$ on $(B,I)\in (R/A)_\Prism$ with respect to $(R/A)_\Prism$ can
    be computed by considering $(B,I)$ as an object in $(\bZ_p/\bZ_p)_\Prism$, which follows by
    observing that the \v{C}ech complexes associated to $(B,I)\rightarrow(C,IC)$ in $(R/A)_\Prism$
    and $(\bZ_p/\bZ_p)_\Prism$ agree.
    
    The result for
    $\Iscr_\Prism=\Oscr_\Prism^{[1]}=\ker(\Oscr_\Prism\rightarrow\Oscr_\Prismbar)$ follows because
    sheaves are closed under kernels and because $\Oscr_\Prism\rightarrow\Oscr_\Prismbar$ is surjective
    pointwise. The fact that the Breuil--Kisin twists
    $\Oscr_\Prism\{i\}$ are sheaves with vanishing higher cohomology on $(B,I)\in(R/A)_\Prism$
    follows from the sheaf property for $\Oscr_\Prism$ and the crystal property for the
    Breuil--Kisin twists. Given a $(p,I)$-completely faithfully flat map of prisms
    $(B,I)\rightarrow (C,IC)$, let $(C^\bullet,IC^\bullet)$ be the \v{C}ech nerve. Then,
    $C^\bullet\{i\}$ is equivalent as a cosimplicial object to $(C^\bullet)\otimes_BB\{i\}$ by the
    base change property of the Breuil--Kisin twists. So,
    the sheaf property for $\Oscr_\Prism\{i\}$ reduces to $(p,I)$-completely faithfully flat
    descent for the invertible $B$-module $B\{i\}$. The result for the Hodge--Tate towers is proven in the same way.
\end{proof}

\begin{remark}\label{rem:htbk}
    There is a natural equivalence
    $\Iscr_\Prism/\Iscr_\Prism^2\iso\Oscr_\Prism^{[1]}/\Oscr_\Prism^{[2]}\iso\Oscr_\Prismbar\{1\}$.
    It follows in general that
    $$\Oscr_\Prism^{[n]}\{i\}/\Oscr_\Prism^{[n+1]}\{i\}\iso\Oscr_\Prismbar\{i+n\}.$$
    See~\cite[Rem.~2.5.7]{bhatt-lurie-apc}.
\end{remark}

\begin{definition}[Site-theoretic prismatic cohomology]
    We define $$\Prism_{R/A}^{\s,[n]}\{i\}=\R\Gamma((R/A)_\Prism,\Oscr_\Prism^{[n]}\{i\}).$$
    The superscript $(-)^\s$ will be used to distinguish this form of prismatic cohomology from
    other forms discussed below, before we have established our comparison theorems under
    quasisyntomicity assumptions in Section~\ref{sec:comparisons}.
    
    In the case of $n=i=0$, the resulting object of the $p$-complete derived $\infty$-category
    $\D(\bZ_p)_p^\wedge$ is denoted by
    $\Prism^\s_{R/A}$, the site-theoretic prismatic cohomology of $R$ relative to $A$.
    It naturally has the structure of an $\bE_\infty$-algebra over $A$
    with an endomorphism $\varphi$ called Frobenius induced by the Frobenius endomorphism of the
    sheaf $\Oscr_\Prism$. The objects $\Prism_{R/A}^{\s,[n]}\{i\}$ assemble into
    Hodge--Tate towers
    $$\Prism_{R/A}^{\s,[\star]}\{i\}\colon\cdots\rightarrow\Prism_{R/A}^{\s,[n+1]}\{i\}\rightarrow
    \Prism_{R/A}^{\s,[n]}\{i\}\rightarrow\Prism_{R/A}^{\s,[n-1]}\{i\}\rightarrow\cdots$$
    of $\Prism^\s_{R/A}$-modules; these towers are complete filtrations on the colimit
    $\Prism_{R/A}^{\s}\{i\}[1/I]$ with weight $0$ part $\Prism_{R/A}^{\s}\{i\}$.
\end{definition}

\begin{warning}[No Breuil--Kisin twists for $\delta$-rings]\label{warn:bktwists}
    When the $\delta$-pair $(A,R)$ is prismatic,
    $\Prism_{R/A}^{\s}\{i\}=\R\Gamma((R/A)_\Prism,\Oscr_\Prism\{i\})$ is
    equivalent to $\Prism_{R/A}^\s\otimes_A A\{i\}$, but when $A$ is a general $\delta$-ring, this
    need not be the case. In fact, the construction $A \mapsto A\{i\}$ does not extend from
    prisms to $\delta$-rings since it depends on the ideal $I$ defining the prism structure.
    
    Moreover, even in the special case  $A = R = \bZ_p$  the
    Breuil--Kisin twisted absolute prismatic cohomologies $\Prism_{\bZ_p/\bZ_p}\{i\}$ are, for $i
    \neq 0$, not invertible over $\Prism_{\bZ_p/\bZ_p}$ and are not generally tensor powers of
    $\Prism_{\bZ_p/\bZ_p}\{1\}$. Indeed, since the pairing between $\Prism_{\bZ_p/\bZ_p}\{i\}$ and $\Prism_{\bZ_p/\bZ_p}\{-i\}$ is compatible with the Hodge-Tate filtrations, it suffices to check that e.g. $\Prismbar_{\bZ_p/\bZ_p}\{1\}$ is not invertible. By \cite[Section 4.7]{bhatt-lurie-apc}, the diffracted Hodge complex $\DiffractedHodge_{\bZ_p}$ is just $\bZ_p$, with Sen operator $\Theta=0$, and $\Prismbar_{\bZ_p/\Z_p}\{1\}$ vanishes (since it is identified as the fiber of $-1$ on $\bZ_p$).
\end{warning}

\begin{definition}[Site-theoretic Hodge--Tate cohomology]
    The Breuil--Kisin twisted Hodge--Tate cohomology of $R$ relative to
    $A$ is $\Prismbar^\s_{R/A}\{i\}=\R\Gamma((R/A)_\Prism,\Oscr_{\Prismbar}\{i\})$. If $i=0$, we
    call $\Prismbar^\s_{R/A}$  the Hodge--Tate cohomology of $R$ over $A$, which
    is naturally a $p$-complete $\bE_\infty$-$R$-algebra. For any $i$, $\Prismbar^\s_{R/A}\{i\}$ is naturally a
    $\Prismbar^\s_{R/A}$-module. For each $i,n\in\bZ$, there are fiber sequences
    $$\Prism^{\s,[n+1]}_{R/A}\{i\}\rightarrow\Prism^{\s,[n]}_{R/A}\{i\}\rightarrow\Prismbar^\s_{R/A}\{i+n\}$$
    by Remark~\ref{rem:htbk}.
\end{definition}

\begin{remark}[Insensitivity to $p$-completion]
    Suppose that $A$ and $R$ are bounded so that their derived $p$-completions are discrete and agree with the classical $p$-completions.
    Since all prisms $B$ and their Hodge--Tate quotients $\overline{B}=B/I$ are derived $p$-complete, the
    prismatic site $(R/A)_\Prism$ is naturally equivalent to $(\widehat{R}/\widehat{A})_\Prism$, the
    relative prismatic site of the derived $p$-completion $\widehat{R}$ over the derived
    $p$-completion $\widehat{A}$, where the $\delta$-ring structure on $A$ extends to $\widehat{A}$
    by~\cite[Lem.~2.18]{prisms}. It follows that $\Prism^\s_{R/A}$ depends only on the derived
    $p$-completion, and similarly for the Breuil--Kisin twists and Hodge--Tate towers. The
    boundedness hypothesis can also be replaced at the cost of using the animated prismatic site of
    Variant~\ref{var:animatedsite}. 
\end{remark}

\begin{remark}[Relative prismatic cohomology comparison]
    If $(A,R)$ is a bounded prismatic $\delta$-pair (so that $A$ is a bounded prism and $R$ is a
    bounded commutative $\overline{A}$-algebra),
    then by Remark~\ref{rem:sitecomparison} our site-theoretic prismatic cohomology agrees with the
    cohomology of the prismatic site  of~\cite[Def.~4.1]{prisms} and
    \cite[Def.~4.1.1]{bhatt-lurie-apc}. (Note that the definition in~\cite{bhatt-lurie-apc}
    should include the hypothesis that the prisms $(B,IB)\in(R/A)_\Prism$ be bounded.)
\end{remark}

\begin{remark}[Absolute prismatic cohomology comparison]
    It follows by definition from Remark~\ref{rem:sitecomparison} that $\Prism^{\s,[n]}_{R/\bZ_p}\{i\}$ agrees with the
    cohomology of $\Oscr^{[n]}_\Prism\{i\}$ on the absolute prismatic site of $\Spf R$ studied
    in~\cite[Def.~4.4.27]{bhatt-lurie-apc} and denoted there by
    $\R\Gamma^{\mathrm{site},[n]}_\Prism(\Spf R)\{i\}$.
\end{remark}

\begin{variant}[The animated prismatic site]\label{var:animatedsite}
    Suppose more generally that $A$ is an animated $\delta$-ring in the sense
    of~\cite[App.~A]{bhatt-lurie-prism} and that $R$ is an animated commutative $A$-algebra.
    We let $\Pairs^{\an\delta}$ be the $\infty$-category of such animated $\delta$-pairs.
    Recall from~\cite{bhatt-lurie-prism} that an animated prism is an animated $\delta$-ring
    $B$ equipped with a generalized Cartier divisor $\alpha\colon I\rightarrow B$ such that $B$ is
    $(p,I)$-complete and for any perfect $\bF_p$-algebra $k$ and any animated $\delta$-ring map
    $B\rightarrow W(k)$ the extension of scalars $I\otimes_BW(k)\rightarrow W(k)$ is equivalent to
    the inclusion of the ideal $(p)\subseteq W(k)$. Every ordinary prism is an animated prism.

    The $\infty$-category
    $(R/A)^\an_\Prism$ is defined to be the opposite of the $\infty$-category of
    commutative squares~\eqref{eq:site} where $B\rightarrow\overline{B}=B/I=\cofib(\alpha)$ is an animated prism, $A\rightarrow
    B$ is a map of animated $\delta$-rings and $R\rightarrow\overline{B}$ is a map of animated
    commutative rings. We equip $(R/A)_\Prism^\an$ with the structure of a site by declaring the
    covers to be $(p,I)$-completely
    faithfully flat maps (of animated $\delta$-rings) in $B$. The objects $\Oscr_\Prism^{[n]}\{i\}$
    naturally extend to sheaves on $(R/A)_\Prism^\an$ with values in $\D(\bZ_p)_p^\wedge$.
    We obtain upon taking global sections
    $$\Prism_{R/A}^{\an\s,[n]}\{i\}=\R\Gamma((R/A)^\an_\Prism,\Oscr_\Prism^{\an,[n]}\{i\}),$$
    which are all modules over the animated site-theoretic prismatic cohomology
    $\Prism_{R/A}^{\an\s}$. If $(A,R)$ is a $\delta$-Pair, there is an inclusion of the (underived)
    prismatic site $(R/A)_\Prism$ into the animated prismatic site $(R/A)_\Prism^\an$. Upon taking
    global sections, this inclusion induces a natural map
    $$\Prism_{R/A}^{\an\s}\we\lim_{(B,\overline{B})\in(R/A)^\an_\Prism}B\rightarrow\lim_{(B,\overline{B})\in(R/A)_\Prism}B\we\Prism_{R/A}^\s.$$
    In the prismatic case,
    if $A$ is a bounded prism, $R$ is bounded, and $\L_{R/\overline{A}}$ has $p$-complete Tor-amplitude in $[0,1]$, then the
    natural map $$\Prism_{R/A}^{\an\s}\rightarrow\Prism_{R/A}^\s$$ is an equivalence;
    see~\cite[Rem.~7.14]{bhatt-lurie-prism}. Similar results hold for the Breuil--Kisin twists,
    Hodge--Tate cohomology, and so forth.
\end{variant}

\section{The prismatic crystal}\label{sec:crystal}

In this section we give the general definition of relative prismatic cohomology using the theory of prismatic crystals. 

\begin{notation}
    If $A$ is a commutative ring and $M,N\in\D(A)$, we let $M\tensorhat_A N$ be the derived
    $p$-completion of the derived tensor product over $A$; this makes the $p$-complete derived
    $\infty$-category $\D(A)_p^\wedge$ into a symmetric monoidal stable $\infty$-category.
    In particular, without further decoration, $M\tensorhat N$ denotes the tensor product
    in $\D(\bZ)_p^\wedge\we\D(\bZ_p)_p^\wedge$.
    If $(B,I)$ is a prism and $M,N\in\D(B)$, we write $(M\otimes_BN)_{(p,I)}^\wedge$, or
    $M\tensorhat_BN$ if the prism structure is clear from context, for the derived
    $(p,I)$-completed derived tensor product; again, this endows the $(p,I)$-complete derived
    $\infty$-category $\D(B)_{(p,I)}^\wedge$ with a symmetric monoidal structure. 
\end{notation}

\begin{definition}[Prismatic crystals]\label{def_crystal}
    A prismatic crystal is a functorial
    assignment to each bounded prism $(B,I)$ of a derived $(p,I)$-complete object
    $\Fscr(B)\in\D(B)_{(p,I)}^\wedge$ where
    this assignment satisfies base change in the sense that if $(B,I)\rightarrow (C,IC)$ is a map of
    bounded prisms, then the natural map $\Fscr(B)\tensorhat_BC\rightarrow\Fscr(C)$ is an
    equivalence. More rigorously, a prismatic crystal is a cocartesian section of the
    cocartesian fibration whose classifying functor associates to a prism $(B,I)$ the stable
    $\infty$-category $\D(B)_{(p,I)}^\wedge$ and to a morphism of prisms the derived completed base
    change.
    The stable $\infty$-category of prismatic crystals is equivalent to $\D(\WCart)$, the stable
    $\infty$-category of quasi-coherent sheaves on the formal stack $\WCart$ by~\cite[Prop.~3.3.5]{bhatt-lurie-apc}.
    Moreover, to define a prismatic crystal, it is enough to make such a functorial assignment on
    transversal prisms, by~\cite[Lem.~3.3.10]{bhatt-lurie-apc}.
\end{definition}

We introduce a prismatic crystal $\Hscr_\Prism(R/A)$
associated to any $\delta$-pair where $A$ is bounded.
The prismatic crystal definition of relative prismatic cohomology will be based upon
$\R\Gamma(\WCart,\Hscr_\Prism(R/A))$.

Our construction is based on derived relative prismatic cohomology, which is introduced
in~\cite[Sec.~7.2]{prisms} (see also~\cite[Sec.~4.1]{bhatt-lurie-apc}). If $k$ is a commutative ring,
let $\widehat{\CAlg_k^\an}$ denote the $\infty$-category of derived $p$-complete animated
$k$-algebras. If $(B,I)$ is a bounded prism,
then $\Prism_{-/B}\colon\widehat{\CAlg_{\overline{B}}^\an}\rightarrow\D(B)_{(p,I)}^\wedge$ is defined
to be the unique functor which preserves sifted colimits
and agrees with the site-theoretic relative prismatic cohomology of~\cite{prisms} on $p$-complete
finitely presented polynomial $\overline{B}$-algebras.

\begin{definition}[Relative prismatic cohomology crystals]\label{def:crystal}
    Let $A\rightarrow R$ be a $\delta$-pair and assume that $A$ is bounded; we can in fact allow $R$
    to be a $p$-complete animated commutative $A$-algebra.
    The prismatic crystal $\Hscr_\Prism(R/A)$ associated to the pair $(R/A)$ is the
    assignment, for each transversal prism $(B,I)$, of the derived prismatic
    cohomology $\Hscr_\Prism(R/A)(B)=\Prism_{R\tensorhat\overline{B}/A\tensorhat B}^\rel$,
    where $A\tensorhat B=(A\tensor B)_{(p,I)}^\wedge$.
    As $\Prism_{R\tensorhat\overline{B}/A\tensorhat
    B}\tensorhat_BC\we\Prism_{R\tensorhat\overline{C}/A\tensorhat C}$ for a map of bounded prisms
    $B\rightarrow C$
    by~\cite[Rem.~4.1.5]{bhatt-lurie-apc}, this functor defines a prismatic crystal, i.e., an object of $\D(\WCart)$.
\end{definition}

\begin{remark}\label{rem:unbounded}
    The boundedness hypothesis on $A$ guarantees that if $B$ is a transversal prism, then
    $A\tensorhat B$ is a bounded prism. However, as we do not impose any assumptions on $R$, the ring
    $R\tensorhat\overline{B}$ is in general an animated commutative ring, which is why we require the generality of derived
    relative prismatic cohomology.
    We could drop the boundedness condition on $A$ and describe the value $\Hscr_\Prism(R/A)(B)$ for a non-transversal prism $B$ at the cost of
    using animated prisms as in~\cite{bhatt-lurie-prism}.
    Alternatively, note that every prism $B$ admits a map of prisms
    $C\rightarrow B$ where $C$ is transversal by~\cite[Prop.~2.4.1]{bhatt-lurie-apc}. Thus, the value of $\Hscr_\Prism(R/A)$ at $B$ is
    computed as $\Prism_{R\tensorhat\overline C/A\tensorhat C}\tensorhat_CB$.
\end{remark}

\begin{definition}[Hodge--Tate crystals]
    If $A\rightarrow R$ is a $\delta$-pair where $A$ is bounded, let $\Hscr_\Prismbar(R/A)$
    be the crystal which assigns to a transversal bounded prism $(B,I)$ the derived Hodge--Tate complex
    $$\Hscr_\Prismbar(R/A)(B)=\Prismbar_{R\tensorhat\overline{B}/A\tensorhat B}^\rel.$$
    The Hodge--Tate crystal $\Hscr_\Prismbar(R/A)$ is equivalent to
    $\iota_*\iota^*\Hscr_\Prism(R/A)$ where $\iota\colon\WCart^\HT\hookrightarrow\WCart$ is the
    inclusion of the Hodge--Tate locus.
\end{definition}

\begin{variant}[Breuil--Kisin twisted prismatic crystals]
    Besides $\Hscr_\Prism(R/A)$, there are the
    Breuil--Kisin twists and these admit associated Hodge--Tate towers. In general, let $\Hscr^{[n]}_\Prism(R/A)\{i\}$ be
    the prismatic crystal
    $\Iscr^n\otimes_{\Oscr_{\WCart}}\Hscr_{\Prism}(R/A)\otimes_{\Oscr_{\WCart}}\Oscr_{\WCart}\{i\}$,
    where $\Iscr$ is the ideal of the Hodge--Tate divisor
    $\iota\colon\WCart^\HT\hookrightarrow\WCart$ and $\Oscr_{\WCart}\{i\}$ is the
    $i$th tensor power of the Breuil--Kisin line
    bundle $\Oscr_{\WCart}\{1\}$ of~\cite[Ex.~3.3.8]{bhatt-lurie-apc}. In other words,
    $$\Hscr_\Prism^{[n]}(R/A)\{i\}(B)=\Prism^{\rel,[n]}_{R\tensorhat\overline{B}/A\tensorhat B}\{i\}$$
    for transverse prisms $B$
    in the notation of~\cite[Const.~4.4.10]{bhatt-lurie-apc}.
    There are fiber sequences
    $$\Hscr_\Prism^{[n+1]}(R/A)\{i\}\rightarrow\Hscr_\Prism^{[n]}(R/A)\{i\}\rightarrow\Hscr_\Prismbar(R/A)\{i+n\}$$
    in $\D(\WCart)$
    as in~\cite[Rem.~4.5.7]{bhatt-lurie-apc} thanks to the equivalence
    $\Iscr/\Iscr^2\iso\iota_*\Oscr_{\WCart^{\HT}}\{1\}$. These assemble into Hodge--Tate towers
    $\Hscr_\Prism^{[\star]}(R/A)\{i\}$ of prismatic crystals.
\end{variant}

\begin{definition}[Cohomology of prismatic crystals]
    Suppose that $A\rightarrow R$ is a $\delta$-pair where $A$ is bounded.
    Let  the relative prismatic cohomology be  $\Prism_{R/A}=\R\Gamma(\WCart,\Hscr_\Prism(R/A))$, the cohomology of the prismatic
    crystal $\Hscr_\Prism(R/A)$. Similarly,
    let
    $\Prismbar_{R/A}=\R\Gamma(\WCart^\HT,\iota^*\Hscr_\Prism(R/A))\we\R\Gamma(\WCart,\Hscr_\Prismbar(R/A))$,
    the cohomology of the Hodge--Tate crystal.
    More generally, we have
    $$\Prism^{[n]}_{R/A}\{i\}=\R\Gamma(\WCart,\Hscr^{[n]}_\Prism(R/A)\{i\}),$$
    which assemble into towers $\Prism^{[\star]}_{R/A}\{i\}$ with associated graded pieces
    computed by fiber sequences
    $$\Prism^{[n+1]}_{R/A}\{i\}\rightarrow\Prism^{[n]}_{R/A}\{i\}\rightarrow\Prismbar_{R/A}\{i+n\}.$$
\end{definition}

    Note that a prism $(B,I)$ can be considered as a formal scheme when equipped with the $(p,I)$-adic topology. We shall denote this formal scheme by 
    $\SSpf(B)$.

    \begin{remark}[Cohomology as a limit over prismatic points]\label{rem:cohomology_as_limit}
    Given a transversal prism $(B,I)$ there is a canonical map $\rho_B\colon\SSpf
    B\rightarrow\WCart$. It is shown in~\cite{bhatt-lurie-apc} that $\WCart$ is the colimit of these
    $\rho$ maps over the opposite of the category of transversal prisms. It follows that the
    cohomology of a prismatic crystal such as $\Hscr_\Prism(R/A)$ is given as a limit
    $$\Prism_{R/A}=\R\Gamma(\WCart,\Hscr_\Prism(R/A))=\lim_{\text{transversal
    $B$}}\Hscr_\Prism(R/A)(B)=\lim_{\text{transversal
        $B$}}\Prism_{R\tensorhat\overline{B}/A\tensorhat B}^\rel.$$
    A similar remark applies to the Breuil--Kisin twists and the Hodge--Tate cohomology.
\end{remark}

There is a conjugate filtration of Hodge--Tate cohomology.

\begin{proposition}[The conjugate filtration on Hodge--Tate crystals]\label{prop:conjugate}
    If $A\rightarrow R$ is a $\delta$-pair where $A$ is bounded, then
    the Hodge--Tate crystal $\Hscr_{\Prismbar}(R/A)$, viewed as a quasi-coherent sheaf on
    $\WCart^\HT$, admits an increasing
    exhaustive multiplicative conjugate filtration
    $\F_{\leq\star}^{\delta\conj}\Hscr_{\Prismbar}(R/A)$ with graded pieces given by
    $$\gr^{\delta\conj}_\star\Hscr_{\Prismbar}(R/A)\we\widehat{\L\Omega}^\star_{R/A}\tensor_{\Oscr_{\WCart^\HT}}\Oscr_{\WCart^\HT}\{-\star\}[-\star],$$
    where $\widehat{\L\Omega}^\star_{R/A}$ denotes the $p$-complete derived differential forms of
    $R$ over $A$.
\end{proposition}

\begin{proof}
    Recall that, as for the Cartier--Witt stack, the $\infty$-category $\D(\WCart^\HT)$ of quasi-coherent sheaves on the
    Hodge--Tate locus is given by the limit over $\D(\overline{B})_p^\wedge$ as
    $(B,I)$ ranges over all bounded prisms.
    For a fixed bounded prism $(B,I)$ and a $p$-complete animated $\overline{B}$-algebra $S$, there
    is a natural increasing multiplicative conjugate filtration
    $\F_{\leq\star}^\conj\Prismbar_{S/B}^\rel$ with
    graded pieces given by
    $$\gr^\conj_\star\Prismbar_{S/B}^\rel\we\widehat{\L\Omega}_{S/\overline{B}}^\star\{-\star\}[-\star],$$
    where $\widehat{\L\Omega}_{S/\overline{B}}^\star$ denotes
    the $p$-complete derived differential forms on $S$ relative to $\overline{B}$;
    see~\cite[Thm.~6.3]{prisms} or~\cite[Rem.~4.1.7]{bhatt-lurie-apc}.
    The conjugate filtration is functorial and satisfies base change: if $(B,I)\rightarrow (C,IC)$ is
    a map of bounded prisms, then the natural map
    $\Prismbar^\rel_{S/B}\rightarrow\Prismbar^\rel_{S\tensorhat_{\overline{B}}\overline{C}/C}$
    preserves the conjugate filtration and the induced map
    $$(\F_{\leq\star}^\conj\Prismbar_{S/B})\tensorhat_{\overline{B}}\overline{C}\rightarrow\F_{\leq\star}^\conj\Prismbar_{S\tensorhat_{\overline{B}}\overline{C}/C}$$
    is a filtered equivalence. From this, it follows that the conjugate filtration descends to give
    a filtration $\F^{\delta\conj}_{\leq\star}\Hscr_\Prismbar(R/A)$ on the Hodge--Tate crystal $\Hscr_\Prismbar(R/A)$.
    As $\L_{R\tensorhat\overline{B}/A\tensorhat\overline{B}}\we\L_{R/A}\tensorhat\overline{B}$,
    the graded pieces of the conjugate filtration are the crystals
    $$(B,I)\mapsto\gr_{u}^{\delta\conj}\Hscr_\Prism(R/A)(B)\we\gr^\conj_u\Prismbar_{R\tensorhat\overline{B}/A\tensorhat
    B}\we\widehat{\L\Omega}^u_{R/A}\tensorhat\overline{B}\tensorhat_{\overline{B}}\overline{B}\{-u\}[-u],$$
    where $\overline{B}\{-u\}\we (I/I^2)^{\otimes -u}$. In other words,
    $$\gr_u^{\delta\conj}\Hscr_\Prismbar(R/A)\we\widehat{\L\Omega}^u_{R/A}\tensor\Oscr_{\WCart^\HT}\{-u\}[-u].$$
    The exhaustiveness follows from the
    fact that $\D(\WCart^\HT)\we\lim_B\D(\overline{B})_p^\wedge$, where the limit is over all
    bounded prisms $(B,I)$. In particular,
    $\colim_u\F_{\leq u}^{\delta\conj}\Hscr_{\Prismbar}(R/A)\rightarrow\Hscr_{\Prismbar}(R/A)$
    is an equivalence since it evaluates to an equivalence in each
    $\D(\overline{B})_p^\wedge$.
\end{proof}

\begin{variant}[Conjugate filtration on Breuil--Kisin twists]\label{var:bkconj}
    There is a conjugate filtration on the $i$th Breuil--Kisin twist of the Hodge--Tate crystal
    obtained by tensoring $\F_{\leq\star}^{\delta\conj}\Hscr_\Prismbar(R)$ over $\Oscr_{\WCart^\HT}$ with
    $\Oscr_{\WCart^\HT}\{i\}$. The associated graded pieces are
    $$\gr_{u}^{\delta\conj}\Hscr_\Prismbar(R/A)\{i\}\we\widehat{\L\Omega}^{\star}_{R/A}\tensor\Oscr_{\WCart^\HT}\{i-u\}[-u].$$
\end{variant}

\begin{construction}
    Taking global sections yields conjugate filtrations
    $$\F_{\leq\star}^{\delta\conj}\Prismbar_{R/A}\{i\}$$ for any $i\in\bZ$.
    The associated graded pieces are computed using the Sen operator:
    $$\gr_u^{\delta\conj}\Prismbar_{R/A}\{i\}
    \we\widehat{\L\Omega}^u_{R/A}\tensorhat\R\Gamma(\WCart^\HT,\Oscr_{\WCart^\HT}\{i-u\}[-u])
    \we\widehat{\L\Omega}^u_{R/A}\tensorhat\fib(\bZ_p\xrightarrow{i-u}\bZ_p)[-u]$$
    by~\cite[Cor.~3.5.14]{bhatt-lurie-apc}.
    These filtrations are exhaustive because $\R\Gamma(\WCart^\HT,-)$ commutes with colimits
    by~\cite[Cor.~3.5.13]{bhatt-lurie-apc}.
\end{construction}

\begin{proposition}[Invariance under quasi-étale extensions for prismatic crystals]
  \label{prop:quasietalecrystal}
    If $(A,R)\to (A',R)$ is a map of bounded $\delta$-pairs where $\L_{A'/A}$ vanishes $p$-adically (we call such a map
    $A\rightarrow A'$ $p$-adically quasi-\'etale), then the map
  \[
    \Hscr_{\Prism}^{[n]}(R/A)\{i\}\to \Hscr_{\Prism}^{[n]}(R/A')\{i\}
  \]
  is an equivalence for each $n$ and $i$.
\end{proposition}
\begin{proof}
  This follows directly from the properties of the conjugate filtration, since $\L\Omega^i_{R/A} \we
    \L\Omega^i_{R/A'}$ if $\L_{A'/A}=0$.
\end{proof}

    \begin{corollary}[Invariance under localizations and completions]\leavevmode\label{cor:inv_loc_comp}
    \begin{enumerate}
        \item[{\em (i)}] For any bounded $\delta$-pair $(A,R)$, if $K\subseteq A$ is a finitely
            generated ideal
            contained in the kernel of
            $A\rightarrow R$ and if the derived $(p,K)$-completion of $A$ is discrete, then the natural map
            $\Hscr_\Prism^{[n]}(R/A)\{i\}\rightarrow\Hscr_\Prism^{[n]}(R_p^\wedge/A_{(p,K)}^\wedge)\{i\}$
            is an equivalence for all $i,n\in\bZ$.
        \item[{\em (ii)}] For any bounded $\delta$-pair $(A,R)$, if $S\subseteq A$ is a set of
            elements which map to units in $R$, then the natural map
            $\Hscr_{\Prism}^{[n]}(R/A)\{i\}\to\Hscr_\Prism^{[n]}(R_p^\wedge,A[S^{-1}]_p^\wedge)$ is
            an equivalence for all $i,n\in\bZ$.
    \end{enumerate}
\end{corollary}

\begin{proof}
        For (i), let $A'$ be the derived $(p,K)$-adic completion of $A$.
        The $\delta$-ring structure on $A$ extends uniquely to $A'$ by ~\cite[Lem.~2.18]{prisms}.
        It follows from the $\delta$-conjugate filtration that it is enough to show that
    the natural map
    $\L_{R/A}\rightarrow\L_{R^\wedge_p/A'}$ is a $p$-adic equivalence. This occurs if and only if
    $R\otimes_{A'}\L_{A'/A}$ vanishes $p$-adically. As $R$ is $K$-complete, this happens if and only
    if $\L_{A'/A}$ vanishes $(p,K)$-completely, which holds because $A\rightarrow A'$ is a
    $(p,K)$-adic equivalence, by definition. This proves (i). From (i), we can assume that $A$ and $R$ and derived
    $p$-complete. In particular, it follows that the $\delta$-ring structure on $A$ extends uniquely
    across $A[S^{-1}]$ by~\cite[Rem.~2.16]{prisms}. Now, part (ii) follows from
    Proposition~\ref{prop:quasietalecrystal}.
\end{proof}

Let $R\rightarrow R^0$ be a morphism of $p$-complete animated commutative rings with $p$-completed \v{C}ech complex $R^\bullet$.
We say that $f\colon R\rightarrow R^0$ has $p$-complete descent if $R\rightarrow
R^\bullet$ is a limit diagram.
Say that $f$ has universal $p$-complete descent if for every $R\rightarrow T$ of $p$-complete
animated commutative rings the
induced map $T\rightarrow T\tensorhat R^0$ has $p$-complete descent.
Note that because $\bF_p\otimes_\bZ(-)$ preserves all limits, $R\rightarrow R^0$ has (universal)
$p$-complete descent if and only if $\bF_p\otimes_\bZ R\rightarrow\bF_p\otimes_\bZ R^0$ has
(universal) descent. We can make analogous definitions in the case of $p$-complete animated {\em
graded} commutative rings which are concentrated in nonnegative weights.

The following is a slightly stronger version of \cite[Thm.~3.1]{bms2}, namely descent for differential forms where we also vary the base ring.

\begin{lemma}[Descent for differential forms]\label{lem:descentdifferentialforms}
  Let $(R/A)\to (R^0/A^0)$ be a map of pairs of animated commutative rings where $R\to R^0$ has universal
    $p$-complete descent and let $(R^\bullet,A^\bullet)$ denote the $p$-completed \v{C}ech nerve. Then,
    the natural map
  \[
      \widehat{\L\Omega}^k_{R/A} \to \Tot \widehat{\L\Omega}^k_{R^\bullet/A^\bullet}
  \]
  is an equivalence.
\end{lemma}

\begin{proof}
    We observe that in the diagram of animated graded-commutative rings
  \[
    \begin{tikzcd}
        \widehat{\L\Omega}^*_{R/A} \rar\dar & \widehat{\L\Omega}^*_{R^0/A^0}\dar\\
      R\rar &  R^0
    \end{tikzcd}
  \]
  the left vertical map is a universal $p$-complete graded descent morphism since the fiber is positively graded.
  While the bottom arrow is assumed to be a universal $p$-complete descent morphism, it also follows
that it has descent for $p$-complete modules in the weak sense that if $M$ is a $p$-complete
connective module, then $M\we\Tot M\tensorhat_R R^\bullet$. Indeed, one can apply universality
to the morphism $R\rightarrow R\oplus M$. (This works more generally for bounded below
$p$-complete modules by suspending.) From this, it follows that $R\rightarrow R^0$ is also a
    universal $p$-complete {\em graded} descent morphism. By composition, it follows that the
    composition $\widehat{\L\Omega}^*_{R/A}\rightarrow R^0$ is a universal $p$-complete graded
    descent morphism; indeed, this follows from the variant of~\cite[Lem.~3.3.2(4)]{liuzheng} in
    which both composable morphisms are assumed to have universal descent.
    By~\cite[Lem.~3.3.2(3)]{liuzheng}, it follows that
    $\widehat{\L\Omega}^*_{R/A}\rightarrow\widehat{\L\Omega}^*_{R^0/A^0}$ has universal
    $p$-complete graded descent, from which the lemma follows by taking the $p$-completed graded \v{C}ech
    complex, which is equivalent to $\widehat{\L\Omega}^*_{R^\bullet}$ by the K\"unneth formula for
    derived de Rham cohomology~\cite[Prop.~2.7]{bhatt-padic}, and then taking weight $k$ pieces.
\end{proof}

\begin{proposition}[Descent for prismatic crystals]
  \label{prop:cwdescent}
  Let $(A,R)\to (A^0,R^0)$ be a map of bounded $\delta$-pairs with \v{C}ech nerve $(A^\bullet,R^\bullet)$,
    and assume that $R\to R^0$ is faithfully flat, $A\to A^0$ is flat, and all
    $\L_{R^\bullet/A^\bullet}$ have $p$-complete Tor-amplitude in $[0,1]$. Then,
  \[
    \Hscr^{[n]}_\Prism(R/A)\{i\}\to \Tot \Hscr^{[n]}_\Prism(R^\bullet/A^\bullet)\{i\}
  \]
  is an equivalence for each $n$ and $i$.
\end{proposition}
\begin{proof}
    Since the $\Hscr^{[n]}(R/A)\{i\}$ for varying $n\geq 0$ form a complete filtration on
    $\Hscr_\Prism(R/A)$, it suffices to check the claim for the associated graded $\Hscr_{\Prismbar}(R/A)\{i\}$. We have that
  \[
      \widehat{\L\Omega}^k_{R/A} \to \Tot \widehat{\L\Omega}^k_{R^\bullet/A^\bullet}
  \]
  is an equivalence, by Lemma \ref{lem:descentdifferentialforms}.
  This implies that
  \[
    \F^{\delta\conj}_{\leq k} \Hscr_\Prismbar(R/A)\{i\}\to \Tot \F^{\delta\conj}_{\leq k} \Hscr_\Prismbar(R^\bullet/A^\bullet)\{i\}
  \]
  is an equivalence for each $k$. Since the Tor-amplitude condition ensures that all of the terms on
    the right are coconnective, $\Tot$ commutes with the colimit over $k$, leading to the desired
    statement.
\end{proof}

\begin{lemma}\label{lem:kancotangent}
    For each $u$, the functor $\L\Omega^u_{-/-}\colon\mathrm{Pairs}^\delta\rightarrow\D(\bZ)$
    commutes with sifted colimits.
\end{lemma}

\begin{proof}
    Let $\mathrm{Pairs}$ be the category of arrows in the category of commutative rings.
    Thus, objects are maps $A\rightarrow R$ of commutative rings and morphisms are commutative squares.
    The forgetful functor $\mathrm{Pairs}^\delta\rightarrow\mathrm{Pairs}$ commutes with all
    colimits. Thus, it is enough to prove that
    $\L\Omega^u_{-/-}\colon\mathrm{Pairs}\rightarrow\D(\bZ)$ commutes with sifted colimits.
    To see this let $(A_i,R_i)_{i \in I}$ be a sifted diagram of pairs with colimit $(A,R)$.
    As the diagonal $I \to I^{\Delta^1}$ is cofinal by an application of Joyal's version of
    Quillen's Theorem A
    (see~\cite[Thm.~4.1.3.1]{htt}),
    we have 
    \begin{align*}
    \colim_I \L\Omega^u_{R_i/A_i} & \simeq \colim_{I^{\Delta^1}} \L\Omega^u_{R_i / A_i} \otimes_{A_i} A_j \\
    & \simeq  \colim_I \L\Omega^u_{R_i / A_i} \otimes_{A_i} A \\
    & \simeq \colim_I \L\Omega^u_{R_i \otimes_{A_i} A / A} \\
    & \simeq \L\Omega^u_{\colim  R_i \otimes_{A_i} A / A} \\
    & \simeq \L\Omega^u_{R/A},\qedhere
    \end{align*}
    where the fourth equivalence follows because $\L\Omega_{-/A}$ commutes with sifted colimits.
\end{proof}

\begin{corollary}\leavevmode\label{cor:kanextended}
    \begin{enumerate}
        \item[{\em (a)}] The functors of bounded $\delta$-pairs to $\D(\WCart^{\HT})$ and
            $\D(\bZ_p)_p^\wedge$, respectively,
            given by $$(A,R)\mapsto\Hscr_\Prismbar(R/A)\{i\}
            \quad\text{and}\quad (A,R)\mapsto\Prismbar_{R/A}\{i\}$$ 
            preserve sifted colimits for all $i\in\bZ$ and hence are left Kan extended from their values on finitely
            presented free $\delta$-pairs.
        \item[{\em (b)}] The functors $\Prism^{[\star]}_{-/-}\{i\}$ from bounded $\delta$-pairs to
            $\widehat{\F\D}(\D(\bZ_p)_p^\wedge)$, the $\infty$-category of complete filtered
            objects in $\D(\bZ_p)_p^\wedge$, preserves sifted colimits for all $i\in\bZ$.
    \end{enumerate}
\end{corollary}

\begin{proof}
    The preservation of sifted colimits in part (a) follows from the facts that the conjugate filtration on $\Hscr_\Prismbar(R/A)\{i\}$ is complete
    and exhaustive and that the associated graded pieces are left Kan extended by
    Variant~\ref{var:bkconj} and
    Lemma~\ref{lem:kancotangent}. For the fact that every $\delta$-pair is a sifted colimit of free
    $\delta$-pairs, we note that the free $\delta$-pairs are compact, projective generators of the
    category of $\delta$-pairs (in the $1$-categorical sense), which is easily seen.
    Part (b) follows from part (a) since the Hodge--Tate tower on $\Prism_{R/A}\{i\}$ is complete by construction.
\end{proof}

\begin{remark}[Hodge--Tate filtration]
    When restricted to $\star\geq 0$, we view $\Prism_{R/A}^{[\star]}\{i\}$ as a complete decreasing
    filtration on $\Prism_{R/A}\{i\}$ and refer to this as the Hodge--Tate filtration.
\end{remark}

\begin{warning}
    It is not the case that $\Prism_{-/-}$ is itself left Kan extended. For example, if $A$ is a
    prism, then the prismatic comparison of Proposition~\ref{prop:prismaticcomparison} implies that
    $$\Prism_{\overline{A}[t_1^{1/p^\infty},\ldots,t_s^{1/p^\infty}]_p^\wedge/A}\we
    A[t_1^{1/p^\infty},\ldots,t_s^{1/p^\infty}]_{(p,I)}^\wedge$$
    Taking the colimit as $s\rightarrow\infty$ results in a ring whose $p$-completion is not
    typically $I$-complete and hence is not the prismatic cohomology of
    $\overline{A}[t_1^{1/p^\infty},\ldots,t_s^{1/p^\infty},\ldots]_p^\wedge$ relative to $A$.
\end{warning}

\begin{warning}
    While relative Hodge--Tate cohomology is left Kan extended, it does not typically preserve all
    colimits, although this is the case in the prismatic setting over a fixed prism. For example, absolute Hodge--Tate
    cohomology $\Prismbar_{-/\bZ_p}$ does not preserve pushouts as a functor to $p$-complete
    $\bE_\infty$-rings. 
    Indeed, writing $\bZ_p\langle x_1,\ldots,x_n\rangle$ for the $p$-completed polynomial algebra, we can, following~\cite[Section 4.7]{bhatt-lurie-apc}, express
    $\Prismbar_{\bZ_p\langle x_1,\ldots,x_n\rangle/\bZ_p}$ as fiber of the Sen operator $\Theta$ on $\DiffractedHodge_{\bZ_p\langle x_1,\ldots,x_n\rangle}$, and the $i$th graded of the conjugate filtration on $\Prismbar_{\bZ_p\langle x_1,\ldots,x_n\rangle/\bZ_p}$ as fiber of $-i$ on $L\Omega^i_{\bZ_p\langle x_1,\ldots,x_n\rangle/\bZ_p}[-i]$. Since $-i$ is invertible for $i=1,\ldots,p-1$, this shows that the conjugate filtration on $\Prismbar_{\bZ_p\langle x_1,\ldots,x_n\rangle/\bZ_p}$ degenerates for $n<p$ and we get that the canonical map from the $0$th stage of the conjugate filtration, which can be identified with
    \[
  \bZ_p\langle x_1,\ldots,x_n\rangle\tensorhat\Prismbar_{\bZ_p/\bZ_p} \to \Prismbar_{\bZ_p\langle x_1,\ldots,x_n\rangle/\bZ_p},
    \]
    is an equivalence for $n<p$. However, for $n\geq p$, the $p$-th graded of the conjugate filtration becomes nontrivial as well, and so this map is no longer an equivalence. In particular $\Prismbar_{\bZ_p\langle x_1,\ldots,x_n\rangle/\bZ_p}$ is not the $n$-fold tensor product of $\Prismbar_{\bZ_p\langle x\rangle/\bZ_p}$ for $n\geq p$.
\end{warning}

\begin{example}[Absolute prismatic comparison]\label{exs:wc}
    \begin{enumerate}
        \item[(a)] If $A=\bZ_p$, then
            $\Hscr_\Prism(R/\bZ_p)\we\Hscr_\Prism(R)$, the prismatic crystal
            introduced in~\cite[Const.~4.4.1]{bhatt-lurie-apc}, and the conjugate filtration
            $\F_{\leq\star}^{\delta\conj}\Hscr_\Prismbar(R/\bZ_p)$ agrees with the conjugate filtration on
            the absolute Hodge--Tate crystal of~\cite[Const.~4.5.1]{bhatt-lurie-apc}.
        \item[(b)] More generally, if $A$ is a perfect $\delta$-ring, then
            $\L_{A/\bZ_p}$ vanishes after $p$-completion and hence
            $\widehat{\L\Omega}^\star_{R/A}\we\widehat{\L\Omega}^\star_{R/\bZ_p}$. It follows that
            for each $i\in\bZ$
            the natural map
            $\Hscr_{\Prismbar}(R)\{i\}\rightarrow\Hscr_{\Prismbar}(R/A)\{i\}$
            is an equivalence of quasi-coherent sheaves on $\WCart^\HT$ and thus, by
            $\Iscr$-completeness,
            $\Hscr_\Prism^{[n]}(R)\{i\}\rightarrow\Hscr_\Prism^{[n]}(R/A)\{i\}$ is an equivalence of quasi-coherent
            sheaves on $\WCart$ for all $n,i\in\bZ$. In particular,
            $\Prismbar_R\we\Prismbar_{R/A}$ and
            $\Prism_R\we\Prism_{R/A}$, where $\Prism_R$ denotes the absolute prismatic
            cohomology of~\cite[Const.~4.4.10]{bhatt-lurie-apc}, defined as the global sections of
            the absolute prismatic crystal $\Hscr_\Prism(R)$. This is a generalization to
            $\delta$-rings of
            the relative-to-absolute prismatic cohomology comparison
            theorem~\cite[Prop.~4.4.12]{bhatt-lurie-apc} for
            perfect prisms.
        \item[(c)] If $A=R$ is derived $p$-complete, then the conjugate filtration implies that $\Hscr_{\Prismbar}(A/A)\{i\}\we
            A\tensor\Oscr_{\WCart^\HT}\{i\}$ for all $i\in\bZ$. Hence, since $\R\Gamma(\WCart^\HT,-)$ preserves
            colimits as a functor to $\widehat{\D}(\bZ_p)$ by~\cite[Cor.~3.5.13]{bhatt-lurie-apc},
            $$\Prismbar_{A/A}\we A\tensorhat\Prismbar_{\bZ_p}\we
            A\oplus A[-1].$$ By $\Iscr$-completeness, the natural map
            $A\tensorhat\Oscr_\WCart\rightarrow\Hscr_\Prism(A/A)$ is an equivalence. Thus, the
            natural map $A\tensorhat\Prism_{\bZ_p}\rightarrow\Prism_{A/A}$ is an equivalence after
            completing the tensor product with respect to the Hodge--Tate filtration on
            $\Prism_{\bZ_p}$.
    \end{enumerate}
\end{example}

\begin{warning}[Two conjugate filtrations]\label{warn:conjugate}
    As indicated by the notation, the conjugate filtration differs in general from
    that considered in~\cite{prisms} in the prismatic case.
    Let $\F_{\leq\star}^{\delta\conj}\Prismbar_{R/A}$ denote the filtration defined by using
    that $A$ is a $\delta$-ring, and let
    $\F_{\leq\star}^{\Prism\conj}\Prismbar_{R/A}^\rel$ denote the prismatic conjugate
    filtration on derived prismatic cohomology, as studied in~\cite{prisms} when $A$ is a prism and
    $R$ is a $p$-complete animated commutative $\overline{A}$-algebra.
    As proven in Proposition~\ref{prop:prismaticcomparison},
    $\Prismbar_{R/A}^\rel\we\Prismbar_{R/A}$, however the conjugate filtrations do not agree, for example when $A=\bZ_p$ and $R=\bF_p$.
    In this case, the
    $\delta$-ring variant agrees with the absolute conjugate filtration, which
    has
    $$\gr_u^{\delta\conj}\Prismbar_{\bF_p/\bZ_p}\we\fib(\L\Omega^u_{\bF_p/\bZ_p}\xrightarrow{-u}\L\Omega^u_{\bF_p/\bZ_p})[-u]\we\fib(\bF_p\xrightarrow{-u}\bF_p),$$
    while
    $$\gr_u^{\Prism\conj}\Prismbar^\rel_{\bF_p/\bZ_p}\we\L\Omega^u_{\bF_p/\bF_p}\{-u\}[-u]\simeq\begin{cases}
        \bF_p&\text{if $u=0$, and}\\
        0&\text{otherwise.}
    \end{cases}
    $$
    The prismatic conjugate spectral sequence degenerates at the $\E_1$-page,
    while the $\delta$-ring-theoretic conjugate spectral sequence supports
    $\d_p$-differentials and degenerates at the $\E_{p+1}$-page.
\end{warning}

We end this section by giving a different perspective on the definition of $\Prism_{R/A}$ that avoids the theory of prismatic crystals. To this end we let 
$\Pairs^{\Prism}$ be the category whose objects are given by pairs consisting of a bounded prism
$(A,I)$ together with a map of commutative rings $A/I \to R$. Note that this is not a full
subcategory of $\Pairs^{\delta}$ since we require the choice and existence of a prismatic ideal $I$,
in contrast to the definition of prismatic $\delta$-pairs (Definition \ref{prismatic}) where we only
required the mere existence of maps. Still there is a forgetful functor
\[
\Pairs^{\Prism} \to \Pairs^{\delta} \ .
\]
Moreover Bhatt--Scholze's derived prismatic cohomology~\cite[Sec.~7.2]{prisms} defines a functor 
\[
\Prism_{R/A}^\rel \colon \Pairs^{\Prism} \to \D(\bZ)^\wedge_p \ .
\]
We will later show that indeed their relative theory agrees in this case with ours, see Proposition \ref{prop:prismaticcomparison}. 
\begin{proposition}
    Prismatic cohomology $\Prism_{R/A}\colon \Pairs^{\delta} \to \D(\bZ)^\wedge_p $ is the right Kan
    extension of $\Prism_{R/A}^\rel$ along $\Pairs^{\Prism} \to \Pairs^{\delta}$.
\end{proposition}

Note that since $\Pairs^{\Prism} \to \Pairs^{\delta}$ is not fully faithful, it is not a priori
clear that the right Kan extension indeed leaves the value on a prismatic $\delta$-pairs unchanged.
However, we will show that it does in Section~\ref{comparisontoBS}.

\begin{proof}
Let $(A,R) \in \Pairs^{\delta}$ be a $\delta$-pair. We would like to compute the value of the right Kan extension on $(A,R)$. This is given by the limit
of $\Prism_{S/C}^\rel$ where $C$ and $S$ vary over the category  $\Pairs^{\Prism}_{(A,R)_/}$, i.e.
    $C$ has a prism structure $K \subseteq C$ a map of rings $C/K \to S$ and is equipped with a map
    of $\delta$-pairs $(A,R) \to (C,S)$. We now consider the functor
\[
\mathrm{TransPrisms} \to \Pairs^{\Prism}_{(A,R)_/}
\]
from the category of transversal prisms given by 
\[
(B,J) \mapsto (A \tensorhat B, R \tensorhat \overline{B}),
\]
as in the definition of the prismatic crystal (Definition \ref{def:crystal}). Now to prove the
    proposition it suffices to show that this functor 
    $\mathrm{TransPrisms} \to \Pairs^{\Prism}_{(A,R)_/}$ is initial (in the sense that taking limits
    induces an equivalence, i.e., the opposite of the notion of cofinality from~\cite[Sec.~4.1.1]{htt}), as the limit over the composite of this functor with $\Prism_{R/A}^\rel$
    is the definition of $\Prism_{R/A}$ (see Remark~\ref{rem:cohomology_as_limit}).
    By Quillen's Theorem A (see~\cite[Thm.~4.1.3.1]{htt}), this initiality comes down to checking that for each object 
\[
(C,S) \in \Pairs^{\Prism}_{(A,R)_/}
\]
the category of transversal prisms $(B,J)$ with a map of prism pairs
\[
(A \tensorhat B, R \tensorhat \overline{B}) \to (C,S)
\]
under $(A,R)$ is weakly contractible. Such a map however is uniquely determined on the part of the
    tensor product given by $A$ and $R$ so that it is equivalently given by a map of prismatic
    $\delta$-pairs
\[
(B, \overline{B}) \to (C,S)
\]
or equivalently a map of prisms $(B,J) \to (C,K)$ since then the map $\overline{B} \to S$ is uniquely
    determined. Thus, the category is equivalent to the category of transversal prisms over $(C,S)$ which is weakly contractible since it is sifted \cite[Cor.~2.4.7]{bhatt-lurie-apc}.
\end{proof}

We note that the proof actually shows something stronger, namely that one can even right Kan extend from transversal prisms instead of all prisms, but we shall not need this here.

\section{Prismatization}\label{sec:prismatization}

The goal of prismatization as developed by Drinfeld~\cite{drinfeld-prismatization} and
Bhatt--Lurie~\cite{bhatt-lurie-apc, bhatt-lurie-prism}
is to construct for each $p$-adic formal scheme $X$ a formal stack (in groupoids on the site of
$p$-nilpotent affine schemes with the flat topology), called $\WCart_X$, such that (1)
quasi-coherent sheaves on $X$ provide a good formalism for coefficients for prismatic cohomology and
(2) $\R\Gamma(\WCart_X,\Oscr_{\WCart_X})$ computes $\Prism_X$, the absolute (derived) prismatic
cohomology of $X$. There is also a relative version, denoted by $\WCart_{X/A}$ when $A$ is a prism,
introduced in~\cite{bhatt-lurie-prism}. For more on formal (higher) stacks, see Appendix~\ref{app}.

\begin{definition}[Cartier--Witt divisors]
    Recall from~\cite[Def.~3.1.4]{bhatt-lurie-apc} that if $S$ is a $p$-nilpotent commutative ring,
    then a Cartier--Witt divisor of $S$ is a generalized Cartier divisor $\alpha\colon I\rightarrow W(S)$
    such that the image of $I\xrightarrow{\alpha} W(S)\rightarrow S$ is nilpotent and the image of
    $I\xrightarrow{\alpha} W(S)\xrightarrow{\delta}W(S)$ generates the unit ideal of $W(S)$.
\end{definition}

\begin{definition}[Absolute Cartier--Witt stacks]
    Let $R$ be a commutative ring with bounded $p$-power torsion. Define $\WCart_R$ to be the formal stack
    which assigns to a $p$-nilpotent ring $S$ the groupoid of pairs $(I\xrightarrow{\alpha}
    W(S),R\rightarrow\overline{W(S)})$, where $\alpha$ is a Cartier--Witt divisor of $S$ and
    $\overline{W(S)} = W(S)/\!/I$ is the animated commutative ring with underlying chain complex given by the cofiber
    of $\alpha$; see~\cite[Sec.~3]{bhatt-lurie-prism}.
\end{definition}

Morally one should think of an $S$-point of $\WCart_R$ as making $W(S)$ an object of the animated
version of the absolute prismatic site of $R$ using \cite[Remark 3.1.5]{bhatt-lurie-prism}. The
stack $\WCart_{R/A}$ that we define below similarly has $S$ points given by ways of making $W(S)$ an
object of the animated version of the relative prismatic site of $R$ relative to $A$.

\begin{remark}[Insensitivity to derived $p$-completion]
    As $\pi_0\overline{W(S)}$ is a $p$-nilpotent ring by~\cite[Lem.~3.3]{bhatt-lurie-prism},
    $\WCart_R$ depends only on the derived $p$-completion of $R$.
\end{remark}

\begin{remark}[Extension to stacks]
    We will need the natural extension of $\WCart_{-}$ to formal stacks. If
    $\Fscr$ is any presheaf of spaces on $p$-nilpotent affine schemes, we can define a presheaf $\WCart_\Fscr$ as follows: for any
    $p$-nilpotent ring $S$, we let
    $\WCart_\Fscr(S)$ be the space of pairs of a
    Cartier--Witt divisor $\alpha\colon I\rightarrow W(S)$ and a $\overline{W(S)}$-point of $\Fscr$.
    With this notation, $\WCart_R=\WCart_{\Spf R}$ if $R$ is bounded.
    The natural map $\colim_{\Spf R\rightarrow\Fscr}\WCart_R\rightarrow\WCart_\Fscr$ is an
    equivalence of presheaves of spaces, where the colimit ranges over maps from $p$-nilpotent affine schemes.
\end{remark}

\begin{construction}[{\cite[Const.~3.11]{bhatt-lurie-prism}}]\label{const:probe}
    If $A$ is a bounded $\delta$-ring, there is a functor $\Spf
    A\times\WCart\rightarrow\WCart_A$ obtained as follows.
    Given a $p$-nilpotent commutative ring $S$ and an $S$-point of $\Spf A\times\WCart$, which consists of a pair
    consisting of a
    map $A\rightarrow S$ and a Cartier--Witt divisor
    $I\xrightarrow{\alpha}W(S)$ over $S$, one uses the universal property of
    the Witt vectors to obtain a $\delta$-ring map $A\rightarrow W(S)$. The composition
    $A\rightarrow W(S)\rightarrow W(S)/\!/I$ together with $(I,\alpha)$ defines an
    $S$-point  of $\WCart_A$.
\end{construction}

\begin{definition}[Prismatization of $\delta$-pairs]\label{def:prismatizationdelta}
    Suppose that $A\rightarrow R$ is a bounded $\delta$-pair. We define the relative Cartier--Witt stack of $R$ over $A$ as
    the pullback
    $$\xymatrix{
        \WCart_{R/A}\ar[r]\ar[d]&\WCart_R\ar[d]\\
        \Spf A\times\WCart\ar[r]&\WCart_A.
    }$$
    In other words, if $S$ is a $p$-nilpotent commutative ring,
    $\WCart_{R/A}(S)$ is the space consisting of quadruples of a Cartier--Witt
    divisor
    $I\xrightarrow{\alpha}W(S)$, a map $R\rightarrow W(S)/\!/I$ of animated commutative
    rings, a
    map $A\rightarrow S$ of commutative rings, and an equivalence between the
    compositions $A\rightarrow W(S)\rightarrow W(S)/\!/I$ and $A\rightarrow R\rightarrow W(S)/\!/I$.
    More generally, one can make an analogous definition when $X$ is a bounded $p$-adic formal
    scheme over $\Spf A$ to obtain $\WCart_{X/A}$.
\end{definition}

This definition agrees with the definition given in~\cite{bhatt-lurie-prism}
when both make sense by Proposition~\ref{prop:comparison} below. We give some notation and the definition from {\em op. cit.}

\begin{notation}
    If $(A,I)$ is a bounded prism, we distinguish between the $p$-adic formal scheme $\Spf A$ and the
    object $\SSpf A$. If $S$ is a $p$-nilpotent commutative ring, then $(\Spf
    A)(S)=\colim_m\Map_{\CAlg}(A/p^m,S)$, while $(\SSpf
    A)(S)=\colim_{m,n}\Map_{\CAlg}(A/(p^m,I^n),S)$. Alternatively, $\SSpf A\subseteq\Spf A$ is the
    formal completion at $I$, corresponding to maps $A\rightarrow S$ where $I$ is sent to a
    nilpotent ideal.
\end{notation}

\begin{definition}[Relative Cartier--Witt stacks~\cite{bhatt-lurie-prism}]\label{def:relativeprismatization}
    If $A$ is a bounded prism, then there is a natural map $\SSpf
    A\rightarrow\WCart_{\overline{A}}$. If $S$ is a $p$-nilpotent commutative ring and $g\colon A\rightarrow
    S$ is a point of $\SSpf A$, so that the image of $I$ in $S$ is nilpotent, then we can take the
    adjoint $\delta$-ring map $g^{\#}\colon A\rightarrow W(S)$ and take the Cartier--Witt divisor
    $I\otimes_A W(S)\rightarrow W(S)$ obtained by extension of scalars. If $R$ is a commutative
    $\overline{A}$-algebra with bounded $p$-power torsion, then Bhatt and Lurie define the relative
    Cartier--Witt stack $\WCart_{R/A}$ in~\cite[Variant~5.1]{bhatt-lurie-prism} as the pullback
    $$\xymatrix{
        \WCart_{R/A}\ar[r]\ar[d]&\WCart_R\ar[d]\\
        \SSpf A\ar[r]&\WCart_{\overline{A}}.
    }$$
\end{definition}

\begin{proposition}[Agreement with prismatization in the prismatic case]\label{prop:comparison}
    If $A$ is a bounded prism and $R$ is a bounded
    $\overline{A}$-algebra, then there is a natural commutative diagram
    $$\xymatrix{
        \WCart_{R/A}\ar[r]\ar[d]&\WCart_R\ar[d]\\
        \SSpf A\ar[r]^{\rho_A}\ar[d]_{\id\times\rho_A}&\WCart_{\overline{A}}\ar[d]\\
        \Spf A\times\WCart\ar[r]&\WCart_A
    }$$
    of pullback squares where the outer square agrees with the defining square of
    Definition~\ref{def:prismatizationdelta}. 
\end{proposition}

In particular, the definition given in Definition~\ref{def:prismatizationdelta} of
$\WCart_{R/A}$ agrees with that of Definition~\ref{def:relativeprismatization} for prismatic $\delta$-pairs.

\begin{proof}
    The pullback of the top square defines the relative Cartier--Witt stack in
    the prismatic case by definition~\cite[Var.~5.1,
    Const.~7.1]{bhatt-lurie-prism}.
    It is enough to identify the bottom pullback as $\SSpf A$; write $P$ for
    this pullback. Evaluated at a $p$-nilpotent
    commutative ring $S$, $P(S)$ is the space consisting of
    \begin{enumerate}
        \item[(i)] a Cartier--Witt divisor $I\xrightarrow{\alpha}W(S)$ equipped with a map
            $f\colon\overline{A}\rightarrow W(S)/\!/I$,
        \item[(ii)] a map $g\colon A\rightarrow S$ and a
            Cartier--Witt divisor $J\xrightarrow{\beta}W(S)$, and
        \item[(iii)] an
            equivalence $(I,\alpha)\we (J,\beta)$ of Cartier--Witt divisors on $S$
            and an equivalence of maps between $A\xrightarrow{g^\#} W(S)\rightarrow
            W(S)/\!/J$ and $A\rightarrow\overline{A}\xrightarrow{f} W(S)/\!/I$, where $g^\#$
            is the map of $\delta$-rings adjoint to $g$.
    \end{enumerate}
    This data determines and is determined by a map $A\rightarrow S$ which
    sends $I$ to a nilpotent ideal of $S$ by rigidity for maps between
    prisms~\cite[Lem.~3.59]{prisms}. So, the pullback $P$ is
    equivalent to $\SSpf A$, as desired.
\end{proof}

\begin{example}\label{ex:wcartofprism}
    If $B$ is a bounded prism, then the natural map $$\WCart_{\overline{B}/B}\rightarrow\SSpf B$$ is
    an equivalence.
\end{example}

\begin{lemma}[Insensitivity to $I$-adic completion]\label{lem:Iinsensitive}
    Suppose that $A\rightarrow R$ is a bounded $p$-complete $\delta$-pair and that the kernel contains an ideal
    $I$ which is locally generated by a distinguished element which is a non-zero divisor. Let $A_I^\wedge$ denote the
    $I$-adic completion of $A$, which is a prism. Then, the natural map
    $$\WCart_{R/A_I^\wedge}\rightarrow\WCart_{R/A}$$ is an equivalence.
\end{lemma}

\begin{proof}
    If $S$ is a $p$-nilpotent commutative ring and $x$ is an $S$-point of
    $\WCart_{R/A}$, corresponding to the data of a map $A\xrightarrow{g} S$,
    a Cartier--Witt divisor $J\xrightarrow{\alpha}W(S)$, a map $R\rightarrow W(S)/J$, and a
    specified equivalence between $A\xrightarrow{g^\#}W(S)\rightarrow W(S)/J$ and
    $A\rightarrow R\rightarrow W(S)/J$,
    then it follows from commutativity that $g^\#$ sends the ideal $I\subseteq A$ into $J$.
    On the other hand, $J$ maps to a nilpotent ideal of $S$ by definition of Cartier--Witt divisors.
    It follows that $g$ and $g^{\#}$ factor uniquely through the $I$-adic completion of $A$ and
    hence that the map on Cartier--Witt stacks is an equivalence.
\end{proof}

\begin{remark}
    In fact, in Lemma~\ref{lem:Iinsensitive}, we do not need $I$ to be locally generated by a distinguished
    element for $A_I^\wedge$ to be a prism. We need only
    for $I$ to be some ideal such that the derived $I$-adic completion of $A$ is discrete.
    However, we will only apply the lemma in the prismatic situation.
\end{remark}

\begin{variant}[Derived (animated) Cartier--Witt stacks]\label{var:anwc}
    Suppose that $R$ is an animated commutative ring.
    The derived Cartier--Witt stack $\WCart_R^\an$ of $R$ is defined to be the functor on
    $p$-nilpotent animated commutative rings $S$ with value the space 
    consisting of pairs of an animated prism structure $I\xrightarrow{\alpha}W(S)$ such that the
    image of $\pi_0(I)\rightarrow\pi_0(W(S))\rightarrow\pi_0(S)$ is nilpotent together with a map
    $R\rightarrow\overline{W(S)}=\cofib(\alpha)$.

    Let $A\rightarrow R$ be an animated $\delta$-pair. There is a natural extension of
    Construction~\ref{const:probe} to a map $\Spf A\times\WCart^\an\rightarrow\WCart^\an_A$ by using
    the universal property~\cite[Prop.~A.23]{bhatt-lurie-prism} of the Witt vectors for animated
    commutative rings as the cofree animated $\delta$-ring functor. Define $\WCart_{R/A}^\an$ as the
    pullback (in presheaves of spaces on derived schemes) of $\Spf
    A\times\WCart^\an\rightarrow\WCart^\an_A$ along
    $\WCart^\an_R\rightarrow\WCart^\an_A$.

    If $A$ is a bounded prism and $R$ is an animated commutative $\overline{A}$-algebra, then the
    proof of Proposition~\ref{prop:comparison} can be extended to the animated situation to give an
    equivalence between $\WCart^\an_{R/A}$, as defined above, and the definition given
    in~\cite[Const.~7.1]{bhatt-lurie-prism}.
\end{variant}

\begin{lemma}\label{lem:presheaf_extensions}
  Assume $\Gscr\to \Fscr$ is a map of presheaves on a site with a subcanonical topology and that the following hold:
  \begin{enumerate}
      \item[{\em (1)}] $\Fscr$ is a sheaf;
      \item[{\em (2)}] for every representable $\underline{U}$ and any map $\underline{U}\to \Fscr$, the pullback $\underline{U}\times_{\Fscr} \Gscr$ is a sheaf.
  \end{enumerate}
  Then, $\Gscr$ is a sheaf.
\end{lemma}

\begin{proof}
  We prove the following more abstract claim in a general $\infty$-category with pullbacks. Assume
    that $f\colon A\to B$
    and $X\to Y$ are morphisms such that
  \begin{enumerate}
      \item[(1)] $\Map(B,B)\to \Map(A,B)$ is an equivalence,
      \item[(2)] $\Map(B,Y)\to \Map(A,Y)$ is an equivalence, and
      \item[(3)] $\Map(B, B\times_{Y} X) \to \Map(A, B\times_Y X)$ is an equivalence.
  \end{enumerate}
  Then, $\Map(B,X)\to \Map(A,X)$ is an equivalence. The claim will follow from this by letting
    $X=\Gscr$, $Y=\Fscr$, and letting $A\to B$ range over covering sieves $S\to \underline{U}$
    (with the first claim coming from the fact that $\underline{U}$ is a sheaf since the topology
    was assumed subcanonical).

  To check the claim, observe that we may view $\Map(B,X)\to \Map(A,X)$ as a morphism of spaces
    over $\Map(B,Y)\simeq \Map(A,Y)$, and check that it is an equivalence on each fiber. The map of
    fibers over $g\in \Map(B,Y)$ coincides with the map of fibers of the map $\Map(B, B\times_Y X)
    \to \Map(A,B\times_Y X)$ over $\id_B$ (resp. $f$) in $\Map(B,B)\simeq \Map(A,B)$. By assumption,
    the map $\Map(B, B\times_Y X)\to \Map(A,B\times_Y X)$ is an equivalence, finishing the proof.
\end{proof}

\begin{lemma}[The Cartier--Witt stacks are sheaves]
    If $A\rightarrow R$ is a bounded $\delta$-pair, then $\WCart_{R/A}$ is a sheaf for the flat
    topology on $p$-nilpotent commutative rings.
\end{lemma}

\begin{proof}
    By the pullback definition of $\WCart_{R/A}$ it suffices to prove that $\WCart_R$ is a flat
    sheaf for any bounded $R$.

    First, $\WCart_{\bZ_p}$ is a flat sheaf. We begin by proving that $S\mapsto\bPic(W(S))$ is a
    flat sheaf, where $\bPic(A)$ denotes the groupoid of invertible $A$-modules.
    Since $p$ is nilpotent in $S$, it follows that each $W_n(S)\rightarrow S$ is a nilpotent
    extension. Therefore, in the limit, $W(S)\rightarrow S$ is a henselian pair. It follows by~\cite[\href{https://stacks.math.columbia.edu/tag/0D4A}{Tag 0D4A}]{stacks} that
    every finitely generated projective $S$-module lifts, uniquely up to isomorphism, to a finitely generated projective
    $W(S)$-module. In particular, invertible $S$-modules $L$ lift, and they lift to $S$-modules
    $\widetilde{L}$ with the property that $\widetilde{L}^\vee\otimes_{W(S)} \widetilde{L}$ is a
    lift of $S$ and hence is isomorphic to $W(S)$ by uniqueness. Thus, $\pi_0\bPic(W(S))\iso\pi_0\bPic(S)$.



    It follows that the map of presheaves of grouplike $\bE_\infty$-spaces $\bPic(W(-))\rightarrow\bPic(-)$ 
    has fiber given by $S\mapsto\B (1+VW(S))^\times$, where $(1+VW(S))^\times$ is the kernel of
    $W(S)^\times\rightarrow S^\times$.
    This kernel has a filtration by the kernels of
    $W_n(S)^\times\rightarrow S^\times$. As $S$ is $p$-nilpotent, $W_n(S)\rightarrow S$ has a
    nilpotent kernel and it follows that the kernel of $W_n(S)^\times\rightarrow S^\times$ can be
    filtered with associated graded pieces given by $S$-modules. These associated pieces can be
    chosen to be functorial in $S$ (for $S$ where a fixed power of $p$ vanishes) and they satisfy flat
    descent, so the associated graded pieces are quasi-coherent sheaves. Therefore, the kernel of $W_n(S)^\times\rightarrow S^\times$ is the value at $S$ of a
    flat sheaf with vanishing higher cohomology. Taking the limit and using that the transition
    maps are surjections one finds that the same is true for the kernel of $W(S)^\times\rightarrow
    S^\times$.
    Now, we have written $\bPic(W(-))$ as an extension of $\bPic(-)$ by a sheaf of abelian groups
    with vanishing higher cohomology. The claim that $S\mapsto\bPic(W(S))$ is a flat sheaf now
    follows from Lemma~\ref{lem:presheaf_extensions}.

    Recall that $(\bA^1/\Gm)(S)$ is the groupoid of generalized Cartier divisors, which consists of
    pairs of an invertible $S$-module $I$ and a map $I\rightarrow S$ of $S$-modules. Using that
    $S\mapsto\bPic(W(S))$ is a flat sheaf and that maps between line bundles are controlled by
    coherent sheaves it follows that $S\mapsto(\bA^1/\Gm)(W(S))$ is a flat sheaf.

    Now, $\WCart_{\bZ_p}$ is a full subpresheaf of $(\bA^1/\Gm)(W(-))$. It therefore defines a separated
    presheaf of anima and it suffices to see that if $\alpha\colon I\rightarrow W(S)$ is a
    generalized Cartier divisor with the property that there is a faithfully flat map $S\rightarrow
    S^0$ such that the induced generalized Cartier divisor $I\otimes_{W(S)}W(S^0)\rightarrow
    W(S_0)$ is a Cartier--Witt divisor, then $I\rightarrow W(S)$ is a Cartier--Witt divisor. This
    can be verified using the criterion of~\cite[Prop.~5.1.2]{bhatt-fgauges} which says that a generalized Cartier divisor $I\rightarrow W(S)$ is
    Cartier--Witt if and only if for every map $S\rightarrow k$ where $k$ is a perfect field the
    induced Cartier--Witt divisor $I\otimes_{W(S)}W(k)\rightarrow W(k)$ is $pW(k)\subseteq W(k)$.

    Second, the fibers of $\WCart_R\rightarrow\WCart_{\bZ_p}$ are flat sheaves over points of
    $\WCart_{\bZ_p}(S)$ for $p$-nilpotent $S$. This follows from
    the fact that once a Cartier--Witt divisor $I\rightarrow W(S)$ is fixed, then if $S\rightarrow S^0$
    is a faithfully flat map with \v{C}ech complex $S^\bullet$ we have
    $\overline{W(S)}\we\Tot(\overline{W(S^\bullet)})$. Thus,
    $\Map(R,\overline{W(S)})\we\Tot\Map(R,\overline{W(S^\bullet)})$ where the mapping anima are
    computed in the $\infty$-category of animated commutative rings.

    The result follows from another application of Lemma~\ref{lem:presheaf_extensions}.
\end{proof}

\begin{definition}[Prismatic cohomology via Cartier--Witt stacks]\label{def:wc}
    For a bounded $\delta$-pair $(A,R)$, let
    $$\Prism_{R/A}^\wc=\R\Gamma(\WCart_{R/A},\Oscr_{\WCart_{R/A}}).$$
    There are Breuil--Kisin twists $\Oscr_{\WCart_{R/A}}\{i\}$ on $\WCart_{R/A}$ obtained by
    pullback from $\WCart$; there is also a quasi-coherent sheaf of ideals $\Iscr$ obtained by
    pullback from the Hodge--Tate ideal on $\WCart$.
    Thus, we can also define
    $\Prism_{R/A}^{\wc,[n]}\{i\}=\R\Gamma(\WCart_{R/A},\Iscr^{\otimes
    n}\otimes_{\Oscr_{\WCart_{R/A}}}\Oscr_{\WCart_{R/A}}\{i\})$ and the Hodge--Tate tower
    $\Prism_{R/A}^{\wc,[\star]}\{i\}$.
    The associated graded pieces of the Hodge--Tate tower fit into fiber sequences
    $$\Prism_{R/A}^{\wc,[n+1]}\{i\}\rightarrow\Prism_{R/A}^{\wc,[n]}\{i\}\rightarrow\Prismbar_{R/A}^{\wc}\{i+n\}.$$
\end{definition}

\begin{definition}[Relative Hodge--Tate stacks]
    Pulling back $\WCart^{\HT}\hookrightarrow\WCart$ along $\WCart_{R/A}\rightarrow\WCart$ produces
    a relative Hodge--Tate stack $\WCart_{R/A}^{\HT}$. By construction,
    $$\R\Gamma(\WCart_{R/A}^\HT,\Oscr_{\WCart_{R/A}^\HT}\{i\})\we\Prismbar_{R/A}^{\wc}\{i\}=\cofib\left(\Prism_{R/A}^{\wc,[1]}\{i\}\rightarrow\Prism_{R/A}^{\wc,[0]}\{i\}\right).$$
\end{definition}

\begin{example}
    In the situation of Proposition~\ref{prop:comparison}, when $A$ is a bounded prism and
    $R=\overline{A}$, one has a natural equivalence $\WCart_{\overline{A}/A}^\HT\we\Spf\overline{A}$.
\end{example}

\begin{variant}\label{var:awc}
    If $A\rightarrow R$ is an animated $\delta$-pair, we can define
    $\Prism_{R/A}^{\an\wc}=\R\Gamma(\WCart_{R/A}^\an,\Oscr_{\WCart_{R/A}^\an})$ as well as
    Breuil--Kisin twists and so on.
    If $R/A$ is a bounded $\delta$-pair, then the classical locus of $\WCart_{R/A}^\an$ is
    $\WCart_{R/A}$ by definition, so there is a natural map
    $\Prism_{R/A}^{\an\wc}\rightarrow\Prism_{R/A}^\wc$.
\end{variant}

Before comparing different definitions of prismatic cohomology, we develop some of the theory of the
prismatization of $\delta$-pairs.

\begin{lemma}[Compatibility with limits]\label{lem:square}
    Suppose that $\{X_i\rightarrow\Spf A_i\}_{i\in I}$ is a natural transformation of $I$-indexed diagrams of bounded $p$-adic formal
    schemes. Assume that the $A_i$ are are equipped with $\delta$-ring structures and the maps in
    the diagram are $\delta$-ring maps. If the diagrams $\{X_i\}_{i\in I}$ and $\{\Spf A_i\}_{i\in
    I}$ are Tor-independent, meaning that their limit in $p$-adic formal schemes agrees with their
    limit in derived $p$-adic formal schemes, then letting $X=\lim X_i$ and $\Spf A=\lim\Spf A_i$
    the natural map $$\WCart_{X/A}\rightarrow\lim_I\WCart_{X_i/A_i}$$ is an equivalence. 
\end{lemma}

\begin{proof}
    Indeed, each of the other vertices in the definition of $\WCart_{X_*/A_*}$ in
    Definition~\ref{def:prismatizationdelta} is compatible with limits
    by~\cite[Rem.~3.5]{bhatt-lurie-prism}.
\end{proof}

The following lemma is similar to, but easier than, \cite[Lem.~6.3]{bhatt-lurie-prism}.

\begin{lemma}\label{lem:surjective}
    Suppose that $A\rightarrow A^0$ is a map of
    bounded $\delta$-rings and $R$ is a bounded commutative $A^0$-algebra.
    If $A\rightarrow A^0$ is $p$-completely faithfully flat, then the natural
    map
    $$\WCart_{R/A^0}\rightarrow\WCart_{R/A}$$
    is surjective in the flat topology.
\end{lemma}

\begin{proof}
    Let $R^0=A^0\tensorhat_AR$. Then, because $R$ is already a commutative $A^0$-algebra, there are maps of
    $\delta$-pairs $(A,R)\rightarrow (A^0,R)\rightarrow(A^0,R^0)$. To prove that
    $\WCart_{R/A^0}\rightarrow\WCart_{R/A}$ is surjective, it is thus enough to prove that
    $\WCart_{R^0/A^0}\rightarrow\WCart_{R/A}$ is surjective. However, this morphism is the top arrow
    of a pullback diagram
    $$\xymatrix{
        \WCart_{R^0/A^0}\ar[r]\ar[d]&\WCart_{R/A}\ar[d]\\
        \Spf A^0\times\WCart\ar[r]&\Spf A\times\WCart.
    }$$ The bottom arrow is surjective by hypothesis, and hence the top one is too.
\end{proof}

\begin{corollary}[Prismatization preserves quasisyntomic covers]\label{cor:surjective}
    Suppose that $(A,R)\rightarrow (A^0,R^0)$ is a $p$-completely faithfully flat map of bounded
    $\delta$-pairs, meaning that $A\rightarrow A^0$ and $R\rightarrow R^0$ are $p$-completely
    faithfully flat. If $\L_{R^0/R}$ has $p$-complete Tor-amplitude in $[0,1]$, then the natural forgetful map
    $$\WCart_{R^0/A^0}\rightarrow\WCart_{R/A}$$
    is surjective in the flat topology.
\end{corollary}

\begin{proof}
    The map in question factors as
    $\WCart_{R^0/A^0}\rightarrow\WCart_{R^0/A}\rightarrow\WCart_{R/A}$.
    The first map is surjective by Lemma~\ref{lem:surjective}.
    The second map is the left vertical map in the cartesian diagram
    $$\xymatrix{
        \WCart_{R^0/A}\ar[r]\ar[d]&\WCart_{R^0}\ar[d]\\
        \WCart_{R/A}\ar[r]&\WCart_R.
    }$$
    The right vertical map is surjective by~\cite[Lem.~6.3]{bhatt-lurie-prism}, which applies by our
    assumption on $\L_{R^0/R}$. Thus, the left vertical map is surjective.
\end{proof}

\section{Comparison theorems}\label{sec:comparisons}

If $A\rightarrow R$ is a bounded $\delta$-pair, then there is a natural commutative diagram of maps
between the
various forms of prismatic cohomology of $R$ relative to $A$ constructed in the previous sections:
$$\xymatrix{
    \Prism_R&&\Prism_{R/A}\ar[d]^{\text{Cor.~\ref{cor:pcwc}}}\ar[ll]^{\text{Ex.~\ref{exs:wc}}}_{A=\bZ_p}\ar[rr]_{\text{Prop.~\ref{prop:prismaticcomparison}}}^{\text{prismatic}}&&\Prism_{R/A}^{\rel}\\
    \Prism^{\an\wc}_{R/A}\ar[rr]^{\text{Lem.~\ref{lem:classicality}}}\ar[d]&&\Prism^{\wc}_{R/A}\ar[d]^{\text{Thm.~\ref{thm:equivalence}}}&&\\
    \Prism^{\an\s}_{R/A}\ar[rr]^{\text{Var.~\ref{var:animatedsiteomparison}}}&&\Prism^{\s}_{R/A}.&&
}$$
The lower horizontal arrows are constructed using the inclusion of discrete commutative rings into
animated commutative rings; see Variants~\ref{var:animatedsite} and~\ref{var:awc}. The vertical
arrows as well as the upper horizontal arrow are constructed below. The top right arrow exists in
the case of a bounded prismatic $\delta$-pair and the top left arrow exists, and is an equivalence,
by Example~\ref{exs:wc}.

We prove that these maps, as well as the corresponding maps on Breuil--Kisin twists and Hodge--Tate
towers, are equivalences for a large class of $\delta$-pairs satisfying
quasisyntomicity conditions.

\subsection{Stack and crystal}

We use the base change results of the appendix to identify the pushforward of the structure sheaf
along $\WCart_{R/A}\rightarrow\WCart$.

\begin{proposition}[Crystal comparison]\label{prop:push}
    Given a bounded $\delta$-pair $(A,R)$ where $\L_{R/A}$ has $p$-complete Tor-amplitude in
    $[0,1]$, the pushforward of the Breuil--Kisin twisted structure sheaf
    $\Iscr_{\WCart_{R/A}}^{\otimes n}\otimes\Oscr_{\WCart_{R/A}}\{i\}$ along
    the composition $\WCart_{R/A}\rightarrow\Spf
    A\times\WCart\xrightarrow{\mathrm{pr}_2}\WCart$ is naturally equivalent to
    $\Hscr_\Prism^{[n]}(R/A)\{i\}$ for each $n,i\in\bZ$.
\end{proposition}

\begin{proof}
    First, we handle the case where $n=i=0$. Consider the pullback diagram
    $$\xymatrix{
        \WCart_{R\tensorhat\overline{B}/A\tensorhat B}\ar[r]\ar[d]^{q'}&\WCart_{R/A}\ar[d]^q\\
        \SSpf B\ar[r]^{\rho_B}&\WCart,
    }$$
    where $(B,I)$ is a transversal prism and where the upper-left vertex is identified using Tor-independence
    of $(B,\overline{B})\leftarrow(\bZ_p,\bZ_p)\rightarrow(A,R)$ and Lemma~\ref{lem:square} to
    obtain identifications 
    $$\SSpf
    B\times_{\WCart}\WCart_{R/A}\we\WCart_{\overline{B}/B}\times_{\WCart_{\bZ_p/\bZ_p}}\WCart_{R/A}\we\WCart_{R\tensorhat\overline{B}/A\tensorhat
    B}.$$
    Note that the $p$-completed tensor product $A\tensorhat B$ is possibly not a prism so that
    $\WCart_{R\tensorhat\overline{B}/A\tensorhat B}$ is not by definition the relative
    Cartier--Witt stack studied in~\cite{bhatt-lurie-prism}.
    However, the natural map
    $\WCart_{R\tensorhat\overline{B}/(A\tensorhat
    B)_I^\wedge}\rightarrow\WCart_{R\tensorhat\overline{B}/A\tensorhat B}$ is an
    equivalence by Lemma~\ref{lem:Iinsensitive}. It
    follows that $$q_*\Oscr_{\WCart_{R/A}}(B)=\rho_B^*q_*\Oscr_{\WCart_{R/A}}\we
    q'_*\Oscr_{\WCart_{R\tensorhat\overline{B}/(A\tensorhat
    B)_I^\wedge}}\we\R\Gamma(\WCart_{R\tensorhat\overline{B}/(A\tensorhat
    B)_I^\wedge},\Oscr_{\WCart_{R\tensorhat\overline{B}/(A\tensorhat B)_I^\wedge}})$$ by base
    change for bounded above quasi-coherent cohomology; see Corollary~\ref{cor:1}.
    The latter is equivalent to $\Prism_{R\tensorhat\overline{B}/A\tensorhat B}^\rel$ under our
    $p$-complete Tor-amplitude hypothesis by~\cite[Thm.~6.4, Rem.~7.23]{bhatt-lurie-prism}.
    This identification is natural in $B$, so the lemma follows for $i=0$.

    For general $n,i$, we use the fact that $\Oscr_{\WCart_{R/A}}\{i\}\we q^*\Oscr_{\WCart}\{i\}$
    and $\Iscr_{\WCart_{R/A}}^{\otimes n}\we q^*\Iscr_{\WCart}^{\otimes n}$, so
    the projection formula (see Proposition~\ref{prop:basicsquare}) implies that $q_*(\Iscr^{\otimes n}\otimes\Oscr_{\WCart_{R/A}}\{i\})\we
    \Iscr^{\otimes n}\otimes q_*(\Oscr_{\WCart_{R/A}})\{i\}\we\Hscr^{[n]}_\Prism(R/A)\{i\}$.
\end{proof}

\begin{remark}
    As $\Hscr_\Prism(R/A)$ is naturally an $A$-module in $\D(\WCart)$, it can be considered as a
    quasi-coherent sheaf on
    $\Spf A\times\WCart$. The proposition shows in fact that it agrees with the
    pushforward of $\Oscr_{\WCart_{R/A}}$ along the canonical map $\WCart_{R/A}\rightarrow\Spf A\times\WCart$.
\end{remark}

\begin{variant}
    If $(A,R)$ is a $\delta$-pair where $A$ is bounded and $R$ is an animated commutative
    $A$-algebra such that $\Omega^1_{\pi_0(R/p)/A/p}$ is finitely generated over $\pi_0(R)/p$
    and if $(B,I)$ is a transversal prism, then Bhatt and Lurie show in~\cite[Thm.~7.20]{bhatt-lurie-prism}
    that $\Prism_{R\tensorhat\overline{B}/(A\tensorhat
    B)_I^\wedge}\we\R\Gamma(\WCart^\an_{R\tensorhat\overline{B}/(A\tensorhat
    B)_I^\wedge},\Oscr_{\WCart^\an_{R\tensorhat B/(A\tensorhat B)_I^\wedge}})$.
    Thus, it follows from the argument in Proposition~\ref{prop:push}, and an appeal to
    Variant~\ref{var:anwc}, that
    $q_*\Oscr_{\WCart^\an_{R/A}}\we\Hscr_\Prism(R/A)$.
\end{variant}

\begin{corollary}\label{cor:htpush}
    Given a bounded $\delta$-pair $(A,R)$ where $\L_{R/A}$ has $p$-complete Tor-amplitude in
    $[0,1]$, the pushforward
    $q_*\Oscr_{\WCart_{R/A}^\HT}\{i\}$ along $q\colon \WCart_{R/A}^\HT \to \WCart^\HT$ is naturally
    equivalent to $\iota^*\Hscr_\Prism(R/A)\{i\}$ for each $i$, where
    $\iota\colon\WCart^\HT\hookrightarrow\WCart$.
\end{corollary}

\begin{proof}
    We have to show that the canonical map 
    \[
    \iota^*\Hscr_\Prism(R/A)\{i\} \to q_*\Oscr_{\WCart_{R/A}^\HT}\{i\}
    \]
    is an equivalence. This can be verified after applying $\iota_*$ in which case both sides can be identified with the cofiber of the map  
    $\Hscr^{[1]}_\Prism(R/A)\{i\} \to \Hscr_\Prism(R/A)\{i\}$ by Proposition \ref{prop:push}.
\end{proof}

\begin{corollary}\label{cor:pcwc}
    If $(A,R)$ is a bounded $\delta$-pair where $\L_{R/A}$ has $p$-complete Tor-amplitude in $[0,1]$, then
    there are natural equivalences $\Prism_{R/A}^{[n]}\{i\}\we\Prism_{R/A}^{\wc,[n]}\{i\}$ and
    $\Prismbar_{R/A}\{i\}\we\Prismbar_{R/A}^{\wc}\{i\}$ for
    each $i,n\in\bZ$.
\end{corollary}

\begin{proof}
    Take global sections in Proposition~\ref{prop:push} and Corollary~\ref{cor:htpush}.
\end{proof}

\subsection{Stack and site}\label{sec:stacksite}

If $(A,R)$ is a bounded $\delta$-pair and $(B,I)\in(R/A)_\Prism$, then by functoriality of the
relative Cartier--Witt stacks there is a canonical morphism
$\WCart_{\overline{B}/B}\rightarrow\WCart_{R/A}$, which induces a map
$\Prism_{R/A}^\wc\rightarrow\Prism_{\overline{B}/B}^\wc\we B$. These assemble into a natural
morphism $$\Prism_{R/A}^\wc\rightarrow\Prism_{R/A}^\s$$ by definition of site-theoretic relative prismatic
cohomology. This natural map extends to a natural transformation on Hodge--Tate towers and
Breuil--Kisin twists.

\begin{theorem}[Stack-site comparison]
\label{thm:equivalence}
    Let $(A,R)$ be a bounded $\delta$-pair.
    If $\L_{R/A}$ has $p$-complete Tor-amplitude in $[0,1]$, then the natural map
    \[
        \Prism_{R/A}^{[n],\wc}\{i\}\rightarrow\Prism_{R/A}^{[n],\s}\{i\}
    \]
    is an equivalence for all $n,i\in\bZ$.
\end{theorem}

\begin{remark}
    Note that for a prism $A$ and an $A/I$-algebra $R$, the condition that
    $\L_{R/A}$ has $p$-complete Tor-amplitude in $[0,1]$ is slightly weaker than
    that $\L_{R/(A/I)}$ having $p$-complete Tor-amplitude in $[0,1]$. So
    Theorem~\ref{thm:equivalence} is slightly more general than the
    comparison result obtained in \cite{bhatt-lurie-prism}.
\end{remark}

We will need a few preliminary results.
Our strategy is to show that both sides satisfy a restricted form of quasisyntomic descent, which
lets us reduce to the case of pairs $(A,R)$ where $A$ is a prism and $R$ is a commutative
$\overline{A}$-algebra, i.e., the situation studied in~\cite{prisms}. In that case, the
comparison theorem will follow from~\cite{bhatt-lurie-prism}. We need the following lemma, which is a variant of a statement due to Rezk~\cite[Proposition 3.8]{rezk-lambda} and can also be deduced from Rezk's statement, but we present a slightly more direct proof here.

\begin{lemma}[Rezk]\label{lem:rezk}
    Given $\delta$-rings $A, A', B$ with $B$ $p$-complete, and maps $A\to A' \to B$ such that
    \begin{enumerate}
        \item[{\em (1)}] the map $A\to A'$ is $p$-completely formally \'etale and a $\delta$-ring map and
        \item[{\em (2)}] the composite $A\to B$ is a $\delta$-ring map.
    \end{enumerate}
    Then the map $A'\to B$ is also a $\delta$-ring map.
\end{lemma}
\begin{proof}
Consider the diagram
\[
\begin{tikzcd}
A\rar\dar & A' \rar\dar & B\dar\\
W_2(A) \rar & W_2(A') \rar\dar & W_2(B)\dar\\
& A'\rar & B
\end{tikzcd}
\]
where the upper vertical maps are the maps encoding the $\delta$-structure, and
    the lower vertical maps are the canonical projections. The lower square
    commutes trivially, the upper left square commutes by the assumption
    that $A\to A'$ is a $\delta$-ring map, the upper outer square commutes by
    the assumption that $A\to B$ is a $\delta$-ring map, and the right outer
    square commutes trivially. To show that $A'\to B$ is a $\delta$-ring map,
    we need to show that the upper right square commutes, i.e. that the two
    composites $A'\to W_2(B)$ agree. By assumption, they fit into a commutative square
\[
\begin{tikzcd}
 A  \rar\dar  & W_2(B)\dar\\
 A'\ar[shift left,ur]\ar[shift right,ur]\rar & B,
\end{tikzcd}
\]
We will show now that $W_2(B)$ is a limit of iterated square zero-extensions and thus
by the uniqueness part of lifting formally \'etale extensions along square-zero maps, the two diagonal maps agree which finishes the proof.

Let $I \subseteq W_2(B)$ be the kernel of $W_2(B) \to B$. Then we have that $I^2 \subseteq pI$ since every element of $I$ is  of the form $V(b)$ for some element $b \in B$ where $V: B \to W_2(B)$ is the Verschiebung. Moreover we have $V(b) \cdot V(b') = p V(bb')$ by the usual properties of $V$. We now consider the sequence of ideals 
\[
 \subseteq p^2I \subseteq pI \subseteq I \subseteq W_2(B)
 \]
 which leads to a sequence of surjective morphisms
 \[
\ldots \to W_2(B) / p^2 I \to W_2(B) / p I \to W_2(B) / I = B
 \]
 with kernels $p^nI / p^{n-1} I$ which are all square zero. Moreover we have that the inverse limit is given by $W_2(B)$ provided that the $p$-adic topology on $B$ is separated, which is true by assumption.
\end{proof}

\begin{lemma}[Insensitivity to formally \'etale extensions I]\label{lem:etaleinsensitive}
    \label{lem:fetale1}
    Given a sequence $A\to A' \to R$, where $A\to A'$ is a map of $\delta$-rings
    which is $p$-completely formally \'etale, the natural map
    $\WCart_{R/A'}\rightarrow\WCart_{R/A}$ is an equivalence. In
    particular, the natural maps
    $$\Prism_{R/A}^{[n],\wc}\{i\}\rightarrow\Prism_{R/A'}^{[n],\wc}\{i\}$$
    are equivalences for all $n,i\in\bZ$.
\end{lemma}

\begin{proof}
    Consider the diagram
    \[
    \begin{tikzcd}
        \WCart_{R/A'} \rar\dar & \WCart_R \dar\\
        \Spf A' \times \WCart \rar\dar & \WCart_{A'}\dar\\
        \Spf A \times \WCart \rar & \WCart_{A}.
    \end{tikzcd}
    \]
    The upper diagram is a pullback square by definition. If we show that the lower
    square is a pullback diagram, it follows that the outer diagram is a
    pullback, and thus $\WCart_{R/A}\simeq \WCart_{R/A'}$.

    Let $P$ denote the pullback of the lower square; there is a canonical map $\Spf
    A'\times\WCart\rightarrow P$.
    An $S$-point in $P$ is given by a map $A\to S$, a
    generalized Cartier-Witt divisor $I\to W(S)$, and a factorization of the
    composite map $A\to W(S)/I$ through $A'$.

    Since $A\to A'$ is $p$-completely formally \'etale, this factorization lifts
    uniquely to a map $A'\to W(S)$ agreeing with the $\delta$-ring map $A\to
    W(S)$ on $A$. This map is automatically a $\delta$-ring map by Lemma
    \ref{lem:rezk}. Since $W(S)$ is the cofree $\delta$-ring on $S$, this is
    the same datum as a map $A'\to S$ extending the original map $A\to S$. So a
    point in the pullback is the same datum as a pair of map $A'\to S$
    and generalized Cartier-Witt divisor $I\to W(S)$, which is to say a point of
    $\Spf A'\times \WCart$.
\end{proof}

\begin{remark}
    The consequence for prismatic cohomology in Lemma~\ref{lem:etaleinsensitive} can be proved using
    the conjugate filtration on Hodge--Tate
    cohomology (Proposition~\ref{prop:conjugate}) and Corollary~\ref{cor:pcwc} under the stronger
    hypothesis that $\L_{A'/A}$ vanishes $p$-adically.
\end{remark}

\begin{lemma}[Insensitivity to formally \'etale extensions II]
    \label{lem:fetale2}
    Given $A\to A' \to R$ with $A\to A'$ a $p$-completely formally \'etale map of
    $\delta$-rings, the sites $(R/A)_\prism$ and $(R/A')_\prism$ agree. In
    particular,
    \[
        \prism_{R/A}^{[n],\s}\{i\}\simeq \prism_{R/A'}^{[n],\s}
    \]
    for all $n,i\in\bZ$.
\end{lemma}
\begin{proof}
Given an object $B$ of $(R/A)_\prism$, we get a unique dashed lift making the diagram
\[
\begin{tikzcd}
A \dar\rar\ar[rr,bend left] & A'\dar\rar[dashed] & B\dar\\
R \rar{\id} & R\rar & B/J
\end{tikzcd}
\]
of commutative rings commute. The map $A'\to B$ is automatically a $\delta$-ring map by Lemma
    \ref{lem:rezk}, so this defines an equivalence $(R/A)_\prism \simeq
    (R/A')_\prism$.
\end{proof}

\begin{lemma}[Comparison for prismatic $\delta$-pairs]
    \label{lem:prismatic_comparison}
    Given a $\delta$-ring $A$ and an $A$-algebra $R$ with bounded $p$-power torsion, if the kernel of the map
    $A\to R$ contains a Cartier divisor $I$ which makes the pair $(A,I)$ into a bounded prism
    and if $\L_{R/\overline{A}}$ has $p$-complete Tor-amplitude in $[0,1]$, then the natural map
    $$\Prism_{R/A}^{[n],\wc}\{i\} \to \Prism_{R/A}^{[n],\s}\{i\}$$
    is an equivalence for all $n,i\in\bZ$.
\end{lemma}

\begin{proof}
    By Proposition~\ref{prop:comparison}, our $\WCart_{R/A}$ agrees in this case with the relative
    prismatization as studied in~\cite[Sec.~5]{bhatt-lurie-prism}. It follows
    from~\cite[Thm.~6.4]{bhatt-lurie-prism} (and~\cite[Rem.~7.23]{bhatt-lurie-prism}) that the cohomology of $\WCart_{R/A}$ with coefficients
    in $\Iscr^{\otimes n}_{\WCart_{R/A}}\{i\}$ computes the cohomology of the corresponding sheaf on
    the relative prismatic site.
\end{proof}

\begin{corollary}\label{cor:extension}
    Fix a $\delta$-pair $A\rightarrow R$.
    If $A\rightarrow R$ factors as $A\rightarrow A'\rightarrow R$ where $A\rightarrow A'$ is a
    $p$-completely formally \'etale map of $\delta$-rings, the kernel of $A'\rightarrow R$
    contains a Cartier divisor $I$ making $A'$ into a prism, and $\L_{R/\overline{A'}}$ has
    $p$-complete Tor-amplitude in $[0,1]$,
    then the canonical map
    \[
        \Prism_{R/A}^{[n],\wc}\{i\}\rightarrow\Prism_{R/A}^{[n],\s}\{i\}
    \]
    is an equivalence for all $n,i\in\bZ$.
\end{corollary}

\begin{proof}
    Combine Lemmas~\ref{lem:fetale1}, \ref{lem:fetale2}, and~\ref{lem:prismatic_comparison}.
\end{proof}

\begin{lemma}\label{lem:specialdescent}
    As a functor of commutative $A$-algebras, $\prism^{[n],\s}_{-/A}\{i\}$ has descent with respect
    to quasisyntomic covers of the form
    \[
    R \to R[z^{1/p^\infty}]_p^\wedge/(z-p)
    \]
    for all $n,i\in\bZ$.
\end{lemma}

\begin{proof}
   For fixed $A$ and $R$ we have a canonical `forgetful' functor 
   $(R/A)_\prism \to (\bZ_p/\bZ_p)_\prism$ that induces  an adjunction
    \[
         f_!: \Shv((R/A)_\prism) \xymatrix{\ar[r]<2pt> & \ar[l]<2pt>}  \Shv((\bZ_p/\bZ_p)_\prism): f^*
    \] 
    between $\infty$-topoi
   whose right adjoint $f^*$ is the restriction functor. The left adjoint $f_!$ sends the terminal object to the object
   \[
       \underline{R/A}: (B, B/J) \mapsto \Map_{\delta-\text{pair}} ((A,R), (B,B/J)) \ .
   \]
   Moreover the right adjoint $f^*$ preserves the sheaves $\Oscr_\Prism^{[n]}\{i\}$.
   As a result we can write 
   $\Prism_{R/A}^{[n],\s}\{i\}$,
   which is defined as the maps out of the terminal object in  $\Shv((R/A)_\prism)$ to $\Oscr_\Prism^{[n]}\{i\}$, as maps in $\Shv((\bZ_p/\bZ_p)_\prism)$ from 
   $\underline{R/A}$ to $\Oscr_\Prism^{[n]}\{i\}$.
   
  For a general morphism $R \to R'$ of rings under $A$  this procedure translates the \v{C}ech complex 
  \[
  \xymatrix{
  R \ar[r] & R' \ar[r]<2pt> \ar[r]<-2pt> & R' \otimes_R R' \ar[r]<4pt> \ar[r]<-4pt> \ar[r] & R' \otimes_R R' \otimes_R R'  \ar[r]<6pt> \ar[r]<-6pt>  \ar[r]<2pt> \ar[r]<-2pt> & \cdots
  } 
  \]
  into the \v{C}ech complex of the map $\underline{R'/A} \to \underline{R/A}$ (considered as an augmented simplicial object) in the $\infty$-topos  $\Shv((\bZ_p/\bZ_p)_\prism)$, as one immediately verifies.
  Thus to verify descent for a map $R \to R'$, that is the fact that the diagram
    \[
  \xymatrix{
\Prism_{R/A}^{[n],\s}\{i\} \ar[r] & \Prism_{R'/A}^{[n],\s}\{i\} \ar[r]<2pt> \ar[r]<-2pt> &\Prism_{R' \otimes_R R'/A}^{[n],\s}\{i\}\ar[r]<4pt> \ar[r]<-4pt> \ar[r] & \Prism_{R' \otimes_R R' \otimes_R R'/A}^{[n],\s}\{i\} \ar[r]<6pt> \ar[r]<-6pt>  \ar[r]<2pt> \ar[r]<-2pt> & \cdots
  } 
  \]
  is a limit diagram, it then suffices to show that the morphism $\underline{R'/A} \to \underline{R/A}$ is an effective epimorphism in the topos $\Shv((\bZ_p/\bZ_p)_\prism)$.

    Concretely, we need to show that for every object $(B,J)$
    in $(R/A)_\prism$, there exists a $(B',J')\in (R'/A)_\prism$ whose restriction
    covers $(B,J)$, i.e. for each $B$ we find a $B'$ in the following diagram
    \[
    \begin{tikzcd}
    A\dar\rar & B\dar\rar & B'\dar\\
    R \dar\rar & B/J \rar & B'/J\\
    R' \ar[rru],
    \end{tikzcd}
    \]
    such that $B\to B'$ is faithfully flat.

  Now we let  
 $R' = R[z^{1/p^\infty}]_p^\wedge/(z-p)$
and simply take $B'$ to be the prismatic envelope
    \[
        B[z^{1/p^\infty}]\left\{\frac{z-p}{J}\right\}_{(p,J)}^\wedge,
    \]
    with $\delta(z) = 0$. This is $p$-completely faithfully flat over $B$
    by~\cite[Prop.~3.13]{prisms}.
\end{proof}

\begin{proof}[Proof of Theorem \ref{thm:equivalence}]
    Both sides of the comparison map have descent with respect to an extension of
    the form $R\to R[z^{1/p^\infty}]_p^\wedge/(z-p)$. For the right-hand side, this is
    Lemma~\ref{lem:specialdescent}. For the left hand side, it follows from
    Proposition~\ref{prop:cwdescent} and Corollary~\ref{cor:pcwc}.

    The $n$th term in the \v{C}ech-nerve of the map $R\to R^0 = R[z^{1/p^\infty}]_p^\wedge/(z-p)$ is given by
    \[
    R^n =  R^{n-1}[z^{1/p^\infty}]_p^\wedge/(z-p) 
    \]
    and $L_{R^{n-1} / A}$ has $p$-complete Tor-amplitude in $[0,1]$. Therefore,
    it suffices to see that the comparison map is an equivalence for
    $A\rightarrow R[z^{1/p^\infty}]_p^\wedge/(z-p)=:R'$. This map factors through $A\rightarrow A'$, where
    $A' = A[z^{1/p^\infty}]_{(p,z-p)}^\wedge$ (with $\delta(z)=0$). As the kernel of $A'\rightarrow
    R[z^{1/p^\infty}]_p^\wedge/(z-p)$ contains a distinguished non-zero divisor, $z-p$, making $A'$
    into a prism, Corollary~\ref{cor:extension} applies
    to give the desired comparison equivalence since
    \[
    \L_{R' / (A'/(z-p))} \simeq \L_{R/A} \otimes_A (A' / (z-p))
    \]
    has $p$-complete Tor-amplitude in $[0,1]$ since $(A' / (z-p))$ is $p$-completely free as an $A$-module.
\end{proof}

\begin{variant}\label{var:animatedsiteomparison}
    Using a similar argument as in the proof of Theorem~\ref{thm:equivalence}, one can prove that the animated
    site-theoretic prismatic cohomology $\Prism_{R/A}^{\an\s}$ satisfies sufficiently fine descent in $R$
    in order to reduce to the prismatic case, where it follows
    from~\cite[Rem.~7.14]{bhatt-lurie-prism} that it agrees with $\Prism_{R/A}^\s$. Thus, one finds
    that if $(A,R)$ is a bounded $\delta$-pair and $\L_{R/A}$ has $p$-complete Tor-amplitude in
    $[0,1]$, then the natural maps
    $\Prism_{R/A}^{[n],\an\s}\{i\}\rightarrow\Prism_{R/A}^{[n],\s}\{i\}$ are equivalences for all
    $n,i\in\bZ$.
\end{variant}

To conclude the section, we give a generalization of the classicality
result~\cite[Cor.~8.13]{bhatt-lurie-prism} to our relative Cartier--Witt stacks.
For the notation $\QSynscr_A$ and $\QRSPerfdscr_A$, see the beginning of Section~\ref{sec:descent}.

\begin{lemma}[Classicality of relative Cartier--Witt stacks]\label{lem:classicality}
    If $A$ is a quasisyntomic $\delta$-ring
    and $R\in\QSynscr_A$, then $\WCart_{R/A}^\an$ is classical
    and $$\Prism_{R/A}^{[n],\an\wc}\{i\}\rightarrow\Prism_{R/A}^{[n],\wc}\{i\}$$ is an equivalence
    for each $n,i\in\bZ$.
\end{lemma}

\begin{proof}
    It suffices to prove that $\WCart_{R/A}^\an$ is a colimit (in derived stacks) of objects represented
    by discrete commutative rings. By quasisyntomic descent, Corollary~\ref{cor:surjective}, and
    Lemma~\ref{lem:square}, we can assume that $R\in\QRSPerfdscr_A$ and in particular that $\L_{R/A}$
    has $p$-complete Tor-amplitude in $[1,1]$.
    Consider the map of $\delta$-pairs $(A,R)\rightarrow
    (A[z^{1/p^\infty}]_p^\wedge,R[z^{1/p^\infty}]_p^\wedge/(z-p))=:(A^0,R^0)$ and let $(A^\bullet,R^\bullet)$ be the
    associated $p$-completed \v{C}ech complex in $\delta$-pairs. By Corollary~\ref{cor:surjective} (or rather its derived
    variant), the induced map
    $\WCart^\an_{R^0/A^0}\rightarrow\WCart^\an_{R/A}$ is a surjective map of stacks in the flat topology on
    $p$-nilpotent commutative rings. By the derived version of Lemma~\ref{lem:square}, the
    \v{C}ech complex of this map is $\WCart_{R^\bullet/A^\bullet}^\an$. However, for each $s\geq 0$, the $\delta$-pair
    $(A^s,R^s)$ is in fact prismatic and the prismatic cohomology $\Prism_{R^s/A^s}$ is discrete
    since $\L_{R^s/A^s}$ has $p$-complete Tor-amplitude in $[1,1]$, so that the Cartier--Witt stack
    $\WCart_{R^s/A^s}^\an$
    is affine and classical by~\cite[Cor.~7.18]{bhatt-lurie-prism}. The lemma now follows by taking
    the geometric realization and global sections.
\end{proof}

\subsection{Comparison to~\cite{prisms}}\label{comparisontoBS}

Let $(A,I)$ be a bounded prism and let $R$ be an animated commutative $\overline{A}$-algebra. There
are now two definitions of $\Prism_{R/A}$, namely as the global sections of the prismatic crystal
$\Hscr_\Prism(R/A)$ and as the derived relative prismatic cohomology of~\cite[Sec.~7.2]{prisms}.
For the remainder of this section, we write $\Prism_{R/A}^\rel$ for the latter.
By combining Propositions~\ref{prop:push} and~\ref{prop:comparison}
with~\cite[Rem.~7.23]{bhatt-lurie-prism}, we find that $\Prism_{R/A}\we\Prism_{R/A}^\rel$ when $R$
is discrete and
$\L_{R/\overline{A}}$ has $p$-complete Tor-amplitude in $[0,1]$. The next proposition shows that in
fact no condition on $R$ is required.

\begin{proposition}\label{prop:prismaticcomparison}
    If $(A,I)$ is a bounded prism and $R$ is an animated commutative $\overline{A}$-algebra,
    then there is a natural equivalence
    $\gamma_{R/A}\colon\Prism_{R/A}^{[\star]}\{\star\}\xrightarrow{\we}\Prism_{R/A}^{\rel,[\star]}\{\star\}$.
\end{proposition}

\begin{proof}
    To construct the natural transformation $\gamma$, let $\Tscr_A$ be the category of transversal
    prisms $(B,J)$ equipped with a map of prisms $B\xrightarrow{f} A$.
    For each $(B,f)\in\Tscr_A$, let $A\tensorhat B$ denote the $(p,J)$-adic completion of $A\tensor B$.
    The map $f$ and multiplication induce a map of prisms $A\tensorhat B\rightarrow A$.
    We can evaluate the crystal
    $\Hscr_\Prism(R/A)$ at $B$ to obtain a natural composition
    $$\Prism_{R/A}\rightarrow\Hscr_\Prism(R/A)(B)\we\Prism_{R\tensorhat \overline{B}/A\tensorhat
    B}^\rel\rightarrow\Prism_{R\tensorhat_A\overline{A}/A}^\rel\rightarrow\Prism_{R/A}^\rel,$$
    where the final map uses the $\overline{A}$-linear multiplication map
    $R\tensorhat_A\overline{A}\rightarrow R$.
    
   The construction of the map is natural in $B$: for another choice $B'$ with a map $B \to B'$ in $\Tscr_A$ we get a natural homotopy 
   between the two induced maps 
   $
    \Prism_{R/A} \to \Prism_{R/A}^\rel
   $
    which is induced from a `hammock' between the constructions for these maps:
     $$
     \xymatrix{
     &\Prism_{R\tensorhat \overline{B}/A\tensorhat
    B}^\rel \ar[r]\ar[dd] &\Prism_{R\tensorhat_A\overline{A}/A}^\rel\ar[rd]\ar[dd] \\ 
     \Prism_{R/A}\ar[ru]\ar[rd] & & & \Prism_{R/A}^\rel \\
     &\Prism_{R\tensorhat \overline{B'}/A\tensorhat
    B'}^\rel\ar[r] &\Prism_{R\tensorhat_A\overline{A}/A}^\rel \ar[ru]
    } \ .
    $$
    
    As $\Tscr_A$ is sifted by~\cite[Cor.~2.4.7]{bhatt-lurie-apc}, this implies that the map
    $\Prism_{R/A}\rightarrow\Prism_{R/A}^\rel$ does not depend on
    the transversal prism $B$ chosen (for example by taking the colimit over all these maps in the arrow category). This specifies the natural transformation $\gamma_{R/A}$ when
    $A$ is fixed. A similar argument can be made with the help of an intermediary transversal prism
    $B\in\Tscr_A$
    to construct commutative diagrams
    $$\xymatrix{
        \Prism_{R/A}\ar[r]\ar[d]&\Prism_{R/A}^\rel\ar[d]\\
        \Prism_{R'/A'}\ar[r]&\Prism_{R'/A'}^\rel
    }$$
    for maps of bounded prismatic $\delta$-pairs $(A,R)\rightarrow(A',R')$.
    Siftedness of $\Tscr_A$ again implies that this diagram does not depend on the choice of $B$.
    Continuing in this way, we obtain a natural transformation as desired. To make this
    construction rigorous, consider the $\infty$-category $\Pairs^{\delta\Prism\an}$ of prismatic
    $\delta$-pairs $(A,R)$ where  $A$ is bounded and $R$ is an animated commutative
    $\overline{A}$-algebra. Let $q\colon\Tscr\rightarrow\Pairs^{\delta\Prism\an}$ denote the coCartesian
    fibration with fiber over $(A,R)$ given by $\Tscr_A$. On $\Pairs^{\delta\Prism\an}$ are defined
    sections $\Prism_{(-)/(-)}$ and $\Prism_{(-)/(-)}^\rel$ of the coCartesian fibration
    $\widehat{\DAlg}\rightarrow\Pairs^{\delta\Prism\an}$ with fiber over $(A,R)$ given by the
    $\infty$-category of $(p,I)$-complete derived commutative $A$-algebras. The construction above induces a natural
    transformation $\widetilde{\gamma}\colon\Prism_{(-)/(-)}\circ
    q\rightarrow\Prism_{(-)/(-)}^\rel\circ q$ where naturality can be seen similarly to the case
    above. Taking a left Kan extension along $q$
    produces a natural transformation from $\Prism_{(-)/(-)}\rightarrow\Prism^\rel_{(-)/(-)}$ since
    the fibers of $q$ are contractible.

    Both sides commute with colimits in the variable $R$ as functors to $\D(A)_{(p,I)}^\wedge$.
    Thus, we can immediately reduce to the case of $R=\overline{A}[x]$.
    Both $\Prism_{R/A}$ and $\Prism^\rel_{R/A}$ satisfy Zariski descent in $(A,R)$.
    This Zariski descent holds for $\Prism_{R/A}$ by
    Proposition~\ref{prop:cwdescent} and for $\Prism_{R/A}^\rel$ because if $A\rightarrow C$ is a
    $(p,I)$-completely faithfully flat of $\delta$-rings with \v{C}ech complex $C^\bullet$, then $$\Prism^\rel_{(R\otimes_A
    C^\bullet)_p^\wedge/C^\bullet}\we(\Prism_{R/A}\otimes_AC^\bullet)_{p,I}^\wedge$$
    by~\cite[Const.~7.6]{prisms}; the totalization of the right-hand side is $(p,I)$-complete and
    agrees, modulo $(p,I)$ with $\Prism_{R/A}/(p,I)$ since $A/(p,I)\rightarrow C/(p,I)$ is
    faithfully flat.
    Now, we can use base change in $A$ and Zariski descent in $(A,R)$ to reduce to the case of the universal
    oriented prism $A^0$.
    Then, we can test whether $\gamma_{R/A^0}$ is an equivalence by
    base change along the faithfully flat map to the perfection $A^0_\perf$ of $A^0$.
    But, then, we have a commutative square
    $$\xymatrix{
        \Prism_{R/\bZ}\ar[r]\ar[d]&\Prism_R\ar[d]\\
        \Prism_{R/A^0_\perf}\ar[r]^{\gamma_{R/A^0_\perf}}&\Prism_{R/A^0_\perf}^\rel.
    }$$ The vertical arrows are equivalences because $A^0_\perf$ is perfect, as one can see using the conjugate
    filtrations on the Hodge--Tate cohomology in both cases. The top arrow is an equivalence by Example~\ref{exs:wc}.
\end{proof}

\begin{remark}
    From now on, we drop the distinction between $\Prism_{R/A}^\rel$ and $\Prism_{R/A}$ in the case
    of prismatic $\delta$-pairs.
\end{remark}

\section{The Nygaard filtration}\label{sec:nygaard}

In this section, we introduce the Nygaard filtration and Nygaard-completed Frobenius-twisted
prismatic cohomology relative to $\delta$-pairs and prove the related statements of
Theorem~\ref{thm:intro}.

\begin{definition}
    Let $(A,R)$ be a bounded $\delta$-pair and let $\Hscr_\Prism(R/A)\{i\}$ be the $i$th
    Breuil--Kisin twisted prismatic crystal associated to $(A,R)$. This object is an $A$-module in
    $\D(\WCart)$, so we can form its Frobenius twist
    $$\Hscr_{\Prism}^{(1)}(R/A)\{i\}=\Hscr_{\Prism}(R/A)\{i\}\otimes_{A,\varphi_A}A.$$
    If $(B,I)$ is a transversal prism, then the value of this crystal on $B$ is naturally equivalent to
    the $(p,I)$-completion of
    $$\Prism_{R\tensorhat\overline{B}/(A\tensorhat B)_I^\wedge}\{i\}\otimes_{A,\varphi}A.$$
    Note that this is not typically equivalent to the usual Frobenius twist $\Prism_{R\tensorhat\overline{B}/(A\tensorhat
    B)_I^\wedge}^{(1)}\{i\}$ of~\cite{prisms} because we use only the Frobenius endomorphism of $A$ and not of
    $(A\tensorhat B)_I^\wedge$.
\end{definition}

In this section, we use an approach analogous to the one from \cite{bhatt-lurie-apc} to construct
the Nygaard filtration. Our goal will first be to produce a filtration on
$\R\Gamma(\WCart,\Hscr_{\Prism}^{(1)}(R/A)\{i\})$.

Let $F$ denote the Frobenius endomorphism of $\WCart$.
Recall from \cite[Theorem 3.6.7]{bhatt-lurie-apc} that the diagram
\begin{equation*}
  \label{eqn:derhamsquare}
  \begin{tikzcd}
    \WCart^\HT \rar\dar & \WCart\dar{F}\\
    \Spf \bZ_p  \rar{\rho_{\dR}} & \WCart
  \end{tikzcd}
\end{equation*}
yields pullback diagrams on global sections for any quasi-coherent sheaf on $\WCart$, where
$\rho_\dR$ is the point of $\WCart$ corresponding to the crystalline prism $(\bZ_p,(p))$.
Commutativity of the diagram yields a canonical equivalence
\begin{equation}
  \label{eqn:frobeniustwistderhampoint}
  F^*\Escr|_{\WCart^\HT} \we \Escr_\dR \otimes \Oscr_{\WCart^{\HT}}
\end{equation}
for any prismatic crystal $\Escr$,
where $\Escr_{\dR}$ denotes the $p$-complete complex $\rho_{\dR}^*\Escr$.
We will apply this for $\Escr=\Hscr_\Prism^{(1)}(R/A)\{i\}$.

\begin{remark}[Prismatic cohomology relative to animated crystalline prisms]
    Since $(\bZ_p,(p))$ is not transversal, our definition of $\Hscr_\Prism(R/A)$ when $(A,R)$ is a
    bounded $\delta$-pair does not allow us
    to {\em a priori} compute $\rho_\dR^*\Hscr_\Prism(R/A)=\Hscr_\Prism(R/A)_\dR$ if $A$ has $p$-torsion. Rather,
    we view $\Hscr_\Prism(R/A)_\dR$ as the definition of $\Prism_{R\otimes\bF_p/A}$ where we view
    $A$ as an animated crystalline prism. (Note that while animated prisms are introduced
    in~\cite{bhatt-lurie-prism}, they are used there only to define derived Cartier--Witt stacks and not
    as possible bases for relative prismatic cohomology.) This notation is unambiguous
    as it agrees with the definition of $\Prism_{R\tensor\bF_p/A}^{\an\wc}$ which uses that $A$ is only a
    $\delta$-ring. A similar comment applies to the Frobenius twist $\Hscr_\Prism^{(1)}(R/A)_{\dR}$.
\end{remark}

\begin{remark}[Prismatic cohomology relative to animated prisms]
    More generally, one can use the formalism of prismatic cohomology relative to animated
    $\delta$-pairs to give a definition of prismatic cohomology $\Prism_{R/A}$ when
    $(A,\overline{A})$ is an
    animated prism and $R$ is an animated commutative $\overline{A}$-algebra, for example as
    $\R\Gamma(\WCart_{R/A}^\an,\Oscr)$; see
    Variant~\ref{var:anwc}.
\end{remark}

\begin{lemma}[de Rham comparison for prismatic cohomology relative to $\delta$-rings]
  \label{lem:derham}
    Let $(A,R)$ be a bounded $\delta$-pair.
  There is a canonical multiplicative equivalence
  $\Hscr_{\Prism}^{(1)}(R/A)_\dR \we \widehat{\dR}_{R/A}$,
  where the target denotes $p$-completed derived de Rham cohomology of $R$ over $A$.
  There is a natural commutative diagram
  \[
    \begin{tikzcd}
      F^*\Hscr_{\Prism}^{(1)}(R/A)|_{\WCart^\HT} \dar\rar & \widehat{\dR}_{R/A}\otimes \Oscr_{\WCart^{\HT}}\\
      \Hscr_{\Prism}^{(1)}(R/A)_{\dR} \otimes\Oscr_{\WCart^\HT}\ar[ru]
    \end{tikzcd}
  \]
  of equivalences of quasi-coherent sheaves on $\WCart^\HT$,
    where the vertical map comes from \eqref{eqn:frobeniustwistderhampoint}, the horizontal map
    from the relative de Rham comparison, and the diagonal map is induced from the equivalence of the
    first part of the lemma by tensoring with $\Oscr_{\WCart^\HT}$.
\end{lemma}

\begin{proof}
  We follow the argument from \cite[Section 5.4]{bhatt-lurie-apc}.
  For a prism $B$, we have a canonical pullback diagram
    $$\xymatrix{
        \Spf \overline{B}\ar[r]\ar[d]&\WCart^\HT\ar[d]\\
        \SSpf B\ar[r]&\WCart.
    }$$
  The quasi-coherent sheaf $F^*\Hscr^{(1)}_{\Prism}(R/A)$ pulls back to
  \[
    \Prism_{R\tensorhat \overline{B} / A\tensorhat B} \otimes_{A\tensorhat B,\varphi} A\tensorhat B
  \]
  on $\SSpf B$, and hence to
  \[
    \Prism_{R\tensorhat \overline{B} / A\tensorhat B} \otimes_{A\tensorhat B,\varphi} \overline{A\tensorhat B}
  \]
    on $\Spf\overline{B}$. By the relative de Rham comparison~\cite[Prop.~5.2.5]{bhatt-lurie-apc}, this is naturally equivalent to
  \[
    \widehat{\dR}_{R\tensorhat \overline{B} / \overline{A\tensorhat B}} \we \widehat{\dR}_{R/A} \tensorhat \overline{B}.
  \]
  Hence we have a natural equivalence
  \[
    F^*\Hscr^{(1)}_{\Prism}(R/A)|_{\WCart^{\HT}} \we \widehat{\dR}_{R/A} \otimes\Oscr_{\WCart^{\HT}},
  \]
  providing the horizontal map in the commutative diagram.
  Together with the equivalence \eqref{eqn:frobeniustwistderhampoint}, this determines a unique multiplicative equivalence
  \[
      \Hscr_\Prism^{(1)}(R/A)_\dR\otimes\Oscr_{\WCart^{\HT}}\we \widehat{\dR}_{R/A} \otimes\Oscr_{\WCart^{\HT}}.
  \]
  of quasi-coherent sheaves on $\WCart^{\HT}$ making the diagram commute. To finish the proof, we need to check that this is induced from a natural multiplicative equivalence 
  \[
      \Hscr_\Prism^{(1)}(R/A)_\dR\we \widehat{\dR}_{R/A}
  \]
    in $\D(\bZ_p)$.
  To see this, it suffices to check that the map
  \[ 
    \Hscr_\Prism^{(1)}(R/A)_\dR\to \R\Gamma(\WCart^{\HT},\widehat{\dR}_{R/A} \otimes\Oscr_{\WCart^{\HT}})
  \]
  factors uniquely and multiplicatively through the canonical map
  \[
    \widehat{\dR}_{R/A} \to \R\Gamma(\WCart^{\HT},\widehat{\dR}_{R/A} \otimes\Oscr_{\WCart^{\HT}}).
  \]
    Since the $\Hscr_{\Prism}^{(1)}(R/A)_\dR$ is left Kan extended from $\delta$-pairs $(A,R)$ with $R$ quasisyntomic
    over $A$ we may reduce to that situation. Now, since $\L_{R/A}$ has $p$-complete Tor-amplitude
    in $[0,1]$, it follows that $\Lambda^i\L_{R/A}$ has $p$-complete Tor-amplitude in $[0,i]$; see
    the argument of~\cite[Lem.~4.34]{bms2}. It follows by reducing modulo $p$ and using the
    exhaustive conjugate filtration, that
    $\widehat{\dR}_{R/A}$ is coconnective. Using that filtered colimits commute with totalizations
    of uniformly
    bounded above cosimplicial objects, quasisyntomic descent for $\widehat{\dR}_{R/A}$ reduces,
    using the conjugate filtration modulo $p$ again, to $p$-complete faithfully flat descent for
    $\Lambda^i\L_{R/A}$ for $i\geq 0$, which follows from~\cite[Sec.~3]{bms2}. Using these
    observations, we may reduce to the case where $\L_{R/A}$ has Tor-amplitude in $[1,1]$.
    But then $\widehat{\dR}_{R/A}$ is discrete (as each and agrees with the connective cover of
    $\R\Gamma(\WCart^{\HT},\widehat{\dR}_{R/A}\otimes\Oscr_{\WCart^{\HT}})$.
    Moreover, $\Hscr_\Prism^{(1)}(R/A)_{\dR}$ is connective, so the desired lift is uniquely determined.
\end{proof}

\begin{corollary}
    If $(A,R)$ is a bounded $\delta$-pair, then there is a natural
  pullback diagram
  \[
    \begin{tikzcd}
        \R\Gamma(\WCart, \Hscr^{(1)}_\Prism(R/A)\{i\}) \dar\rar & \widehat{\dR}_{R/A}\dar\\
        \R\Gamma(\WCart,F^*\Hscr^{(1)}_\Prism(R/A)\{i\}) \rar &
        \widehat{\dR}_{R/A}\otimes\R\Gamma(\WCart^{\HT},\Oscr_{\WCart^{\HT}})
    \end{tikzcd}
  \]
  for each $i\in\bZ$.
\end{corollary}
\begin{proof}
  We combine the pullback diagram obtained from \eqref{eqn:derhamsquare} with the equivalence
    $\Hscr^{(1)}_{\Prism}(R/A)\{i\}_\dR \simeq \widehat{\dR}_{R/A}$ obtained by multiplying the
    equivalence from Lemma \ref{lem:derham} with the canonical trivialization of the Breuil--Kisin twist $\bZ_p\{i\}$ (compare \cite[Variant 5.4.13]{bhatt-lurie-apc}).
\end{proof}

\begin{definition}[Nygaard-filtered Frobenius-twisted prismatic cohomology]
    Let $(A,R)$ be a bounded $\delta$-pair and fix $i\in\bZ$.
  We define the Nygaard filtration on $\R\Gamma(\WCart,\Hscr^{(1)}_\Prism(R/A)\{i\})$ as the pullback of filtered spectra
  \[
    \begin{tikzcd}
      \N^{\geq \star}\R\Gamma(\WCart,\Hscr^{(1)}_\Prism(R/A)\{i\})\dar\rar& \F^{\geq \star}_{\H} \widehat{\dR}_{R/A}\dar\\
      \R\Gamma(\WCart, \N^{\geq \star} F^*\Hscr^{(1)}_\Prism(R/A)\{i\}) \rar & \F^{\geq\star}_{\H}
        \widehat{\dR}_{R/A} \otimes\R\Gamma(\WCart^{\HT},\Oscr_{\WCart^{\HT}}),
    \end{tikzcd}
  \]
  where the lower left term uses the pointwise Nygaard filtration on the prismatic crystal
  \[
    B\mapsto \Prism_{R\tensorhat \overline{B} / A\tensorhat B}\{i\}\otimes_{A\tensorhat B, \varphi} A\tensorhat B
  \]
  (see \cite[Section 12]{prisms})
  corresponding to $F^*\Hscr^{(1)}_{\Prism}(R/A)$, and the bottom map uses that the composite 
  \[
    F^*\Hscr^{(1)}_{\Prism}(R/A)\{i\} \to \Hscr_\Prism^{(1)}(R/A)\{i\}_{\dR}\otimes \Oscr_{\WCart^\HT} \simeq \widehat{\dR}\otimes\Oscr_{\WCart^\HT}
  \]
  is a combination of the canonical trivialisation of $F^*\Oscr_{\WCart}\{i\}$ on the Hodge-Tate locus (\cite[Remark 2.5.7 \& Example 3.5.2]{bhatt-lurie-apc}), and the de Rham comparison map
  \[
    \Prism_{R\tensorhat \overline{B} / A\tensorhat B}\otimes_{A\tensorhat B, \varphi} A\tensorhat B \to \Prism_{R\tensorhat \overline{B} / A\tensorhat B}\otimes_{A\tensorhat B, \varphi}\overline{A\tensorhat B} \we \widehat{\dR}_{R/A} \tensorhat \overline{B}
  \]
    which takes the Nygaard filtration into the Hodge filtration
    (see~\cite[Const.~5.5.3]{bhatt-lurie-apc}).

\end{definition}

\begin{definition}[Nygaard-completed Frobenius-twisted prismatic cohomology]\label{Nygaard}
    If $(A,R)$ is a bounded $\delta$-pair,
  we define Nygaard-completed (Frobenius-twisted, Breuil--Kisin twisted)
    prismatic cohomology of $R$ over $A$ as the completion 
  \[
    \N^{\geq \star} \Prismhat_{R/A}^{(1)}\{i\} := (\N^{\geq
    \star}\R\Gamma(\WCart,\Hscr^{(1)}_\Prism(R/A)\{i\}))^{\wedge}.
  \]
\end{definition}

\begin{remark}[Absolute Nygaard comparison]\label{rem:absolute_nygaard_comparison}
    Since the Frobenius endomorphism of the $\delta$-ring $\bZ_p$ is the identity,
    $\N^{\geq\star}\Prismhat_{R/\bZ_p}^{(1)}\{i\}\we\N^{\geq\star}\Prismhat_R$, where the latter denotes
    the Nygaard complete absolute prismatic cohomology of~\cite{bms2,bhatt-lurie-apc}.
    Indeed, in this case $\Hscr_\Prism^{(1)}(R/\bZ_p)\we\Hscr_\Prism(R)$, in the notation
    of~\cite{bhatt-lurie-apc}, and our construction of the Nygaard filtration and completion agrees
    with the one in~\cite[Sec.~5.5]{bhatt-lurie-apc}.
\end{remark}

\begin{construction}\label{constr_diff}
Recall the Hodge--Tate point $\eta\colon \Spf(\bZ_p) \to \WCart$ which assigns to every
$p$-nilpotent ring $R$ the Cartier-Witt divisor given by $V(1)\colon W(R)\to W(R)$. In analogy to
\cite{bhatt-lurie-apc}, we write $\DiffractedHodge_{R/A}$ for $\eta^*\Hscr_{\Prismbar}(R/A)$.
This inherits a conjugate filtration $\F_\star^\conj\DiffractedHodge_{R/A}$ from the $\delta$-conjugate filtration on the relative
Hodge--Tate crystal by Proposition~\ref{prop:conjugate}, and, since
\[
    \gr^{\delta\conj}_i \Hscr_{\Prismbar}(R/A) \we \widehat{\L\Omega}^i_{R/A} \otimes
    \Oscr_{\WCart^\HT}\{-i\}[-i],
\]
we have $\gr^\conj_i \DiffractedHodge_{R/A} \we\widehat{\L\Omega}^i_{R/A}[-i]$.
\end{construction}

\begin{lemma}
  \label{lem:fibersequences}
    If $(A,R)$ is a bounded $\delta$-pair, then there are natural
  maps of horizontal fiber sequences
  \[
    \begin{tikzcd}
      \gr^i_\N \Prismhat_{R/A}^{(1)}\{j\} \rar\dar & \F^{\conj}_{\leq i}
        \DiffractedHodge_{R/A}\rar{\Theta+i}\dar{\id} & \F^{\conj}_{\leq i-1} \DiffractedHodge_{R/A}\dar{\mathrm{incl}} \\
      \R\Gamma(\WCart, \gr^i_\N F^*\Hscr^{(1)}_\Prism(R/A)\{j\}) \rar  \dar & \F^{\conj}_{\leq i}
        \DiffractedHodge_{R/A}\rar{\Theta+i}\dar & \F^{\conj}_{\leq i} \DiffractedHodge_{R/A}\dar\\
      \Prismbar_{R/A}\{i\} \rar & \DiffractedHodge_{R/A}\rar{\Theta+i} & \DiffractedHodge_{R/A}
    \end{tikzcd}
  \]
    in $\D(A)_p^\wedge$,
    where $\Theta$ is the Sen operator on $\F_{\leq\star}^\conj\DiffractedHodge_{R/A}$ induced from
    the Hodge--Tate crystal $\Hscr_\Prismbar(R/A)$ under the description of quasi-coherent sheaves on
    $\WCart^\HT$ of~\cite[Sec.~3.5 \& Remark 4.9.10 \& Remark 5.5.8]{bhatt-lurie-apc}.
\end{lemma}

\begin{proof}
  The Hodge--Tate comparison provides that for every transversal prism $B$ the relative Frobenius map
  \[
    \gr^i_\N(\Prism_{R\tensorhat \overline{B}/ A\tensorhat B}\{j\} \otimes_{A\tensorhat B,\varphi} A\tensorhat B) \to \Prismbar_{R\tensorhat\overline{B}/A\tensorhat B}\{i\}
  \]
  factors through an equivalence to $\F^{\conj}_{\leq i}\Prismbar_{R\tensorhat
    \overline{B}/A\tensorhat B}\{i\}$. In particular, $\gr^i_\N F^*\Hscr^{(1)}_\Prism(R/A)$ is
    supported on the Hodge--Tate locus, and global sections can be expressed through the Sen
    operator, leading to the bottom two rows.

  For the top row (and the map between the top two rows), we first observe that the pullback diagram
    defining $\N^{\geq *} \Prismhat_{R/A}^{(1)}$ provides a pullback diagram
  \[
    \begin{tikzcd}
        \gr^i_\N \Prismhat_{R/A}^{(1)}\{j\} \dar\rar& \widehat{\L\Omega}^i_{R/A}[-i]\dar\\
        \R\Gamma(\WCart, \gr^i_\N F^*\Hscr^{(1)}_\Prism(R/A)\{j\}) \rar &
        \widehat{\L\Omega}^i_{R/A}[-i]\otimes \R\Gamma(\WCart^{\HT},\Oscr_{\WCart^{\HT}}).
    \end{tikzcd}
  \]
  Relative de Rham comparison on the level of prismatic crystals identifies $\gr^i_\N F^*\Hscr^{(1)}_\Prism(R/A)\{j\}$ with the prismatic crystal
  \[
    B \mapsto \F^{\conj}_{\leq i} \Prismbar_{R\tensorhat \overline{B}/ A\tensorhat B},
  \]
  and the bottom map is obtained from the canonical map
  \[
    \F^{\conj}_{\leq i} \Prismbar_{R\tensorhat \overline{B}/ A\tensorhat B} \to \gr^\conj_i
    \Prismbar_{R\tensorhat \overline{B}/A\tensorhat B} \we  \widehat{\L\Omega}^i_{R/A}[-i].
  \]
  The bottom two terms can now be expressed as fibers of the respective Sen operators $\Theta + i$, and the right vertical map is the canonical map
  \[
      \widehat{\L\Omega}^i_{R/A}[-i] \to \fib(\widehat{\L\Omega}^i_{R/A}[-i] \xto{\Theta+i}
      \widehat{\L\Omega}^i_{R/A}[-i]) \we \widehat{\L\Omega}^i_{R/A}[-i] \oplus \widehat{\L\Omega}^i_{R/A}[-i-1].
  \]
  This yields the claimed first row (and the map to the second row).
\end{proof}

\begin{corollary}
    \label{cor:nygaardproperties}
    For each $i\in\bZ$, $\gr^i_\N\Prismhat_{R/A}^{(1)}\{j\}$ is sifted-colimit preserving, has universal $p$-complete descent as a functor on bounded $\delta$-pairs, satisfies base change in the base
    $\delta$-ring, and is invariant under $p$-adically quasi-\'etale extensions, finitely generated completions,
    and localizations in the base $\delta$-ring (as in Proposition~\ref{prop:quasietalecrystal} and
    Corollary~\ref{cor:inv_loc_comp}). Moreover, if we consider $\N^{\geq
    \star}\Prismhat_{R/A}^{(1)}\{j\}$ as a functor to complete filtered objects then it is sifted-colimit preserving, has universal $p$-complete descent as a functor on bounded
    $\delta$-pairs, satisfies base change in the base $\delta$-ring, and is invariant under
    $p$-adically quasi-\'etale  extensions in the base $\delta$-ring.
\end{corollary}

\begin{proof}
    The second claim follows from the first. The first claim follows from the fact that
    $\gr^i_\N\Prismhat_{R/A}^{(1)}$ is by Construction \ref{constr_diff} and the first row of Lemma~\ref{lem:fibersequences} built by finite limits and colimits from the $p$-completed differentials forms $\widehat{\L\Omega}^k_{R/A}$.
   The forms have the necessary descent by Lemma~\ref{lem:descentdifferentialforms}, the
    necessary invariance under quasi-\'etale extensions, finitely generated completions, and
    localizations,
    and by definition preserve sifted colimits and satisfy base change.
\end{proof}

In the case of a prismatic $\delta$-pair $(A,R)$, i.e. one where $A$ admits the structure of a
prism, and $R$ the structure of an $\overline{A}$-algebra, $\N^{\geq
\star}\Prismhat_{R/A}^{(1)}\{i\}$ now potentially has two meanings: The one defined above, and the
Nygaard filtration from relative prismatic cohomology. Our goal is now to identify the two.
Temporarily, we will write $\N^{\geq \star}_{\rel}\Prismhat_{R/A}^{(1)}\{i\}$ for the Nygaard
filtration from relative prismatic cohomology.

\begin{definition}
  \label{def:nygaardcomparison}
  We define a comparison map $\N^{\geq\star}\Prismhat_{R/A}^{(1)}\{i\} \to \N^{\geq
    \star}_{\rel}\Prismhat_{R/A}^{(1)}\{i\}$ as in the proof of Proposition~\ref{prop:prismaticcomparison}:
    before completion, given a transversal prism $B$ with a map of prisms to $A$, we have the composite
    map of filtered spectra
  \[
    \N^{\geq \star}R\Gamma(\WCart,\Hscr_{\Prism}^{(1)}(R/A)\{i\}) \to
    R\Gamma(\WCart,\N^{\geq\star}F^*\Hscr_{\Prism}^{(1)}(R/A)\{i\}) \to \N^{\geq\star}_{\rel}
    \Prism_{R\tensorhat \overline{B} / A\tensorhat B}^{(1)}\{i\} \to \N^{\geq\star}_{\rel}
    \Prism_{R/A}^{(1)}\{i\},
  \]
  where the first map comes from the definition, the second is evaluation at $\rho_B\colon\SSpf B \to \WCart$,
    and the third uses the multiplication map. After completion, this gives rise to a map
  \[
    \N^{\geq \star} \Prismhat_{R/A}^{(1)}\{i\} \to \N^{\geq \star}_{\rel}\Prismhat_{R/A}^{(1)}\{i\},
  \]
    which one shows as in Proposition~\ref{prop:prismaticcomparison} does not depend on
    $B$ and which is natural in prismatic $\delta$-pairs $(A,R)$.
\end{definition}

\begin{proposition}[Relative prismatic Nygaard comparison]\label{prop:nygaardcomparison}
  The map from Definition \ref{def:nygaardcomparison} is an equivalence for all prismatic
    $\delta$-pairs $(A,R)$.
\end{proposition}
\begin{proof}
  As functors to complete filtered spectra, both sides preserve sifted colimits. For the left hand
    side this is part of Corollary~\ref{cor:nygaardproperties}, for the right hand side this is seen
    analogously using $\gr^i_{\N_{\rel}}\Prismhat_{R/A}\{j\} \we \F^{\Prism\conj}_{\leq i}
    \Prismbar_{R/A}\{j\}$ and the fact that the associated graded pieces of the $\Prism$-conjugate
    filtration are also built out of differential forms.
  This reduces us to the case $R = \overline{A}[x_1,\ldots,x_n]$. Also, both sides satisfy
    base change and descent in $A$ as functors to complete filtered spectra, so we may assume that $A$ is the free
    oriented
    prism $\bZ\{d\}[\delta(d)^{-1}]^\wedge_{p,d}$. For this $A$, the map $A\to A^\perf$ to its
    colimit perfection is $(p,d)$-completely faithfully flat, so by again appealing to base change,
    we may reduce to the case of perfect prisms $A$. In that case, we have a
    commutative diagram
  \[
    \begin{tikzcd}
      \N^{\geq\star} \Prismhat_{R/\bZ_p}^{(1)}\{j\}\ar[rd]\dar{\we}&\\
      \N^{\geq\star} \Prismhat_{R/A}^{(1)}\{j\}\rar& \N^{\geq\star}_{\rel}\Prismhat_{R/A}^{(1)}\{j\}.
    \end{tikzcd}
  \]
    The left vertical map is an equivalence, which follows from the fact that
    $\widehat{\L\Omega}^i_{R/\bZ_p}\to\widehat{\L\Omega}^i_{R/A}$ is an equivalence, since
    $(\L_{A/\Z_p})_p^\wedge \we 0$ as $A$ is perfect.
    The diagonal composite agrees with the completion of the composite
  \[
    \N^{\geq\star}\R\Gamma(\WCart, \Hscr_{\Prism}(R/\bZ_p)^{(1)}\{j\}) \to 
    R\Gamma(\WCart, \N^{\geq\star} F^*\Hscr_{\Prism}(R/\bZ_p)^{(1)}\{j\}) \to 
    \N^{\geq\star}_{\rel} \Prism_{R\tensorhat \overline{A}/ \bZ_p \tensorhat A}^{(1)}\{j\} \to \N^{\geq\star}_{\rel} \Prism_{R/A}\{j\},
  \]
  which is the map from \cite[Construction 5.6.1]{bhatt-lurie-apc}. According to
    \cite[Theorem 5.6.2]{bhatt-lurie-apc}, it is an equivalence.
\end{proof}

From now on we drop the notation $\N^{\geq\star}_{\rel}$.

\begin{corollary}
  \label{cor:rqrsp-nygaard}
    If $(A,R)$ is a bounded $\delta$-pair and $A\rightarrow R$
  is relatively quasiregular semiperfectoid (see Definition \ref{rel_quasiperf}), then $\N^{\geq \star}\Prismhat_{R/A}^{(1)}\{j\}$ is
    discrete with discrete associated graded, and the filtration on it agrees with the underived
    pullback of the Hodge--Tate filtration along $\varphi\colon \Prismhat_{R/A}^{(1)}\{j\}\to
    \Prism_{R/A}^{[-j]}\{j\}$.
\end{corollary}

\begin{proof}
  Let $A\to A' \to R$ be a factorization of $A\to R$ such that $A\to A'$ is $p$-completely relatively perfect,
    $(A',R)$ is a prismatic $\delta$-pair, and $\L_{R/A'}$ has $p$-complete Tor-amplitude in
    $[1,1]$.
  In the diagram
  \[
    \begin{tikzcd}
         \Prismhat_{R/A}^{(1)}\{j\} \rar\dar & \Prism_{R/A}^{[-j]}\{j\}\dar\\
         \Prismhat_{R/A'}^{(1)}\{j\} \rar & \Prism_{R/A'}^{[-j]}\{j\}
    \end{tikzcd}
  \]
  the vertical maps are all isomorphisms by Corollary \ref{cor:nygaardproperties}, so it suffices to prove the claim in the case of a
    prismatic $\delta$-pair. For that, it suffices to check that the associated graded terms are
    discrete, and that the map on associated graded pieces is injective. In the prismatic case, we have a factorization
    \[
      \gr^i_\N \Prismhat_{R/A'}^{(1)}\{j\}\we \F^{\conj}_{\leq i}\Prismbar_{R/A'}\{i\} \to \Prismbar_{R/A'}\{i\},
    \]
    where the latter map is an injective map between discrete objects since all associated graded
    terms of the conjugate filtration $\widehat{\L\Omega}^i_{R/A}[-i]$ are discrete.
\end{proof}

\begin{remark}
    Since we can cover every $R/A$ with $R$ quasisyntomic over $A$ by relatively quasiregular
    semiperfectoids, Corollary \ref{cor:rqrsp-nygaard} yields another characterization of the Nygaard filtration through descent and left Kan extension
    by Lemma~\ref{lem:unfold_ring}.
\end{remark}

\begin{remark}
  If $R$ is smooth over $A$ (which never happens in the prismatic case), then the conjugate
    filtration on $\DiffractedHodge_{R/A}$ is the Postnikov filtration, and the fiber sequences from
    Lemma~\ref{lem:fibersequences} exhibit $\gr^i_\N\Prismhat^{(1)}_{R/A}\{j\}$ as the
    $(-i)$-connective cover of $\Prismbar_{R/A}\{i\}$. This identifies
    $\N^{\geq\star}\Prismhat^{(1)}_{R/A}\{j\}$ as the connective cover in the Beilinson
    $t$-structure (see~\cite[Sec.~5]{bms2}) of the Hodge--Tate
    filtered $\Prism_{R/A}\{j\}$, providing yet another characterization. The analogue of this fact
    curiously \emph{also} holds in the prismatic case if $R$ is smooth over $\overline{A}$. The
    relationship between these two facts can be explained as follows: if $R$ is a smooth
    $A$-algebra, then
    the $I$-adic filtration on $R$ has
    associated graded given by the smooth $\overline{A}$-algebra $\overline{R}[I]$, with
    $\overline{R} = R\otimes_A \overline{A}$ and $\overline{R}[I]$ the symmetric algebra
    on $R\otimes_A I$. Filtered prismatic cohomology, developed below in
    Section~\ref{sec:filtered}, provides a filtration on
    $\gr^i_\N\Prismhat^{(1)}_{R/A}$ with associated graded given by $\gr^i_\N
    \Prismhat^{(1)}_{\overline{R}[I]/A}$, from which the connectivity statement from relative
    prismatic cohomology implies the one in prismatic cohomology relative $\delta$-rings.
\end{remark}

\section{Relative syntomic cohomology}\label{sec:syntomic}

We write $\can\colon\N^{\geq i}\Prismhat_{R/A}^{(1)}\{i\}\rightarrow\Prismhat_{R/A}^{(1)}\{i\}$ for the
canonical map. We define the divided Frobenius maps following~\cite[Sec.~5.7]{bhatt-lurie-apc}.
Finally, we define relative syntomic cohomology.

\begin{definition}[The relative Frobenius map]
    The Frobenius endomorphism on the prismatic crystal $\Hscr_\Prism(R/A)$ factors as
    $$\Hscr_\Prism(R/A)\xrightarrow{\text{$\Oscr$-linear}}\Hscr_\Prism^{(1)}(R/A)\xrightarrow{\text{$\varphi$-semilinear}}
    F^*\Hscr_\Prism^{(1)}(R/A)\xrightarrow{\text{$\Oscr$-linear}}\Hscr_\Prism(R/A).$$
    By definition, $\N^{\geq i}F^*\Hscr_\Prism^{(1)}(R/A)\rightarrow\Hscr_\Prism(R/A)$ factors
    canonically through $\Iscr^{i}\otimes\Hscr_\Prism(R/A)$.
    Tensoring with $\Oscr_{\WCart}\{i\}$ and using the equivalence
    $F^*(\Oscr_{\WCart}\{i\})\we\Iscr^{-i}\{i\}$ (\cite[Remark 2.5.9 \& Remark 3.6.5]{bhatt-lurie-apc}) we obtain a map
    $$\N^{\geq i}F^*\Hscr^{(1)}_\Prism(R/A)\{i\}\rightarrow\Hscr_\Prism(R/A)\{i\},$$
    which is moreover compatible with the Nygaard filtration on the left-hand side (in weights at least $i$)
    and the $\Iscr$-adic filtration on the right-hand-side.
    Hence, after taking global sections, there are maps $$\N^{\geq
    i}\R\Gamma(\WCart,\Hscr_\Prism^{(1)}(R/A)\{i\})\rightarrow\R\Gamma(\WCart,\N^{\geq
    i}F^*\Hscr_\Prism^{(1)}(R/A)\{i\})\rightarrow\Prism_{R/A}\{i\}.$$
    Since the $\Iscr$-adic filtration on $\Hscr_\Prism(R/A)\{i\}$ is complete, inducing the complete
    Hodge--Tate filtration on $\Prism_{R/A}\{i\}$,
    this map factors through the Nygaard completion, inducing a filtered map
    $$\N^{\geq\star}\Prismhat^{(1)}_{R/A}\{i\}\rightarrow\Prism^{[\star-i]}_{R/A}\{i\}.$$
\end{definition}

\begin{definition}[The map $c_0$]
    \label{def:c0}
    The $\Oscr_{\WCart}$-linear map of crystals
    $\Hscr_\Prism(R/A)\{i\}\rightarrow\Hscr_\Prism^{(1)}(R/A)\{i\}$ induces a map
    $\Prism_{R/A}\{i\}\rightarrow\R\Gamma(\WCart,\Hscr_\Prism^{(1)}(R/A)\{i\}$ on global sections,
    which we may compose with the map to the Nygaard completion $\Prismhat_{R/A}^{(1)}\{i\}$ to
    obtain a map $$c_0\colon\Prism_{R/A}\{i\}\rightarrow\Prismhat_{R/A}^{(1)}\{i\}.$$
\end{definition}

\begin{construction}\label{filtration}
    For every filtration $N^{\geq\star}$ of objects of $\D(\bZ)$ indexed on $(\bN,\geq)$, we form a new filtration indexed on $(\bN, \geq )$ given as 
    \[
    N^{\geq p \cdot \star} \otimes p^\star \bZ,
    \]
    the (derived) tensor product of filtered objects with respect to Day convolution
    of the re-scaled filtration $N^{\geq p\cdot \star}$ and the $p$-adic
    filtration on $\bZ$. Concretely, this filtration has
    the form
    \[
    \ldots \to N^{\geq 2p} \oplus_{N^{\geq p}, p} N^{\geq p} \oplus_{N^{\geq 0}, p} N^{\geq 0} \to
    N^{\geq p} \oplus_{N^{\geq 0}, p} N^{\geq 0} \xto{(\can, p)} N^{\geq 0},
    \]
    where the pushout maps to the left are the canonical maps and the maps to the right are the multiplication by $p$ maps. 
    For example, if $N^{\geq\star}$ is the $(d)$-adic filtration
    on an oriented transversal prism, then this construction produces the $(d^p,p)$-adic filtration. 
\end{construction}

\begin{lemma}\label{lem:complete}
    If $N^{\geq p \star}$ is complete and every $N^{\geq pk}$ for $k \in\bN$ is $p$-complete, then
    the filtration $N^{\geq p\cdot\star}\otimes p^\star\bZ$ is complete too.
\end{lemma}
    \begin{proof}
   We first note that all the filtered pieces in $N^{\geq p\cdot\star}\otimes p^\star\bZ$ are finite colimits of the filtered pieces of $N^{\geq p \cdot \star}$, thus they are $p$-complete. But then also the inverse limit of $N^{\geq p\cdot\star}\otimes p^\star\bZ$ is $p$-complete. Therefore in order to show that it vanishes, it is enough to show that 
   \[
   0 = \underleftarrow{\lim} \left(N^{\geq p\cdot\star}\otimes p^\star\bZ\right) / p =  \left(N^{\geq p\cdot\star}\otimes p^\star\bF_p\right)
   \]
   But in the filtration $p^\star\bF_p$ all transition maps are zero. As a result, we see that the
        the Day convolution tensor product of $N^{\geq p\cdot\star}$ with $p^\star\bF_p$ is given by the filtration which in weight $k$ is given by the mod $p$ reduction of 
   \[
   N^{\geq pk} \oplus  N^{\geq p(k-1)}/N^{\geq pk} \oplus \ldots \oplus  N^{\geq 0}/N^{\geq p}
   \] 
   with transition maps only non-zero on the first summands (without the mod $p$-reduction this would be the tensor product of $N^{\geq p\cdot\star}$ with the filtration
   $\cdots \xto{0} \bZ \xto{0} \bZ$). In other words: the filtration $N^{\geq p\cdot\star}\otimes p^\star\bF$ is given by the direct
   sum of $N^{\geq p\cdot\star}/ p$ with a filtration all of whose transition maps are zero. This implies that the inverse limit is the inverse limit of 
   $N^{\geq p\cdot\star}/p$ which vanishes by assumption. 
\end{proof}

\begin{lemma}[The map $c$]
    There exists a natural
    $\varphi$-semilinear, $\bN$-filtered multiplicative map
    $$c\colon\Prism_{R/A}^{[\star]}\{\star\}\rightarrow(\N^{\geq p\star}\Prismhat_{R/A}^{(1)}\{\star\}\otimes
    p^\star\bZ).$$
\end{lemma}

\begin{proof}
    Both sides are complete filtrations, the left-hand side by definition and the right-hand side
    by Lemma~\ref{lem:complete}. Moreover, they are left Kan
    extended as functors to complete filtrations from finitely presented free $\delta$-pairs by
    Corollaries~\ref{cor:kanextended} and~\ref{cor:nygaardproperties}.
    Thus to construct the map we can assume that $(A,R)$ is a finitely presented free
    $\delta$-pair. By relative quasisyntomic descent in $R$ (Proposition~\ref{prop:cwdescent} and Corollary~\ref{cor:nygaardproperties})
    we further reduce to the prismatic case where $(A,I)$ is a prism and $R$ is an animated
    commutative $\overline{A}$-algebra
    and assume additionally that $\L_{R/\overline{A}}$ has $p$-complete Tor-amplitude in $[1,1]$. In this case, all objects
    in question are discrete. We use the map
    $c_0\colon\Prism_{R/A}\{i\}\rightarrow\Prismhat^{(1)}_{R/A}\{i\}$
    from Definition~\ref{def:c0}. It is enough to prove the lemma when $i=0$ and when
    $A$ is orientable and transversal, with $I=(d)$, in which
    case $c_0\colon \Prism_{R/A}\rightarrow\Prismhat^{(1)}_{R/A}$ is a $\varphi_A$-semilinear map of
    commutative rings, which sends $d$ to $\varphi_A(d)=d^p+p\delta(d)$. As $d\in\N^{\geq
    1}\Prismhat_{R/A}^{(1)}$, for example by functoriality in $R$ and~\cite[Ex.~5.1.4]{prisms}, it follows that
    $\varphi_A(d)$ is in weight $\geq 1$ with respect to the Day convolution filtration
    $(\N^{\geq p\star}\Prismhat_{R/A}^{(1)}\otimes p^\star\bZ)$. By multiplicativity,
    $\varphi_A(d^n)$ is in weight $\geq n$. This shows that $c_0$ refines to a multiplicative filtered map as
    desired for such pairs $(A,R)$. Descending and left Kan extending produces the map $c$.
%
%
\end{proof}

\begin{remark}
    We expect that by forgetting the filtered structure on $c$ one recovers the map $c_0$, but we do
    not explore that here.
\end{remark}

\begin{definition}[The divided Frobenius map]
    We let $c\varphi$ denote the resulting
    composition
    $$\N^{\geq
    i}\Prismhat_{R/A}^{(1)}\{i\}\rightarrow\Prism_{R/A}\{i\}\xrightarrow{c}\Prismhat_{R/A}^{(1)}\{i\}.$$
\end{definition}

\begin{definition}[Relative syntomic complexes]
    The relative $p$-adic syntomic complexes $\bZ_p(i)(R/A)$ are defined as
    $$\bZ_p(i)(R/A)=\fib\left(\N^{\geq
    i}\Prismhat_{R/A}^{(1)}\{i\}\xrightarrow{\can-c\varphi}\Prismhat_{R/A}^{(1)}\{i\}\right).$$
\end{definition}

\begin{remark}\label{rem:syncomparison}
    When $A=\bZ_p$, we have that $\bZ_p(i)(R/\bZ_p)$ agrees with the syntomic complexes
    of~\cite{bms2,bhatt-lurie-apc} by Remark~\ref{rem:absolute_nygaard_comparison}.
\end{remark}

The following result establishes Corollary~\ref{cor:intro} from the introduction.

\begin{corollary}\label{cor:syntomicdescent}
    For each $i\in\bZ$,
    the relative $p$-adic syntomic complexes $\bZ_p(i)(R/A)$ satisfy descent for maps of pairs
    $(A,R)\rightarrow (B,S)$ such that $R\rightarrow S$ is
    a universal descent morphism (with no condition on $A \to B$).
\end{corollary}

\begin{proof}
    The corollary follows immediately from Corollary~\ref{cor:nygaardproperties}.
\end{proof}

 We close this section by noting that one can compute syntomic cohomology also by a non-completed
 version of Frobenius-twisted prismatic cohomology. Specifically we set 
 $\Prism_{R/A}^{(1)}\{i\} = \Prism_{R/A}\{i\} \tensorhat_{A,\varphi} A$ and denote by $\tilde{c}$
 the  map $\Prism_{R/A}^{(1)}\{i\}\rightarrow\Prismhat_{R/A}^{(1)}\{i\}$ induced from $c$. 
 Let $c'\colon\Prism_{R/A}\{i\}\rightarrow\Prism_{R/A}^{(1)}$ be the map induced by base change.
 \begin{construction}
 We can define the Nygaard filtration on $\Prism_{R/A}^{(1)}\{i\}$ as the pullback
 \[
 \xymatrix{
 \N^{\geq \star}\Prism_{R/A}^{(1)}\{i\}  \ar[r] \ar[d] \ar@{}[dr]|{\lrcorner} & \N^{\geq \star}\Prismhat_{R/A}^{(1)}\{i\}  \ar[d]\\
 \Prism_{R/A}^{(1)}\{i\} \ar[r]^{\tilde c} &  \Prismhat_{R/A}^{(1)}\{i\}.
 }
 \]
The completion of $\N^{\geq\star}\Prism_{R/A}^{(1)}\{i\}$ recovers $\Prismhat_{R/A}^{(1)}\{i\}$ by
construction. Since the map $c\colon\Prism_{R/A}\{i\}\rightarrow\Prismhat_{R/A}\{i\}$ factors
     through $\tilde{c}$ by definition, there are decompleted maps $\varphi'\colon\N^{\geq
     i}\Prism_{R/A}^{(1)}\rightarrow\Prism_{R/A}$ and there are decompleted divided Frobenius maps
$c'\varphi'\colon\N^{\geq i}\Prism^{(1)}_{R/A}\{i\}\rightarrow\Prism^{(1)}_{R/A}\{i\}$ compatible with
the completed ones.
We can then define the equalizer
$$\bZ_p(i)(R/A)^{\mathrm{nc}}  = \fib\left(\N^{\geq
    i}\Prism_{R/A}^{(1)}\{i\}\xrightarrow{\can-c'\varphi'}\Prism_{R/A}^{(1)}\{i\}\right).$$
The clumsy notation is only used temporarily, since this construction turns out to be nothing new.
 \end{construction}

\begin{proposition}\label{prop_uncompleted}
    The canonical map
    \[
    \bZ_p(i)(R/A)^{\mathrm{nc}}  \to  \bZ_p(i)(R/A)
    \]
    is an equivalence for any $\delta$-pair where $A$ is bounded and $R$ is an animated commutative
    $A$-algebra.
\end{proposition}

\begin{proof}
    We can write any abstract equalizer $\fib(A \xto{f-g} B)$ as a pullback of the diagram
    \[
    A \xto{(f,g)} B \times B \xleftarrow{\Delta} B .
    \]
    Now consider the following commutative diagram
    \[
    \xymatrix{
        \N^{\geq i}\Prism_{R/A}^{(1)}\{i\} \ar@{}[dr]|{\lrcorner}\ar[rr]^-{(\can, c'\varphi')} \ar[d]&&
        \Prism_{R/A}^{(1)}\{i\} \times \Prism_{R/A}^{(1)}\{i\}\ar[d]  &&
        \Prism_{R/A}^{(1)}\{i\}  \ar[ll]_-\Delta\ar@{=}[d] \\
        \N^{\geq i}\Prismhat_{R/A}^{(1)}\{i\} \ar[rr]^-{(\can, c'\varphi)} \ar@{=}[d]&&
        \Prismhat_{R/A}^{(1)}\{i\}
        \times \Prism_{R/A}^{(1)}\{i\} \ar[d]\ar@{}[dr]|{\lrcorner} && \Prism_{R/A}^{(1)}\{i\}
        \ar[ll]_-{\Delta'}\ar[d] \\
        \N^{\geq i}\Prismhat_{R/A}^{(1)}\{i\} \ar[rr]^-{(\can, c\varphi)} && \Prismhat_{R/A}^{(1)}\{i\}
        \times \Prismhat_{R/A}^{(1)}\{i\}  && \Prismhat_{R/A}^{(1)}\{i\}.  \ar[ll]_-\Delta
    }
    \]
    The upper left square is cartesian as the vertical fibers are are both given by the fiber of
    the completion map $\Prism_{R/A}^{(1)}\{i\}\rightarrow\Prism_{R/A}^{(1)}\{i\}$. The same is
    true for the bottom right square.
    The pullback of the top and bottom horizontal row are $\bZ_p(i)(R/A)^{\mathrm{nc}}$ and
    $\bZ_p(i)(R/A)$, respectively. We claim that in the diagram the maps between the pullback of all rows are in
    fact equivalences. For the map from the pullback of the first row to the second this follows by
    the fact that the upper left square is a pullback and the right upper vertical map is an
    equivalence. For the map from the pullback of the middle row to the pullback of the last row
    this follows similarly since the left lower vertical map is an equivalence and the right lower
    square a pullback.
\end{proof}

\section{The prismatic package}\label{sec:package}

Recall from the introduction the definition of the $\infty$-category $\mathcal{C}_A$ which is the natural home of the prismatic package (we will review this soon). 
The main goal of this section is to conclude from the results of the previous sections that our objects 
\[
 \underline{\,\Prism\,}_{R/A} = 
    \left(\Prism_{R/A}^{[\star]}\{\star\},\N^{\geq\star}\Prismhat_{R/A}^{(1)}\{\star\},c,\varphi\right)
    \]
indeed form objects of this category and then use this to define and study syntomic cohomology.  

First we would like to rigorously define  $\mathcal{C}_A$. It will be the category of
$\bE_\infty$-algebras in a symmetric monoidal category $\Gscr_A$ so that we would like to define
$\Gscr_A$ first.

By a filtered, graded $A$-module we mean a functor 
\[
H\colon (\bZ, \geq ) \times \bZ \to \D(A),
\]
where $(\bZ,\geq)$ denotes the poset of integers and $\bZ$ denotes the discrete set of integers. 
We shall write $H^{\geq i}\{j\}$ for the evaluation at $i$ and $j$. The category of  filtered,
graded $A$-modules admits a symmetric monoidal structure given by Day convolution with respect to the
sum operation on both $\bZ$'s. We note that this tensor product does not involve any signs or so
shifts (as sometimes used on graded/filtered things).
Such a filtered graded $A$-module  is called complete if for fixed $j$ all the filtered objects
$H^{\geq i}\{j\}$ are complete, i.e. $\underleftarrow{\lim}_{i}H^{\geq i}\{j\} = 0$. The tensor
product descends to a tensor product on complete things.

There are two operations that we can perform on a filtered, graded object. The first is the \emph{shearing}, which changes the filtration degree by subtracting the grading degree:
\[
(\mathrm{sh}H)^{\geq i}\{j\} = H^{\geq i-j}\{j\}
\]
This is indeed a lax symmetric monoidal endofunctor of the $\infty$-category of filtered, graded objects since it is induced from the symmetric monoidal endofunctor of $(\bZ, \geq ) \times \bZ$ given on objects by sending $(i,j) \mapsto (i - j, j)$. This clearly is symmetric monoidal since the function is linear (the tensor product is induced by the additon on $\bZ$).

The second construction sends a filtered, graded object $N$ to a new filtration indexed on $(\bN, \geq ) \times \bZ$ given as 
\[
\mathrm{scp}(N) = (N^{\geq p \cdot \star}\{j\} \otimes p^\star \bZ) 
\]
explained in Construction \ref{filtration} above.
Again this is a lax symmetric monoidal operation since $p^\star \bZ$ is an algebra. 
The filtration $\mathrm{scp}(N)$ has the property that if we take the levelwise modulo $p$-reduction
we get the convolution of the scaled filtration
$
N^{\geq p \cdot \star }/p
$
with the filtration on $\bF_p$ given by 
$\ldots \to \bF_p \xto{0} \bF_p \xto{0} \bF_p$.
This filtration has a multiplicative map to the filtration
$\ldots \to 0 \to 0 \xto{0} \bF_p$
which is the Day convolution tensor unit over $\bF_p$. Thus we conclude that there is a canonical map of filtrations
\[
\mathrm{scp}(N) / p \to N^{\geq p \cdot \star } / p \ .
\]
Also note that if $N$ is complete and $p$-complete, then so is $\mathrm{scp}(N)$.

\begin{definition}
We define the $\infty$-category $\Gscr_A$ to be the $\infty$-category consisting of quadruples $(H,N,c,\varphi)$ where
$H$ and $N$ are $p$-complete and complete filtered graded $A$-modules, the filtration on $N$ is
    constant in non-positive degrees (equivalently it is $(\bN, \geq ) \times \bZ$ filtered), the
    map $c\colon H \to \mathrm{scp}(N) $ is a $\varphi_A$-semilinear map
of $(\bN, \geq ) \times \bZ$-filtered, graded $A$-modules and $\varphi$ is a map of graded filtered objects $N \to \mathrm{sh}(H)$ over $A$.

As an $\infty$-operad, the category $\Gscr_A$ is given as the pullback
\[
\xymatrix{
&&& \widehat{\Fun}((\bN, \geq ) \times \bZ,\D(A))^{\Delta^1 \amalg \Delta^1}  \ar[d]^{\mathrm{ev}} \\
\widehat{\Fun}\left((\bZ, \geq ) \times \bZ,\D(A)\right) \times  \widehat{\Fun}((\bN, \geq ) \times
    \bZ,\D(A))\ar[rrr]^-{(\pi_1, \mathrm{sh} \circ \pi_2, \pi_2, \varphi^*\circ \mathrm{scp} \circ
    \pi_1)}) &&& \widehat{\Fun}((\bN, \geq ) \times \bZ,\D(A))^4,
}
\]
where each vertex is viewed as an $\infty$-operad using the Day convolution symmetric monoidal
structures on $\Fun$ and the pointwise monoidal structures on $(-)^{\Delta^1 \amalg \Delta^1}$
and $(-)^4$,
where the notation $\widehat{\Fun}$ indicates that we consider complete and $p$-complete filtrations,
and where the map $\mathrm{ev}$ is evaluation at the four vertices of $\Delta^1 \amalg \Delta^1$.
     
The $\infty$-category $\mathcal{C}_A$ is defined as the category of $\bE_\infty$-algebras in $\Gscr_A$.
Note that since $\mathcal{C}_A$ is an $\infty$-category of $\bE_\infty$-algebras in a pullback, this $\infty$-category is also the
pullback of the respective categories of $\bE_\infty$-algebras,
which are given by $\infty$-categories of lax symmetric monoidal functors (resp. arrows in lax symmetric monoidal
functors).
\end{definition}

\begin{remark}\label{rem_module}\leavevmode
\begin{enumerate}
    \item[(i)]
    The category $\Gscr_A$ is a variant of the category of prismatic $F$-gauges of Bhatt--Lurie, except
        that it is not a category of quasi-coherent sheaves on a variant of $\WCart$ but rather the
        global sections of such sheaves. We view this category $\mathcal{C}_A$ simply as a tool to
        formalize the structure present on prismatic cohomology and streamline some of the coming
        proofs.
    \item[(ii)]
        For any $\delta$-ring $A$ we have an object $\underline{A} \in \Gscr_\bZ$ given by setting 
        \[
        N = H = \left( \cdots\to 0 \to 0 \to A \xto{\id} A \xto{\id} \cdots \right),
        \]
        which we view as being the trivial filtration on $A$ concentrated in graded weight $0$.
        In this case, $c=\varphi$ is the canonical map from the trivial filtration on $A$ to
        $A\otimes p^\star\bZ_p$ and $\varphi = \id$. (The reason for the surprising interchange is that we
        think of the $N$ as the Frobenius twist of $H$, but that identifies with $A$.) Then an
        $\underline{A}$-module internal to $\Gscr_\bZ$ is the same as an object in $\Gscr_A$. Thus we might as
        well work with the absolute category. 
    \item[(iii)] {The assignment $A\mapsto\Cscr_A$ defines an fpqc stack (in stable
        $\infty$-categories) on the opposite of the
        category of $\delta$-rings and $\delta$-ring maps. Thanks to the pullback definition of
        $\Cscr_A$ this boils down to checking that each vertex in the pullback is an fpqc stack
        which in turn follows from the fact that $A\mapsto\widehat{\Fun}((\bZ,\geq),\D(A))$ defines
        an fpqc stack as $\widehat{\Fun}((\bZ,\geq),\D(A))\we\Mod_{\bD_-}(\Gr\D(A))$ where
        $\bD_-$ is the graded ring $\bZ\oplus\bZ[-1](1)$. See for example~\cite[Thm.~3.2.14]{raksit}.}
\end{enumerate}
\end{remark}

\begin{proposition}
The $\infty$-categories $\mathcal{C}_A$ and $\mathcal{G}_A$ have all colimits and taking the associated graded of $H$ and $N$
preserves colimits. 
\end{proposition}
\begin{proof}
All functors on the pullback description of $\mathcal{G}_A$ preserve colimits. Thus colimits in $\mathcal{G}_A$ are computed underlying in 
$\widehat{\Fun}\left((\bZ, \geq ) \times \bZ,\D(A)\right) \times  \widehat{\Fun}((\bN, \geq ) \times
    \bZ,\D(A))$. In the category of complete filtrations colimits are computed by taking the colimit
    in filtrations and then completing. These induce colimits on associated graded pieces and can in fact
    be determined there.

The case of algebras either follows the same way (it is also a pullback of the categories of
    algebras) or simply using that algebras have colimits if the underlying category has operadic
    colimits.
\end{proof}

\begin{definition}
    The syntomic complex
    $\bZ_p(i)(H,N, c, \varphi)$ for an object of our category $(H,N, c, \varphi)$ is defined as 
    $$
    \fib\left(N^{\geq
        i}\{i\}\xrightarrow{\can-c\varphi}N^{\geq 0}\{i\}\right) \in  \D(\bZ)^\wedge_p
    $$
    Similarly, we define $\bF_p(i)$ as the mod $p$ reduction $\bZ_p(i) / p$.
\end{definition}

\begin{lemma}
Fix an integer $i$. If $j$ satisfies $(p-1) j > pi$,
then for every object of $\mathcal{C}_A$
the morphism 
\[
N^{\geq
    j}\{i\}/p\xrightarrow{\can-c\varphi}N^{\geq j-i}\{i\}/p
\]
lifts canonically and functorially as
\[
N^{\geq
    j}\{i\}/p\xrightarrow{\simeq}N^{\geq
    j}\{i\}/p \xto{\can} \N^{\geq j-i}\{i\}/p .
\]
\end{lemma}
\begin{proof}
We use that mod $p$ we have a morphism $c\colon H^{\geq \star}/p \to N^{\geq p\star}/p$ of filtrations. Thus in particular 
a map
\[
l\colon H^{\geq j-i} \xto{c} N^{\geq p (j-i)} \xto{\can} N^{\geq j} 
\]
since $p (j-i) > j$ by assumption. Thus the map $c \phi\colon N^{\geq j} \to  N^{\geq j - i}$ in
    fact lifts through a map $N^{\geq j} \to  N^{\geq j}$ and thus the difference $\can - c
    \phi$ lifts through the map $\id - l$. Now, we claim that $\id - l$ is in fact an equivalence.
    This follows since $l$ lifts one filtration bit further and both $\id$ and $l$ are
    filtered maps (applying the construction for all $j$). Since the filtration on $N^{\geq j}$ is
    complete we can test equivalences on graded pieces, where $\id$ is obviously an isomorphism and $l$ is
    zero, since it lifts further.
\end{proof}

\begin{proposition}\label{prop:syntomic_bound}
    Fix an integer $i$. If $j$ satisfies $(p-1) j \geq i$,
    then we have an equivalence
\[
\bF_p(i)(H,N, \varphi, c) \simeq \fib\left( N^{\geq i} \{i\} / N^{\geq j}\{i\} /p \xto{\can-c\varphi} N^{\geq 0} \{i\} / N^{\geq j}\{i\} /p  \right) .
\]
In other words: we may quotient out a high enough Nygaard filtered pieces to compute syntomic cohomology. 
\end{proposition}
\begin{proof}
Consider the diagram
\[
\xymatrix{
N^{\geq
    j}\{i\}/p\ar[rr]^\simeq \ar[d]^\can && N^{\geq j-i}\{i\}/p \ar[d]^\can \\
N^{\geq
    i}\{i\}/p \ar[rr]^{\can-c\varphi} && N^{\geq 0}\{i\}/p
}
\]
which exist by the previous lemma. We take the vertical cofiber.
    Now the claim is that this doesn't change the horizontal fibers, which is obvious since the upper horizontal fiber is zero.
\end{proof}

\begin{corollary}\label{cor_sifted}
The functors $\bZ_p(i), \bF_p(i)\colon \Gscr_A \to \D(\bZ)^\wedge_p$ preserve all  colimits and all limits. 
\end{corollary}
\begin{proof}
    For a fixed $i$, fix a $j$ as in the statement of Proposition~\ref{prop:syntomic_bound}.
    Then, $\bF_p(i)$ is a fiber of two functors which evidently preserve all
    colimits since they only depend on finitely many graded pieces. Thus, $\bF_p(i)$
    preserves all finite colimits. To check the statement for $\bZ_p(i)$ we can reduce mod $p$ since
    we land in $p$-complete complexes and thus to the statement about $\bF_p(i)$. The preservation
    of limits is clear.
\end{proof}

\begin{theorem}\label{thm:kan_extended_package}
The object 
\[
 \underline{\,\Prism\,}_{R/A} = 
    \left(\Prism_{R/A}^{[\star]}\{\star\},\N^{\geq\star}\Prismhat_{R/A}^{(1)}\{\star\},c,\varphi\right)
    \]
 is an object of $\Cscr_A$. In fact it refines to a functor 
 $
 \Pairs^\delta \to \mathcal{C}_\bZ \ ,
 $
 which is a algebra over the functor $(A,R) \mapsto \underline{A}$ (see Remark \ref{rem_module}). As such it is left Kan extended from free $\delta$-pairs. \end{theorem}
\begin{proof}
The object $\Prism_{R/A}^{[\star]}\{\star\}$ is constructed in Section \ref{sec:crystal} as a
    functor to complete filtered objects. It is shown in Corollary \ref{cor:kanextended} that it is
    Kan extended as such (note that this Kan extension only involves filtered colimits, so we do not
    have to worry about the distinction whether we Kan extend with or without algebra structure).
    The fact that it is a graded, filtered $\bE_\infty$-algebra follows from the combination of two
    facts. First the fact that the crystal $\Hscr^{[\star]}_\Prism(R/A)\{\star\}$ is a graded
    filtered algebra, which follows from the fact that the filtration is multiplicative (since it is
    $I$-adic on the crystal) and that we tensor with a tensor algebra over an invertible, which is
    also multiplicative. Secondly the functor taking global sections of prismatic crystals is a lax
    symmetric monoidal functor.

    The object $\N^{\geq\star}\Prismhat_{R/A}^{(1)}\{\star\}$ is produced in Definition \ref{Nygaard}
    and it is multiplicative as it is a pullback of multipliciative filtrations.
    This term is Kan extended by Corollary
    \ref{cor:nygaardproperties}.

    The maps $\varphi$ and $c$ are produced at the beginning of Section~\ref{sec:syntomic} and are clearly multiplicative
    (the filtration by definition).
Finally we note that the description of colimits in $\mathcal{C}_\bZ$ then implies the Kan extension statement and we are done. 
\end{proof}

Now, we can prove Corollary~\ref{cor:intro_kan}.

\begin{corollary}\label{cor:syntomickanextension}
    For each $i\in\bZ$,
    syntomic cohomology as a functor 
    \[
        \bZ_p(i)(-/-)\colon\Pairs^\delta \to \D(\bZ)^\wedge_p
    \]
    is left Kan extended from finitely presented free $\delta$-pairs.
\end{corollary}
\begin{proof}
We consider the prismatic package 
\[
 \underline{\,\Prism\,}_{R/A} = 
    \left(\Prism_{R/A}^{[\star]}\{\star\},\N^{\geq\star}\Prismhat_{R/A}^{(1)}\{\star\},c,\varphi\right)  \in \mathcal{C}_A
\]
    as an object of $\mathcal{C}_\bZ$ by forgetting the $A$-algebra structure.
    As such a functor, it is left Kan extended by Theorem~\ref{thm:kan_extended_package},
    so the result follows by Corollary \ref{cor_sifted} and the fact that syntomic cohomology on
    $\Cscr_A$ factors through the forgetful functor $\Cscr_A\rightarrow\Cscr_{\bZ}$.
\end{proof}

\begin{corollary}
    For $i<0$, $\bZ_p(i)(R/A)\we 0$.
\end{corollary}

\begin{proof}
    By Corollary~\ref{cor:syntomickanextension}, it is enough to prove the corollary when $(A,R)$ is
    a finitely presented free $\delta$-pair and then we can use Corollary~\ref{cor:syntomicdescent}
    to reduce to the bounded prismatic case and even to the bounded oriented prismatic case, so assume that $(A,(d))$ is an
    oriented bounded prism. Using Proposition~\ref{prop:nygaardcomparison}, we know our
    construction agrees with the Nygaard filtration on relative prismatic cohomology introduced
    in~\cite{prisms}. For $i<0$, the Frobenius map
    $c\varphi\colon\Prismhat_{R/A}^{(1)}\{i\}\rightarrow\Prismhat_{R/A}^{(1)}\{i\}$
    is divisible by $d^{-i}$. It follows that, for
    $i<0$, $c\varphi$ is topologically nilpotent, so that $\can-c\varphi$ is an equivalence.
\end{proof}

\section{Computation using descent}\label{sec:descent}

Recall that $\Prism_{R/A}$ has descent in $A$ and $R$ (Proposition \ref{prop:cwdescent}) and is
invariant under relatively perfect maps and completion in $A$ (Proposition \ref{prop:quasietalecrystal} and Corollary~\ref{cor:kanextended}).
Knowing this, we can give a descent style description of $\Prism_{R/A}$ using a relative version of the quasisyntomic site, that we will describe now.
In fact one could also use this to give an independent definition of the prismatic package,
specifically the Nygaard filtration. This is not only a theoretical tool, but will be applied for practical calculations. 
After this discussion, we prove the uniqueness claim from Theorem~\ref{thm:intro}.

\begin{definition}[Relatively quasisyntomic rings]\label{def:relqsyn}
    A bounded $\delta$-pair $(A,R)$ is relatively quasisyntomic if $R$ is derived $p$-complete and $\L_{R/A}$ has $p$-complete
    Tor-amplitude in $[0,1]$. If $A$ is a bounded $\delta$-ring, let $\QSynscr_A$ be the category of
    relatively quasisyntomic $\delta$-pairs $(A,R)$.
\end{definition}

The notation
$\QSynscr_A$ diverges from that of~\cite{ammn}, which would require also that $A$ and $R$ be
quasisyntomic in the sense of~\cite{bms2}. However, $\QSynscr_{\bZ_p}$ does agree with the
category by the same name in~\cite{ammn} and also with $\QSyn$ from~\cite{bms2}.

\begin{definition}[Relatively quasiregular semiperfectoids]\label{rel_quasiperf}
    Say that a bounded $\delta$-pair $(A,R)$ is relatively quasiregular semiperfectoid (or
    relatively qrsp) if $R$ is $p$-complete, there
    is a factorization $A\rightarrow A'\rightarrow R$ where $A\rightarrow A'$ is a $p$-completely relatively perfect
    map of $\delta$-rings, the $\delta$-pair $(A',R)$ is pre-prismatic as exhibited by an ideal $I \subseteq A'$,
    the map $A' \to R$ is surjective, and
    $\L_{R/\overline{A'}}$ has $p$-complete Tor-amplitude in $[1,1]$.
    Call such a factorization a pre-prismatic factorization.
    Note that a relatively quasiregular semiperfectoid is relatively quasisyntomic in the sense of
    Definition~\ref{def:relqsyn}.
    Let $\RQRSPerfdscr \subseteq \Pairs^\delta$ be the full subcategory consisting of the
    relatively quasiregular semiperfectoid $\delta$-pairs. 
\end{definition}

Note that we can replace $A'$ by $A'[\delta(I)^{-1}]_{(p,I)}^\wedge$, so that we can assume in
fact that $A'$ is a prism, with $\Prism_{R/A} = \Prism_{R/A'}$.

\begin{example}
    Let $A$ be a bounded $\delta$-ring.
    If $r_0,r_1,r_2,...$ forms a regular sequence in $A$ and if $A \to R = A/(r_0,r_1,r_2,...)$ is
    pre-prismatic, then it is relatively quasiregular semiperfectoid with $A' = A$.
\end{example}

\begin{example}
Assume that $R$ is a quasiregular semiperfectoid ring in the sense of \cite[Definition 4.20]{bms2}. Then we claim that 
$\bZ_p \to R$ is relatively quasiregular semiperfectoid in the sense of Definition~\ref{rel_quasiperf}.
    To see this we pick a perfectoid ring $R'$ such that $R = R' / I$ and $L_{R/R'}$ has Tor
    amplitude in $[1,1]$, see \cite[Lemma 4.25]{bms2}.  Then $A' = A_\mathrm{inf}(R)$ exhibits that
    $\bZ_p \to R$ is relatively qrsp.
Assume conversely that $\bZ_p \to R$ is relatively quasiregular semiperfectoid in the sense of
    Definition \ref{rel_quasiperf}. Then, the prism in the definition has to be perfect so that
    $A'/I$ is perfectoid and $R$ a quotient thereof, which implies that $R$ is qrsp in the sense of
    \cite[Definition 4.20]{bms2}.
\end{example}

\begin{example}
    Assume that $(A,I)$ is a bounded prism. Recall from~\cite[Definition 15.1]{prisms} that a
   $A/I$-algebra $S$
    is called large if it is flat, the cotangent complex has Tor-amplitude in $[1,1]$ and  there is a surjection
\[
A/I[X_i^{1/p^\infty}]^\wedge_p \to S
\]
for some set $i \in I$. In this case we can consider $A ' = A[X_i^{1/p^\infty}]^\wedge_{(p,I)}$
    which is relatively perfect over $A$. Moreover we have an induced surjective map $A' \to S$, for
    which $L_{S/A'}$ has Tor-amplitude in $[1,1]$.
\end{example}

The significance of relatively qrsp pairs is that prismatic cohomology can be understood very explicitly. 
\begin{theorem}[Bhatt--Scholze]\label{thm_BS}
Let $A \to R$ be relatively qrsp. Then the prismatic cohomologies $\Prism_{R/A}^{[\star]}\{\star\}$ and $\N^{\geq\star}\Prismhat_{R/A}^{(1)}\{\star\}$
including all the filtered pieces and all the graded pieces are concentrated in degree $0$ and the
    relative Frobenius map $\N^{\geq\star}\Prismhat_{R/A}^{(1)}\{\star\} \to
    \Prism_{R/A}^{[\star]}\{\star\}$ is injective.

If, more specifically, $R = A'/(I,r_1,r_2,...)$ for some prism $(A',I)$ and a regular sequence
    $r_1,r_2,... \in A'/I$, then we have 
\[
\Prism_{R / A} \simeq A'\left\{\frac{r_1}  I, \frac{r_2} I, ... \right\}^\wedge_{(p,I)} \qquad \text{and} \qquad
\Prismhat^{(1)}_{R / A} \simeq A'\left\{\frac{\varphi(r_1)}  {\varphi(I)},\frac{\varphi(r_2)}
    {\varphi(I)}, ... \right\}^\wedge_{(p,N)},
\]
    where $A'\left\{\frac{r_1}  I, \frac{r_2} I, ... \right\}^\wedge_{(p,I)}$ is the
    $(p,I)$-completed prismatic envelope construction of~\cite[Prop.~3.13]{prisms} and 
    $A'\left\{\frac{\varphi(r_1)}  {\varphi(I)},\frac{\varphi(r_2)}
    {\varphi(I)}, ... \right\}^\wedge_{(p,N)}$ is the closure in $A'\left\{\frac{r_1}  I, \frac{r_2} I, ...
    \right\}^\wedge_{(p,I)}$ of the sub-$\delta$-ring generated by
    $\frac{\varphi(r_1)}{\varphi(I)},\frac{\varphi(r_2)}{\varphi(I)},\ldots$ with respect to the
    $(p,I)$-adic topology. Here the subscript $(p,N)$ indicates that we think of this as the completion with respective to the the $p$-adic and Nygaard filtration, described below.

    The Hodge-Tate filtration on $\Prism_{R / A}$ is given by the $I$-adic filtration and the
    Nygaard filtration on $\Prismhat^{(1)}_{R / A}$ is given by the pullback of the
    $I$-adic filtration on $  A'\left\{\frac{r_1}  I, \frac{r_2} I, ... \right\}^\wedge_{(p,I)}$ along the inclusion
    \[
        A'\left\{\frac{\varphi(r_1)}  {\varphi(I)},\frac{\varphi(r_2)} {\varphi(I)}, ... \right\}^\wedge_{(p,N)}   \subseteq 
        A'\left\{\frac{r_1}  I, \frac{r_2} I, ... \right\}^\wedge_{(p,I)}.
    \]
    (In particular, $(p,N)$-convergent sequences are exactly the ones that converge $(p,I)$-adically in
    the ambient space of the inclusion $\Prismhat^{(1)}_{R/A} \subseteq \Prism_{R/A}$.)
    The map $c\colon \Prism_{R / A} \to \Prismhat^{(1)}_{R / A}$ is given by applying $\varphi$ and the
    relative Frobenius map $\Prismhat^{(1)}_{R / A} \to \Prism_{R / A}$ is given by the inclusion. The
    Breuil--Kisin twisted versions have the analogous form with $A'$ replaced by $A'\{i\}$.
\end{theorem}
\begin{proof}
    Using that $\underline{\,\Prism\,}_{R/A} \simeq \underline{\,\Prism\,}_{R/A'}$ this reduces to the work
    of Bhatt--Scholze \cite{prisms} where it is shown in Sections 12 and 15. 
\end{proof}

\begin{definition}[Quasismooth maps]
    Say that a map of $\delta$-rings $A \to A'$ is quasismooth if it is $p$-completely flat and $\L_{A'/A}$
    has $p$-complete Tor-amplitude in $[0,0]$. Such a map is a quasismooth cover if it is
    additionally $p$-completely faithfully flat. 
\end{definition}

\begin{remark}\label{remark_qs}
Note that if we have a factorization $A \to A' \to R$ with $A \to A'$ quasismooth and 
    $R\in\QSynscr_A$
 then the transitivity triangle 
 \[
\L_{A' / A} \otimes_{A'} R \to \L_{R / A} \to \L_{R/A'}
\]
implies that $R\in\QSynscr_{A'}$.
\end{remark}

\begin{lemma}
Assume that $(A, R)$ is a relatively quasisyntomic $\delta$-pair. Then there is a factorization $A
    \to A' \to R$ with $A \to A'$ a quasismooth cover and $(A',R)$ relatively qrsp. In fact we can
    choose $A'$ in such a way that $(A',R)$ is pre-prismatic, exhibited by an ideal $I$, such that  $A'/I \to R$ is a
    surjection whose cotangent complex has Tor-amplitude in $[1,1]$.
\end{lemma}
\begin{proof}
Choose a surjection $\kappa\colon A[\{z_i\}_{i\in I}]^\wedge_p \to R$ extending $A \to R$. Now set
\[
A' = A[\{z_i\}_{i\in I},z]^\wedge_p
\]
with the $\delta$-ring structure with $\delta(z_i) = \delta(z) = 0$.
    Moreover $A'$ is a quasismooth $A$ algebra. Now consider the map 
$A' \to R$
extending $\kappa$ with $z \mapsto p$. Then, the kernel contains $I = (z-p)$ which exhibits the
$\delta$-pair $(A',R)$ as pre-prismatic. Moreover the map $\overline{A'} = A'/I \to R$ is surjective and 
the transitivity triangle for $A \to \overline{A'} \to R$ takes the form
\[
\L_{\overline{A' }/ A} \otimes_{\overline{A'}} R \to \L_{R / A} \to \L_{R/\overline{A'}}
\]
The first term has Tor-amplitude in degree $0$.
    The second term has, by quasisyntomicity of $(A,R)$, Tor-amplitude in $[0,1]$. Moreover since
    $\overline{A'} \to R$ is surjective, the first map is surjective on $\pi_0$. It follows that
    $\L_{R/\overline{A'}}$ has Tor-amplitude in $[1,1]$.
\end{proof}
%
We note that the converse also holds: whenever we have a bounded $p$-complete $\delta$-pair $(A,R)$
that admits a quasismooth cover by a relatively qrsp $(A',R)$, the pair $(A,R)$ is relatively
quasisyntomic.

\begin{lemma}
For any factorization $A \to A^0 \to R$ as in the last lemma (i.e.,  $A \to A^0$ is a quasismooth
    cover and $(A^0,R)$ is relatively qrsp), consider the \v{C}ech diagram $A^\bullet$ for $A \to
    A^0$ with $A^n \we A^0 \otimes_A \ldots \otimes_A A^0)$. All the terms $A^n \to R$ are relatively qrsp.
\end{lemma}
\begin{proof}
    Let us first assume that $(A^0,R)$ is already a prismatic $\delta$-pair.
    Then all the terms $A^n$ are also prisms using the ideals given as base changes along any of the flat maps
    $A^0 \to A^n$. Moreover consider the transitivity triangles
    \[
    \L_{A^n / A^0} \otimes_{A^n} R \to \L_{R / A^0} \to \L_{R/A^n} \ .
    \]
    where the first term has Tor-amplitude in $[0,0]$ and the second in $[1,1]$.
    It follows that $\L_{R/A^n}$ has Tor-amplitude in $[1,1]$. 
    
    In the general case, suppose that $A^0\rightarrow B\rightarrow R$ is a factorization where
    $A^0\rightarrow B$ is $p$-completely relatively perfect, $B\rightarrow R$ is a surjective pre-prismatic
    $\delta$-pair, and $\L_{R/\overline{B}}$ has Tor-amplitude in $[1,1]$.
    Then, $A^n\rightarrow R$ factors as $A^n\rightarrow B\otimes_{A^0} A^n\rightarrow R$, which is a
    factorization exhibiting $R$ as relatively qrsp over $A^n$, as desired.
\end{proof}

For a fixed bounded $R$, consider the category $\QSynscr_{/R}$ of relatively quasisyntomic bounded $\delta$-pairs $(A,R)$.
We have the full subcategory 
\[
\RQRSPerfdscr_{/R} \subseteq \QSynscr_{/R}
\]
and consider the opposite of both categories as sites with the quasismooth topologies.
We leave it as an exercise to see that these are indeed sites:
the key is to verify that for a quasismooth cover $A \to A'$ and a map $A \to B$ over $R$
the map $A' \otimes_A B \to R$ is relatively quasisyntomic.
This follows using that $B \to A' \otimes_A B$ is a quasismooth cover using Remark \ref{remark_qs}. 

\begin{proposition}[Unfolding in the base]\label{lem:unfold_base}
    For each fixed $R$, the inclusion induces an equivalence 
    \[
    \mathrm{Shv}_\Cscr(\QSynscr_{/R}^\op)\we\mathrm{Shv}_\Cscr(\RQRSPerfdscr_{/R}^\op)
    \]
    of
    $\infty$-categories of $\Cscr$-valued sheaves for any $\infty$-category $\Cscr$.
\end{proposition}

\begin{proof}
    This follows from the previous lemmas by an argument similar to the proof of~\cite[Prop.~4.31]{bms2}.
\end{proof}

\begin{definition}[Quasisyntomic maps]
    Say that a map $R\rightarrow R'$ of commutative rings is quasisyntomic if it is $p$-completely flat and $\L_{R'/R}$
    has $p$-complete Tor-amplitude in $[0,1]$. Such a map is a quasisyntomic cover if it is
    additionally $p$-completely faithfully flat. Note that if $A$ is a bounded $\delta$-ring,
    $R\in\QSynscr_A$, and $R\rightarrow R'$ is quasisyntomic, then $R'\in\QSynscr_A$.
\end{definition}

\begin{lemma}
    A bounded $\delta$-pair $(A,R)$ with $R$ derived $p$-complete is relatively quasisyntomic if and
    only if there
    is a quasisyntomic cover $R \to R'$ with $R' \in \RQRSPerfdscr_A$.
\end{lemma}
\begin{proof}
Assume that there exists a quasisyntomic cover $R \to R'$ with $R' \in \RQRSPerfdscr_A$. Then, there
    is a transitivity triangle 
\[
\L_{R/A} \otimes_R R' \to \L_{R' / A} \to \L_{R'/R}
\] 
which shows that $\L_{R/A} \otimes_R R'$ has Tor-amplitude in $[-1,1]$, which implies that $L_{R/A}$
    has Tor-amplitude in $[-1,1]$ by faithful flatness. Since $L_{R/A}$ is connective it follows
    that it has Tor-amplitude in $[0,1]$, i.e. $(A,R)$ is relatively quasisyntomic.

Assume conversely that $(A,R)$ is relatively quasisyntomic.
Choose a set of $p$-complete generators $\{r_i\}_{i \in I}$ of $R$ as an $A$-algebra and consider the cover 
\[
R \to  R[ p^{1/p^\infty}, \{r_i^{1/p^\infty}\}_{i \in I}]_p^\wedge = R'
\] 
This cover is quasisyntomic. We claim moreover that $(A,  R')$ is relatively qrsp. To see this we consider 
\[
A' = A[\{z_i^{1/p^\infty}\}_{i \in I},z^{1/p^\infty}]_p^\wedge
\]
 with the $\delta$-ring structure with $\delta(z_i^{1/p^k}) = 0$ and $\delta(z^{1/p^k}) = 0$.
    The map $A \to A'$ of $\delta$-rings is $p$-completely relatively perfect. Moreover, the map
   \[
    A' \to R'
   \]
   which sends $z^{1/p^k}$ to $p^{1/p^k}$ and $z_i^{1/p^k}$ to $r_i^{1/p^k}$ 
    is surjective by construction and has the ideal $(z - p)$ in its kernel, so $(A',R)$ is
    pre-prismatic.
\end{proof}

Similarly to the unfolding in the base we would also like to see that one can use descent in the ring to compute prismatic cohomology. 
For a fixed $\delta$-ring $A$ let $\RQRSPerfdscr_A$ be the full subcategory of  $\QSynscr_A$
consisting of the qrsp $A$-algebras.
We claim that for fixed $A$ 
    the categories $\QSynscr_A^\op$ and $\RQRSPerfdscr_A^\op$ form sites with respect to the
    quasisyntomic topologies.To see this, it is enough to see that if $S'$ is the
    $p$-completed pushout of a diagram $R'\leftarrow R\rightarrow S$ of maps in $\QSynscr_A$ (resp.
    $\RQRSPerfdscr_A$) where $R \to R'$ is a quasisyntomic cover, then $S'\in\QSynscr_A$ (resp.
    $\RQRSPerfdscr_A$). This is left to the reader
    as an exercise with the conormal sequence for the cotangent complex. 

\begin{proposition}[Unfolding in the ring]\label{lem:unfold_ring}
The
    inclusion $\RQRSPerfdscr_A^\op\hookrightarrow\QSynscr_A^\op$ induces an equivalence 
    $\mathrm{Shv}_\Cscr(\QSynscr_A^\op)\we\mathrm{Shv}_\Cscr(\RQRSPerfdscr_A^\op)$ of
    $\infty$-categories of $\Cscr$-valued sheaves for any presentable $\infty$-category $\Cscr$.
\end{proposition}

\begin{proof}
    If $R\in\QSynscr_A$ and $f\colon R\rightarrow S$ is a quasisyntomic map with
    $S\in\RQRSPerfdscr_A$, then every $S^n$ in the \v{C}ech complex of $f$ is in $\RQRSPerfdscr_A$
    as well. Indeed, $\L_{S/A}$ has $p$-complete Tor-amplitude in $[1,1]$, so $\L_{S/R}$ has in fact
    $p$-complete Tor-amplitude in $[1,1]$ as well. Thus, every $\L_{S^n/S^{n-1}}$ (via any choice of
    maps in the cosimplicial diagram) has $p$-complete Tor-amplitude in $[1,1]$, so
    $\L_{S^n/\overline{A'}}$ has $p$-complete Tor-amplitude in $[1,1]$ for any choice of prismatic
    factorization of $S=S^0$.
    Now, follow the proof of~\cite[Prop.~4.31]{bms2} using the previous lemma to show that every object is locally relatively quasiregular
    semiperfectoid.
\end{proof}

Now we want to explain why and how Theorem \ref{thm:intro} uniquely determines relative prismatic
cohomology and how to compute it in practice. This will also show the advantage of the relative
approach since we will use descent in $A$ which would not be possible otherwise.  
We would like to argue that the  assignment
\[
(R,A) \mapsto \underline{\,\Prism\,}_{R/A} \in \mathcal{C}_A
\]
is uniquely determined by properties (1),  (2), (4), (6) and (8) (or alternatively by (1), (2),
(4), (5) and (8)).  In fact a much weaker axiom (\$) replacing (1), (2), and (8) is sufficient to
determine the whole theory in the presence of (4) and one of (5) or (6):

\begin{enumerate}
    \item[(\$)] If $(A,R)$ is relatively qrsp, with prism $(A',I)$, then
            $\underline{\,\Prism\,}_{R/A}$
            agrees naturally with the derived prismatic cohomology $\underline{\,\Prism\,}_{R/A'}$ of Bhatt and
            Scholze~\cite{prisms}, which in this case is discrete and the relative Frobenius map is injective.
%
\end{enumerate}
The advantage of this formulation is that it can be formulated without reference to Bhatt--Scholze's
theory, as long as one constructs the prismatic cohomology $\underline{\,\Prism\,}_{R/A'}$
explicitly, e.g. as an initial prism. Here naturality in $(A',I)$ means that we have to make the
choice of $(A',I)$ and it depends on that, since the Bhatt--Scholze theory a priori does. More
precisely there is a category $\RQRSPerfdscr'$ of relatively qrsp pairs $(A,R)$ with a choice of a
prism $(A',I)$ and maps given by maps of $\delta$-pairs that come with a compatible choice of a map
between the respective prisms. Axiom $(\$)$ asserts the existence of  a commutative diagram
\begin{equation}\label{dollar_diag}
    \begin{gathered}
\xymatrix{
\RQRSPerfdscr' \ar[r]\ar[d]^T &\Pairs^{\Prism} \ar[d]^-{\Prismpackage^\rel} \\
 \RQRSPerfdscr   \ar[r]^-{\Prismpackage} & \Cscr_\bZ,
}
    \end{gathered}
\end{equation}
where the left functor sends $(A,R)$ to $R$ and the top functor is the fully faithful inclusion in
to $\delta$-pairs.

\begin{proposition}\label{prop_unique}
  The functor $\Prismpackage$ is uniquely characterized on $\Pairs^\delta$ by (\$),  (4) and (6) or alternatively by (\$), (4) and (5).
\end{proposition}
\begin{proof}
  Using (4), the functor is uniquely determined by its restriction to quasisyntomic $\delta$-pairs since it is Kan extended from those. Now using (6) and Proposition \ref{lem:unfold_base} 
  or (5) and Proposition \ref{lem:unfold_ring} it is uniquely determined by its restriction to relatively qrsp pairs. 
Now we want to argue that diagram \eqref{dollar_diag} already uniquely determines the functor $\Prismpackage$, in other words there is at most one dashed functor
\[
\xymatrix{
\RQRSPerfdscr' \ar[r]^-{\Prismpackage^\rel}\ar[d]^T & \Cscr_\bZ \\
 \RQRSPerfdscr   \ar@{-->}[ru]& 
}
\]
As a first step we note that the upper horizontal functor $\Prismpackage^\rel$ really takes values in the full subcategory 
$\Cscr_\bZ' \subseteq \Cscr_\bZ$ given by those quadruples $(H,N,c,\varphi)$ where $H$ and $N$ are
    discrete, strict filtrations (i.e. all transition maps are injective), and $\varphi$ is
    injective. This subcategory is actually a $1$-category.  Since the left vertical functor
    $T\colon
    \RQRSPerfdscr' \to \RQRSPerfdscr$ is by definition essentially surjective this also implies that
    any lift has to take values in $\Cscr_\bZ'$ and therefore we may replace $\Cscr_\bZ$ by
    $\Cscr_\bZ'$.
Moreover, there is a forgetful functor
\[
U\colon \Cscr_\bZ' \to \widehat{\mathrm{Fil}}(\mathrm{Ab}) \qquad (H,N,c,\varphi) \mapsto H,
\]
where $\widehat{\mathrm{Fil}}(\mathrm{Ab})$ is the category of derived complete strictly filtered abelian groups, i.e.
    functors $(\bZ,\geq) \to \mathrm{Ab}$ whose transition functions are injective and whose derived
    inverse limit vanishes. The functor $U$ is faithful, which follows since maps between quadruples
    are determined on the filtration $H$ since $\varphi$ is injective, i.e. the map on $N$ is
    determined by the map on $H$. Applying the next Lemma \ref{lifting_1} to $U$ and $T$ shows that
    to prove the proposition it suffices to show that there is at most one dashed lift in the
    diagram
\begin{equation}\label{lifting_diag}
    \begin{gathered}
\xymatrix{
\RQRSPerfdscr' \ar[rd]^-{\Prism^\rel}\ar[d]^T &\\
 \RQRSPerfdscr   \ar@{-->}[r] & \widehat{\mathrm{Fil}}(\mathrm{Ab}).
}
    \end{gathered}
\end{equation}
To see this we note that for a given object $A \to A' \to R \in  \RQRSPerfdscr'$ we have that $(B,
    \overline{B}) = (\Prism^\rel_{R/A'}, \Prismbar^\rel_{R/A'})$ is a prism equipped with a natural
    map $\chi\colon A' \to B$ of prisms by \cite[Lemma 7.7]{prisms}. Moreover, this map induces an
    equivalence
\[
\Prism^\rel_{R/A} \xto{\simeq} \Prism^\rel_{\overline{B}/ B} 
\]
by definition
    since $\Prism^\rel_{\overline{B}/B}\iso B$.
We declare a class of morphisms in $ \RQRSPerfdscr'$ by
\[
W' = \left\{ (A \to A' \to R) \xto{\chi} (B \xto{\id} B \to \overline{B}) \mid A \to A' \to R \in  \RQRSPerfdscr' \right\}
\]
and let $W$ be the image of $W'$ under the functor $T:  \RQRSPerfdscr' \to  \RQRSPerfdscr$. Now the
    functor $\Prism^\rel$ in diagram \eqref{lifting_diag} sends $W'$ to equivalences. Thus any
    factorization $\RQRSPerfdscr \to  \widehat{\mathrm{Fil}}(\mathrm{Ab})$ has to send $W'$ to
    equivalences as well. Therefore we can invert the classes $W'$ and $W$ in Diagram
    \eqref{lifting_diag}. But then the uniqueness of the dashed arrow follows from the assertion
    that the induced functor
\[
\Psi\colon \RQRSPerfdscr' [(W')^{-1}] \to \RQRSPerfdscr[W^{-1}] 
\]
is an equivalence. 
To see this, we consider the functor
\[
    \RQRSPerfdscr \to \RQRSPerfdscr' \qquad (A \to R) \mapsto (\Prism_{R/A}  \xto{\id} \Prism_{R/A} \to
    \Prismbar_{R/A}).
\]
This functor depends on the choice of an extension $\Prism$ in \eqref{lifting_diag} which we fix (if
    there were no extension, then there would be nothing to show). This functor  sends $W$ to equivalences
    since the prismatic cohomology agrees. Thus it induces a functor 
\[
 \Omega\colon \RQRSPerfdscr[W^{-1}] \to \RQRSPerfdscr' [(W')^{-1}] . 
\]
We have natural isomorphisms 
\[
(A \to R) \xto{\simeq \in W} (\Prism_{R/A} \to \Prismbar_{R/A}) = \Psi \Omega (A \to R)
\]
and
\[
(A \to A' \to R) \xto{\chi} (B \to B \to \overline{B}) = \Omega\Psi(A \to A' \to R),
\]
which finishes the proof. 
\end{proof}

\begin{lemma}\label{lifting_1}
Assume that we have a commutative square of 1-categories
\[
\xymatrix{
\mathcal{R}_1 \ar[d]_T \ar[r]^{f} & \Cscr_1 \ar[d]^U \\
\mathcal{R}_2 \ar[r]^{g} \ar@{-->}[ru]& \Cscr_2
}
\]
where $T$ is essentially surjective and $U$ is faithful. Then the space of lifts $\mathcal{R}_2 \to \Cscr_1$ making the diagram commutative is empty or contractible.
\end{lemma}
\begin{proof}
  Fix two lifts $F_1, F_2$, with natural isomorphisms $\eta_i\colon U\circ F_i \to g$ and
    $\nu_i\colon F_i\circ T\to f$.
    (These are required to be compatible in the sense that their composites $g\circ T \to U\circ f$
    have to agree with the natural isomorphism provided as part of the commutative square.) For
    $x\in \Rscr_2$, we have the composite
  \[
    \begin{tikzcd}
      U(F_1(x)) \rar{\eta_{1,x}} & g(x) \rar{\eta_{2,x}^{-1}}& U(F_2(x)).
    \end{tikzcd}
  \]
  By faithfulness of $U$, there exists at most one lift of this composite along $U$. If there exists
    one for each $x$, together the resulting maps $\varepsilon_x: F_1(x)\to F_2(x)$ form a natural
    transformation again using faithfulness of $U$ and naturality of $\eta$ and $\nu$. Any natural transformation compatible with the diagram needs to be of this
    form. For existence, we pick $\widetilde{x}$ with $T(\widetilde{x}) \cong x$, and compose the
    resulting isomorphism
  \[
    U(F_i(x)) \cong U(F_i(T(\widetilde{x}))) \cong U(f(\widetilde{x}))
  \]
  for $i=1$ with the inverse of the corresponding isomorphism for $i=2$.
%
%
%
\end{proof}

\section{Filtered prismatic cohomology}\label{sec:filtered}

As a consequence of the base change property, prismatic cohomology of $\delta$-pairs inherits
gradings and filtrations, generalizing the observation from \cite[Remark
3.8]{bhatt-mathew-syntomic}. Indeed, recall that graded objects can be encoded as quasicoherent
sheaves on $\B\bG_m$, i.e., as objects with an $\Oscr_{\bG_m}=\bZ[x^{\pm 1}]$-coaction. Since
$\bZ[x^{\pm 1}]$ carries a canonical $\delta$-ring structure with $\delta(x)=0$, if
$(A,R)$ is a $\delta$-pair, we may consider $(A\otimes \Oscr_{\bG_m}, R\otimes \Oscr_{\bG_m})$ as a $\delta$-pair as
well. Then, we may define graded $\delta$-pairs as $\delta$-pairs with an $\Oscr_{\bG_m}$-coaction. By
base change, the map
\[
  \Prism_{R/A}\to \Prism_{R\otimes \Oscr_{\bG_m} / A\otimes \Oscr_{\bG_m}} \we \Prism_{R/A}\tensorhat \Oscr_{\bG_m}
\]
then induces a $\Oscr_{\bG_m}$-coaction on $\Prism_{R/A}$ (as a complete filtered object), so
$\Prism_{R/A}$ inherits a grading for a graded $\delta$-pair $(A,R)$. Analogous statements hold for
filtered $\delta$-pairs, as we will show.

\begin{convention}[Nonnegative decreasing filtrations]
    Throughout this section, all filtered commutative rings $\F^\star R$ and all filtered
    $\delta$-rings $\F^\star A$ (to be defined below) will be
    decreasingly filtered commutative rings where $\F^0R\rightarrow\F^{-m}R$ is an
    equivalence for $m\geq 0$. We set $R=\F^0R$.
\end{convention}

\begin{definition}
    A filtered abelian group $\F^\star M$ is strict if $\F^mM\rightarrow \F^{m_0}M$ is a monomorphism for
    each $m\geq m_0\in\bZ$.
\end{definition}

\begin{definition}[Filtered $\delta$-rings]
    A filtered $\delta$-ring is a pair $(\F^\star A,\delta)$ consisting of a strict filtered commutative
    ring $\F^\star A$ and a $\delta$-ring structure on $A=\F^0A$ with the property that
    $\delta(\F^mA)\subseteq\F^{pm}A$ for all $m\geq 0$.
\end{definition}

\begin{remark}
    Given a filtered $\delta$-ring $\F^\star A$, the Frobenius endomorphism $\varphi$ on $A$ also has the property that
    $\varphi(\F^mA)\subseteq\F^{pm}A$. Conversely, given a filtered commutative ring $\F^\star A$
    such that $A$ and each graded piece $\gr^mA$ is $p$-torsion free, then any
    lift of Frobenius $\varphi$ satisfying $\varphi(\F^mA)\subseteq\F^{pm}A$
    arises from a unique filtered $\delta$-ring structure in this way. This follows since $p \delta(x) \in \F^{pm}$ implies $\delta(x) \in \F^{pm}$ by the assumption on the graded pieces.
\end{remark}

\begin{remark}
    If $\F^\star A$ is a filtered $\delta$-ring, then there is an induced $\delta$-ring structure on
    $\gr^0A$. Note that this implies that no power of $p$ can be in positive filtration weight in
    $A$.
\end{remark}

\begin{definition}[Filtered prisms]
    A filtered prism consists of a filtered $\delta$-ring $\F^\star A$ together with a prism
    structure $(A,I)$ on $A=\F^0A$ such that each $\gr^iA$ is $I\otimes_A\gr^0A$-torsion free and
    $I\otimes_A\gr^0A$-complete. Note that the assumption means that $(\gr^0A,I\otimes_A\gr^0A)$
    is a prism as is the $(p,I\otimes_A\gr^0A)$-completion of $\bigoplus_{m\geq 0}\gr^mA$. A filtered prism is
    called complete if the filtered commutative ring $\F^\star A$ is derived complete with respect to the
    filtration.
\end{definition}

\begin{example}[Filtered Breuil--Kisin prisms]
    Consider the ring $A=W(k)\llbracket z\rrbracket$ equipped with the $z$-adic filtration and a prism
    structure given by $(E(z))$, where $E(z)$ is an Eisenstein polynomial. Then
    $(E(z))\otimes_A\gr^0A=(p)\subseteq W(k)$. As each graded piece $\gr^uA\iso W(k)$ is $p$-torsion
    free, this defines a (complete) filtered prism, which we call a filtered Breuil--Kisin prism.
\end{example}

\begin{definition}[Flat filtrations]
    Given a filtered ring $\F^\star A$ and an $A=\F^0A$-module $M$, the flat
    filtration on $M$ is defined to be $\F^\star M=M\otimes_A\F^\star A$.
\end{definition}

\begin{remark}\label{rem_filteredbar}
    If $\F^\star A$ is a filtered $\delta$-ring and $(A,I)$ is a prism structure, not assumed yet to
    make $A$ into a filtered prism, then
    $$0\rightarrow\F^\star I\rightarrow\F^\star A\rightarrow\F^\star\overline{A}\rightarrow 0$$
    is an exact sequence of filtered abelian groups, which follows because $A$ is $I$-torsion
    free and hence so is each $\F^i A$ by the strictness assumption in the definition of a filtered
    $\delta$-ring. The filtration $\F^\star I$ is automatically strict in this situation.
    The pair $(\F^\star A,I)$ is a filtered prism if and only if the associated graded pieces of the flat
    filtration on $\overline{A}$ are discrete or equivalently that the flat filtration on
    $\overline{A}$ is strict.
\end{remark}

\begin{definition}
    A filtered $\delta$-pair consists of a filtered $\delta$-ring $\F^\star A$ and a map of filtered
    commutative rings $\F^\star A\rightarrow\F^\star R$. A $\delta$-pair is strict if the
    filtration on $R$ is strict (recalling that the filtration on the filtered $\delta$-ring is
    strict by definition). We say that a filtered $\delta$-pair is
    prismatic if the kernel of $\F^\star A\rightarrow\F^\star R$ contains a Cartier divisor
    $I\subseteq A$ making $(\F^\star A,I)$ into a filtered prism.
\end{definition}

We introduce the filtered prismatic site and filtered prismatic cohomology. This is not the main
focus of our study, so we include it only for the sake of the curious reader. Below, we will give a
more detailed study of filtered prismatic cohomology using prismatization and we compare the two
theories in the case of a regular filtered quotient of a filtered prism (see
Proposition~\ref{prop:filteredenvelope}).

\begin{definition}[Filtered prismatic site]
    Fix a strict filtered $\delta$-pair $\F^\star A\rightarrow\F^\star R$.
    The filtered prismatic site $(\F^\star R/\F^\star A)_\Prism$ of $\F^\star R$ relative to the
    filtered $\delta$-ring $\F^\star A$ is the opposite 
    of the category of commutative squares
    \begin{equation}\label{eq:filteredsquare}
        \begin{gathered}
            \xymatrix{
                \F^\star A\ar[r]\ar[d]&\F^\star B\ar[d]\\
                \F^\star R\ar[r]&\F^\star\overline{B},
            }
        \end{gathered}
    \end{equation}
    of filtered commutative ring maps
    where $(\F^\star B,I)$ is a bounded filtered prism
    and $\F^\star A\rightarrow\F^\star B$ is a filtered $\delta$-ring map.
    The topology is the filtered $p$-adically faithfully flat topology, meaning that $\F^\star
    B\rightarrow\F^\star C$ is a cover if $B\rightarrow C$ is $(p,I)$-adically faithfully flat and
    $\gr^\star B\rightarrow\gr^\star C$ is a $(p,\gr^0I)$-adically faithfully flat map of graded commutative rings.
\end{definition}

\begin{definition}[Filtered structure sheaves]
    The presheaf $\F^\star\Oscr_\Prism$ which to any square~\eqref{eq:filteredsquare} associates $\F^\star
    B$ is a sheaf of strictly filtered commutative rings by arguing as in~\cite[Cor.~3.12]{prisms};
    see also Lemma~\ref{lem:sheaf}. Similarly, the presheaf $\Oscr_\Prismbar$, which
    sends~\eqref{eq:filteredsquare} to $\F^\star\overline{B}$ is a sheaf of strictly filtered
    commutative rings.
\end{definition}

\begin{definition}[Filtered Breuil--Kisin twists]
    The filtered Breuil--Kisin twists $\F^\star\Oscr_\Prism\{i\}$
    assign to each filtered prism $B$ the flat filtration $B\{i\}\otimes_B\F^\star B$ on $B\{i\}$.
\end{definition}

\begin{definition}[Filtered prismatic cohomology]
    We let
    $\F^\star\Prism_{\F^\star R/\F^\star A}^\s$, or more briefly
    $\F^\star\Prism_{R/A}^\s$, denote $\R\Gamma((\F^\star R/\F^\star
    R)_\Prism,\F^\star\Oscr_\Prism)$ and similarly for filtered Hodge--Tate cohomology
    $\F^\star\Prismbar_{R/A}^\s$ and the Breuil--Kisin twists $\F^\star\Prismbar_{R/A}^s\{i\}$, etc.
\end{definition}

\begin{example}[Filtered absolute prismatic cohomology]
    The case where $\F^\star A=\bZ_p$ with the trivial filtration (so $\F^i\bZ_p=0$ for $i>0$) gives
    $\F^\star\Prism_{R/\bZ_p}^\s$, a filtered version of absolute prismatic
    cohomology for any filtered commutative ring $\F^\star R$.
\end{example}

\begin{example}
    Suppose that the filtrations on $\F^\star A$ and $\F^\star R$ are trivial, meaning that
    $\F^iA\we 0$ and $\F^iR\we 0$ for $i>0$. Then, any square~\eqref{eq:filteredsquare}
    can be factored as
    $$
        \xymatrix{
            A\ar[r]\ar[d]&B\ar[r]\ar[d]&\F^\star B\ar[d]\\
            R\ar[r]&\overline{B}\ar[r]&\F^\star\overline{B}.
        }
    $$
    Conversely, any object in the unfiltered site $(R/A)_\Prism$ can be equipped with the trivial
    filtration. It follows in this case that $\F^\star\Prism_{R/A}^\s$ is naturally equivalent to
    $\Prism_{R/A}^\s$ with the trivial filtration.
\end{example}

The goal of the remainder of this section is to outline the general theory of filtered prismatic
cohomology from the stacky perspective of~\cite{bhatt-lurie-prism}.

\begin{construction}[Graded and filtered $\Spec$]\label{filteredspec}
    Given a filtered commutative ring $\F^\star R$ we can associate to it the graded Rees algebra
    $$\mathrm{Rees}(\F^\star R)=\bigoplus_{m\in\bZ} \F^m R\cdot t^{m},$$
    where $t$ has weight $1$,
    whose spectrum $\Spec\mathrm{Rees}(\F^\star R)$
    is a $\Gm$-equivariant affine scheme over $\bA^1$ (on the coordinate $t$).
    Let $\FSpec R=\FSpec\F^\star R $ denote the quotient $$\left(\Spec\mathrm{Rees}(\F^\star
    R)\right)/\Gm,$$ viewed as a stack over $\bA^1/\Gm$.
    We have a commutative diagram in which we can identify the fibers of $\FSpec\F^\star R$ over
    $0,1\in\bA^1/\Gm$:
    $$\xymatrix{
        \Spec\bigoplus_{m\in\bZ}\gr^mR\ar[r]\ar[d]&\GrSpec\gr^\star R\ar[r]\ar[d]&\FSpec\F^\star R\ar[d]&\Spec R\ar[d]\ar[l]\\
        \Spec\bZ\ar[r]&\B\Gm\ar[r]^0&\bA^1/\Gm&\Spec\bZ\ar[l]_1,
    }$$
    where $\GrSpec\gr^\star R$ denotes the quotient stack
    $\left(\Spec\bigoplus_{m\in\bZ}\gr^mR\right)/\Gm$. Note that these constructions are examples of
    relative $\Spec$ for $\F^\star R$ and $\gr^\star R$ viewed as commutative algebra objects in
    $\D(\bA^1/\Gm)$ and $\D(\B\Gm)$, respectively.
\end{construction}

As we are interested in functors on $p$-nilpotent commutative rings, we will in fact use the
following $p$-adic formal version of the construction above. See Appendix~\ref{app} for more details
on $p$-adic formal stacks, especially Warning~\ref{warn:formalaffineline} which clarifies that
$\Ahat=\Spf\bZ[t]\we\colim\Spf\bZ[t]/p^n$, not $\colim\Spf\bZ_p[t]/t^n$.

\begin{construction}[Graded and filtered $\Spf$]
    For a filtered commutative ring $\F^\star R$, we can restrict the functors above to
    $p$-nilpotent commutative rings to obtain a commutative diagram of pullback squares
    $$\xymatrix{
        \Spf\bigoplus_{m\in\bZ}\gr^mR\ar[r]\ar[d]&\GrSpf\gr^\star R\ar[r]\ar[d]&\FSpf\F^\star
        R\ar[d]&\Spf R\ar[d]\ar[l]\\
        \Spf\bZ_p\ar[r]&\B\Gmhat\ar[r]^0&\widehat{\bA}^1/\Gmhat&\Spf\bZ_p\ar[l]_1
    }$$
    of formal stacks.
\end{construction}

\begin{remark}[Functor of points for graded and filtered $\Spf$]\label{rem:fop}
    Given a filtered commutative ring $\F^\star R$, the filtered formal spectrum $\FSpf R$ has the
    following universal property as a stack over $\Ahat/\Gmhat$. Given a point $\Spf
    S\rightarrow\Ahat/\Gmhat$, described by the choice of a rank $1$ projective module $L$ over $\Spf S$ and a
    section $s\colon L\rightarrow S$, the pullback
    $$\xymatrix{
        P\ar[r]\ar[d]&\FSpf R\ar[d]\\
        \Spf S\ar[r]^{(L,s)}&\Ahat/\Gmhat.
    }$$
    giving the fiber over $(L,s)$ is the space of maps of graded commutative rings $\Rees(\F^\star R)\rightarrow S[L^{\pm 1}]$ taking $t^{-1} = 1 \cdot t^{-1} \in \F^{-1}R \cdot  t^{-1} 
    \subseteq \Rees(\F^\star R)$
    to the element $s^{-1}$ of $L^{\otimes -1}$ corresponding to the dual of $s$,
    where $S[L^{\pm 1}]=\bigoplus_{m\in\bZ}L^{\otimes m}$ is the graded commutative ring
    with the natural multiplication of sections.

    Similarly, if $R^\star$ is a graded commutative ring, then $\Gr\Spf R$ is described as a stack
    over $\B\Gmhat$ by saying that the fiber of $\GrSpf R$ over a point $\Spf S\rightarrow\B\Gmhat$
    corresponding to a rank $1$ projective module $P$ is the space of maps
    $R^\star\rightarrow S[L^{\pm 1}]$ of graded commutative rings.
\end{remark}

Now, we describe the Cartier--Witt stack controlling filtered prismatic cohomology.

\begin{construction}[$\delta$-stacks]
    Say that $\Xscr$ is
    a formal $\delta$-stack if $\Xscr\we\colim_{i\in I}\Spf A_i$ for some $I$-indexed diagram
    which is opposite to an $I^\op$-indexed diagram of $\delta$-rings. We view such a presentation
    as part of the structure of $\Xscr$ as a $\delta$-stack.
\end{construction}

More sophisticated notions of $\delta$-stack are possible, but this suffices for our purposes.

\begin{lemma}
    Let $\F^\star A$ be a filtered $\delta$-ring. Let $\FSpf\F^\star A$, or $\FSpf A$ for short,
    denote the formal stack defined above. Then, $\FSpf A$ is naturally a $\delta$-stack and the
    canonical map
    $\FSpf A\rightarrow\widehat{\bA}^1/\Gmhat$ is a map of $\delta$-stacks.
\end{lemma}

\begin{proof}
    By definition, $\FSpf A=(\Spf\bigoplus \F^m A\cdot t^{m})/\Gmhat$, so it can be computed as the
    colimit of of the simplicial object $(\Spf\bigoplus\F^mA\cdot t^{m})\times\Gmhat^{\times\bullet}$
    corresponding to the one-sided bar construction. The filtered $\delta$-ring structure on $A$
    gives $\Spf\bigoplus\F^mA\cdot t^{m}$ the structure of a formal $\delta$-scheme. Specifically,
    as $\delta(t)=0$,
    the $\delta$-operation on $\bigoplus\F^mA\cdot t^{m}$ sends
    $x\cdot t^{m}\in\F^mA\cdot t^{m}$ to $\delta(x)\cdot t^{pm}\in\F^{pm}A\cdot t^{pm}$. As $\varphi(x)=x^p+p\delta(x)$
    is a lift of Frobenius on $A$, it follows that the formula $\varphi(x\cdot
    t^{m})=(x^p+p\delta(x))\cdot t^{pm}$ defines a lift of Frobenius on
    the direct sum. Now, by~\cite[Lems.~2.17 and 2.18]{prisms}, there is an induced $\delta$-ring structure on the
    $p$-completion. Finally, the comultiplication $\nabla\colon\oplus\F^mA\cdot
    t^{m}\rightarrow\oplus\F^mA\cdot t^{m}\otimes_{\bZ}\bZ[t^{\pm 1}]$
    encoding the $\Gmhat$-action
    is defined for $x\in\F^m A$ by $\nabla(x\cdot t^{m})=(x\cdot t^{m})\otimes t^{m}$.
    The comultiplication is thus a map of $\delta$-rings as $$\delta(\nabla(x\cdot t^m))=\delta((x\cdot t^m)\otimes
    t^m)=\delta(x\cdot t^m)\otimes t^{pm}+(x^p\cdot t^{pm})\otimes \delta(t^m)+p\delta(x\cdot
    t^m)\otimes\delta(t^m)=\delta(x\cdot t^m)\otimes t^{pm}=\nabla(\delta(x\cdot t^m))$$ and remains such upon $p$-completion.   
    It follows that the simplicial object from the start of the proof is a simplicial object in
    formal affine $\delta$-schemes.
\end{proof}

\begin{construction}[Filtered prismatization]
    Let $\F^\star A\rightarrow\F^\star R$ be a filtered $\delta$-pair. The filtered prismatization
    $\F\WCart_{\F^\star R/\F^\star A}$, or $\F\WCart_{R/A}$ for short when the filtrations are clear
    from context, is the formal stack defined as the pullback
    $$\xymatrix{
        \F\WCart_{\F^\star R/\F^\star A}\ar[r]\ar[d]&\WCart_{\FSpf R}\ar[d]\\
        \FSpf A\times\WCart\ar[r]&\WCart_{\FSpf A},
    }$$
    where the bottom map is defined thanks to the naturality
    of~\cite[Const.~3.11]{bhatt-lurie-prism}; see also Construction~\ref{const:probe}. 
\end{construction}

\begin{construction}[Graded prismatization]
    Given a graded $\delta$-ring $A^\star$ and a graded $A^\star$-algebra $R^\star$,
    define $\Gr\WCart_{R^\star/A^\star}$ as the analogous pullback where $\GrSpf A$ replaces the filtered formal
    spectrum.
\end{construction}

\begin{remark}
    Filtered and graded prismatization are examples are prismatization relative to $\delta$-stacks.
    We leave the details to the reader.
\end{remark}

\begin{remark}[Functor of points of filtered prizmatizations]
    Using Remark~\ref{rem:fop}, we can describe $\F\WCart_{R/A}$ as a stack over $\Ahat/\Gmhat$
    along the compositions $\F\WCart_{R/A}\rightarrow\FSpf A\times\WCart\rightarrow\FSpf
    A\rightarrow \Ahat/\Gmhat$. If $p$ is nilpotent in $S$, an $S$-point of $\FSpf A\times\WCart$
    corresponds to a collection $$(L,s,f\colon\Rees(\F^\star A)\rightarrow S[L^{\pm 1}],\alpha\colon
    I\rightarrow W(S)),$$ where $\alpha$ is a Cartier--Witt divisor and $(L,s,f)$ describes an $S$
    point of $\FSpf A$. On the other hand, an $S$-point of $\WCart_{\FSpf A}$ corresponds to a
    Cartier--Witt divisor $(\alpha\colon I\rightarrow W(S))$ together with a $W(S)/I$-point of
    $\FSpf A$, and similarly for $\WCart_{\FSpf R}$. The bottom natural transformation $\FSpf
    A\times\WCart\rightarrow\WCart_{\FSpf A}$ sends the collection $(L,s,f,\alpha)$ to the
    Cartier--Witt divisor $\alpha\colon I\rightarrow W(S)$ together with the Teichm\"uller lift
    $[L]$ of $L$ and $[s]$ of $s$ and the composition $$\Rees(\F^\star A)\rightarrow W(S)[[L]^{\pm
    1}]\rightarrow W(S)/I[[L]^{\pm 1}]$$
    obtained from adjunction using the $\delta$-stack structure of $\FSpf A$. (For any
    $\delta$-stack $\Xscr$ and any commutative ring $R$ there is a natural map
    $\Xscr(R)\rightarrow\Xscr(W(R))$. As $p$ is not generally nilpotent in $W(R)$, we view
    $\Xscr(W(R))$ as $\lim(\Xscr(W_n(R)))$ for example. Applied to $\Ahat/\Gmhat$, this yields the
    Teichm\"uller lifts of line bundles and sections.)
    Thus, an $S$-point of $\F\WCart_{R/A}$ consists of $(L,s,f,\alpha)$ as above together with a
    Cartier--Witt divisor $\beta\colon J\rightarrow W(S)$, a rank $1$ projective module $M$ over
    $W(S)/I$, a section $u\colon M\rightarrow W(S)/I$, and a map $\Rees(\F^\star R)\rightarrow
    W(S)/I[M^{\pm 1}]$ sending $t^{-1}$ to $u^{-1}$ together with a fixed equivalence between $\alpha$
    with $\beta$, a fixed equivalence between $[L]$ and $M$ over $W(S)/I$ compatible with the
    equivalence between $\alpha$ and $\beta$, and a fixed equivalence between $\Rees(\F^\star A)\rightarrow W(S)/I[[L]^{\pm 1}]$
    and $\Rees(\F^\star A)\rightarrow\Rees(\F^\star R)\rightarrow W(S)/J[M^{\pm 1}]$ compatible with
    the other identifications.
\end{remark}

\begin{lemma}\label{lem:filteredcrystals}
    There are natural equivalences
    $$\D(\widehat{\bA}^1/\Gmhat\times\WCart)\we\F\D(\WCart)\quad\text{and}\quad\D(\B\Gmhat\times\WCart)\we\Gr\D(\WCart).$$
\end{lemma}

\begin{proof}
    See Appendix~\ref{app_1}.
\end{proof}

\begin{construction}[Filtered and graded prismatic crystals]
    Define the filtered prismatic crystal $\F^\star\Hscr_\Prism(R/A)$ to be the (derived)
    pushforward of the structure sheaf of $\F\WCart_{R/A}$ along the maps
    $\F\WCart_{R/A}\rightarrow\FSpf A\times\WCart\rightarrow\widehat{\bA}^1/\Gmhat\times\WCart$.
    Define the graded prismatic crystal $\gr^\star\Hscr_\Prism(R/A)$ in the analogous way as a sheaf
    on $\B\Gmhat\times\WCart$.
\end{construction}

\begin{warning}
    As in Proposition~\ref{prop:push}, we view this definition of the filtered prismatic crystals as correct
    only under filtered quasisyntomicity conditions.
\end{warning}

\begin{definition}[Filtered and graded prismatic cohomology]\label{def:filtered_cohomology}
    If $\F^\star A\rightarrow\F^\star R$ is a filtered $\delta$-pair, let
    $\F^\star\Prism_{R/A}=\R\Gamma(\F\WCart_{R/A},\Oscr_{\F\WCart_{R/A}})$, the filtered prismatic cohomology of
    $\F^\star R$ relative to $\F^\star A$. Similarly, if $A^\star\rightarrow R^\star$ is a
    graded $\delta$-pair, let
    $\gr^\star\Prism_{R/A}=\R\Gamma(\Gr\WCart_{R^\star/A^\star},\Oscr_{\Gr\WCart_{R^\star/A^\star}})$.
\end{definition}

\begin{remark}
    By Lemma~\ref{lem:filteredcrystals}, we can view $\F^\star\Hscr_\Prism(R/A)$ as a filtered object
    in prismatic crystals, and similarly for the graded prismatic crystal. By pushing forward along
    $\Ahat/\Gmhat\times\WCart\rightarrow\Ahat/\Gmhat$, we can view $\F^\star\Prism_{R/A}$ as
    being a filtered object of $\D(\bZ)_p^\wedge$.
\end{remark}

\begin{proposition}\label{prop:adjointable}
    The commutative diagram
    $$\resizebox{1\textwidth}{!}{
        \xymatrix{\ar @{} [drrr] |{(A)}
            \WCart_{\widehat{\bigoplus\gr^mR}/\widehat{\bigoplus\gr^mA}}\ar[r]\ar[d]&\Gr\WCart_{R/A}\ar[d]^{\gr
            q}\ar[r]&\F\WCart_{R/A}\ar[d]^{\F q}&\WCart_{\Rees(R)/\Rees(A)}\ar[l]\ar[d]&\WCart_{R/A}\ar[d]\ar[l]\\
            \WCart\ar[r]&\B\Gmhat\times\WCart\ar[r]^0&\widehat{\bA}^1/\Gmhat\times\WCart&\Ahat\times\WCart\ar[l]&\WCart\ar[l]_1
        }}
    $$
    consists of pullback squares which satisfy base change for homologically bounded above quasi-coherent cohomology, where
    $\widehat{\bigoplus\gr^mR}$ and $\widehat{\bigoplus\gr^mA}$ denote the $p$-completed direct sums.
\end{proposition}

\begin{proof}
    For simplicity, we will prove the result for square (A); the rest of the proof uses the
    same ideas. Square (A) fits into a commutative diagram
    \begin{equation}\label{eq:3x3}
        \begin{gathered}
            \xymatrix{
                \Gr\WCart_{R/A}\ar[r]\ar[d]&\F\WCart_{R/A}\ar[r]\ar[d]&\WCart_{\FSpf R}\ar[d]\\
                \GrSpf A\times\WCart\ar[r]\ar[d]&\FSpf A\times\WCart\ar[r]\ar[d]&\WCart_{\FSpf A}\ar[d]\\
                \B\Gmhat\times\WCart\ar[r]&\widehat{\bA}^1/\Gmhat\times\WCart\ar[r]&\WCart_{\widehat{\bA}^1/\Gmhat}.
            }
        \end{gathered}
    \end{equation}
    The top outer square fits into another commutative diagram
    $$\xymatrix{
        \Gr\WCart_{R/A}\ar[r]\ar[d]&\WCart_{\GrSpf R}\ar[r]\ar[d]&\WCart_{\FSpf R}\ar[d]\\
        \GrSpf A\times\WCart\ar[r]&\WCart_{\GrSpf A}\ar[r]&\WCart_{\FSpf A}.
    }$$
    In this diagram, the left square is a pullback by definition and the right square is a pullback
    as $\WCart_{(-)}$ commutes with limits. It follows that the outer square is a pullback which
    proves that the top outer square of~\eqref{eq:3x3} is a pullback. As the top right square
    of~\eqref{eq:3x3} is a pullback by definition, it follows that the top left square
    of~\eqref{eq:3x3} is a pullback, as desired (see for
    example~\cite[Lem.~4.4.2.1]{htt}). But, the bottom left square of~\eqref{eq:3x3} is a pullback
    square as it is the pullback of a pullback square to $\WCart$. It follows that square (A) is a
    pullback square as it is the outer left square in~\eqref{eq:3x3}.

    To prove base change for bounded above quasi-coherent cohomology for square (A) in the
    proposition, it is enough to prove it for the top-left and bottom-left squares
    of~\eqref{eq:3x3}. The bottom-left square is left to the reader: it is geometric (in the sense
    that all of the maps have affine formal scheme fibers) and standard
    arguments work. For the top-left square, we can pull back everything along the \v{C}ech complex
    of $\Ahat\times\WCart\rightarrow\Ahat/\Gmhat\times\WCart$. This results in a cosimplicial
    commutative diagram
    \begin{equation*}
        \begin{gathered}
            \xymatrix{
                \WCart_{\gr R^\bullet/\gr A^\bullet}\ar[r]\ar[d]&\WCart_{R^\bullet/A^\bullet}\ar[d]\\
                \Spf(\bigoplus\gr^iA)\times\Gmhat^\bullet\times\WCart\ar[r]\ar[d]&\Spf(\Rees(A))\times\Gmhat^\bullet\times\WCart\ar[d]\\
                \Gmhat^\bullet\times\WCart\ar[r]&\widehat{\bA}^1\times\Gmhat^\bullet\times\WCart,
            }
        \end{gathered}
    \end{equation*}
    of pullback squares,
    where $A^\bullet$ and $R^\bullet$ are the global sections of $\Spf(\Rees(A))\times\Gm^\bullet$ and $\Spf
    (\Rees(R))\times\Gm^\bullet$, respectively, while $\gr A^\bullet$ and $\gr R^\bullet$ are
    shorthands for
    the global sections of $\Spf(\bigoplus\gr^iA)\times\Gm^\bullet$ and $\Spf(\bigoplus\gr^i
    R)\times\Gm^\bullet$, respectively.
    Again, the bottom square satisfies base change for quasi-coherent cohomology, which uses that
    the closed inclusions
    $\Gmhat^\bullet\hookrightarrow\Ahat\times\Gmhat^\bullet$ have finite flat dimension.
    The top square satisfies base change for bounded above quasi-coherent cohomology by
    Corollary~\ref{cor:2} in each cosimplicial degree. Arguing as in the proof of that corollary,
    one deduces base change for bounded above quasi-coherent cohomology for the colimit diagram. 
\end{proof}

\begin{definition}
    Say that a strict filtered $\delta$-pair $\F^\star A\rightarrow\F^\star R$ is filtered relatively
    quasisyntomic if $\mathrm{Rees}(\F^\star A)_p^\wedge\rightarrow
    \mathrm{\Rees}(\F^\star R)_p^\wedge$ is relatively quasisyntomic.
\end{definition}

\begin{remark}
    A strict $\delta$-pair $\F^\star A\rightarrow\F^\star R$ is filtered relatively
    quasisyntomic if and only if
    $A\rightarrow R$ and 
    $\left(\bigoplus\gr^iA\right)_p^\wedge\rightarrow\left(\bigoplus\gr^iR\right)_p^\wedge$ are
    relatively quasisyntomic. Indeed, the $p$-complete Tor-amplitude of the cotangent complex
    $\L_{\mathrm{Rees}(\F^\star R)_p^\wedge/\mathrm{Rees}(\F^\star A)_p^\wedge}$ can be checked on
    underlying and associated graded pieces.
\end{remark}

\begin{corollary}\label{cor:underlying}
    For a relatively quasisyntomic filtered $\delta$-pair $\F^\star A\rightarrow \F^\star R$,
    \begin{enumerate}
        \item[{\em (i)}] $\R\Gamma(\WCart,\bigoplus_{m\in\bZ}\gr^m\Prism_{\gr^\star
            R/\gr^\star A})\we\Prism_{\widehat{\bigoplus_{m\in\bZ}\gr^mR}/\widehat{\bigoplus_{m\in\bZ}\gr^mA}}$,
        \item[{\em (ii)}] $\gr^\star\Prism_{\F^\star R/\F^\star A}\we\gr^\star\Prism_{\gr^\star
            R/\gr^\star A}$, and
        \item[{\em (iii)}] $\Rees(\F^\star\Prism_{\F^\star R/\F^\star A})_\HT^\wedge\we\Prism_{\Rees(\F^\star
            R)/\Rees(\F^\star A)}$, and
        \item[{\em (iv)}] $\F^{-\infty}\Prism_{\F^\star R/\F^\star A}\we\Prism_{R/A}$,
    \end{enumerate}
    where the direct sum in (i) is taken in $\D(\WCart)$ and hence implicitly completed along $p$ and the
    Hodge--Tate locus and the Rees algebra in (iii) is completed at the Hodge--Tate tower.
    In particular, the associated graded pieces of the filtration $\F^\star\Prism_{R/A}$ depend only
    on the graded $\delta$-pair $\gr^\star A\rightarrow\gr^\star R$.
\end{corollary}

\begin{proof}
    This follows from Proposition~\ref{prop:adjointable} and Proposition~\ref{prop:push}.
    Specifically, (i) follows base change for the left-most square in the statement of
    Proposition~\ref{prop:adjointable}; part (ii) is the square marked (A);
    part (iii) is base change for the second square from the right; part (iv) is base change for the
    right-most square.
\end{proof}

\begin{warning}
    The filtration $\F^\star\Prism_{R/A}$ is not typically complete, even when $\F^\star
    A\rightarrow\F^\star R$ is a complete prismatic $\delta$-pair. This occurs for example in the
    case of the prismatic envelope $\bZ_p\llbracket z\rrbracket\{\tfrac{z^n}{E(z)}\}$ with respect
    to a Breuil--Kisin prism $(\bZ_p\llbracket z\rrbracket,(E(z)))$ for an Eisenstein polynomial
    $E(z)$, which is studied in~\cite[Prop.~3.33]{akn-kzpn}.
\end{warning}

\begin{remark}
    The right square in
    $$\xymatrix{
        \WCart\ar[r]^1\ar[d]&\Ahat\times\WCart\ar[r]\ar[d]&\Ahat/\Gmhat\times\WCart\ar[d]\\
        \Spf\bZ_p\ar[r]^1&\Ahat\ar[r]&\Ahat/\Gmhat
    }$$ does not satisfy base change for quasi-coherent cohomology.
    The issue is that pullback along the bottom horizontal arrow is the $p$-completed Rees
    constsruction, which need not be complete with respect to the filtration induced by the
    Hodge--Tate locus. However, the outer square does satisfy base change for quasi-coherent
    cohomology on those filtered crystals $\F^\star\Hscr$ which eventually stabilize at equivalences
    $\cdots\rightarrow\F^m\Hscr\we\F^{m-1}\Hscr\we\F^{m-2}\Hscr\we\cdots$.
    This explains the Hodge--Tate completion necessary in (iii) of Corollary~\ref{cor:underlying} while there is
    no completion necessary in (iv).
\end{remark}

\begin{proposition}[Filtered and graded Hodge--Tate comparison]\label{prop:filteredhodgetate}
    Suppose that $\F^\star A\rightarrow\F^\star R$ is a prismatic filtered $\delta$-pair, such that 
    $\Rees\F^\star \bar A \to \Rees\F^\star R$ is $p$-adically smooth for some choice of prism structure on $\F^\star A$ yielding the quotient $\F^\star \bar A$ as in Remark \ref{rem_filteredbar}.
    Then, $\F\WCart_{R/A}^\HT\rightarrow\FSpf R$ is naturally
    a $T_{\F^\star R/\F^\star A}\{1\}^\#$-gerbe. Similarly in the graded case.
\end{proposition}

\begin{proof}
    This follows immediately from~\cite[Prop.~5.12]{bhatt-lurie-prism} by pulling back over the \v{C}ech
    complex $\Spf \Rees{\overline{A}}\times\Gmhat^\bullet$ of $\Spf\Rees{\overline{A}}\rightarrow\FSpf\overline{A}$.
    The graded case is similar.
\end{proof}

In order to gain understanding of filtered prismatic cohomology, especially in the case of prismatic filtered
$\delta$-pairs, we need to verify that a specific filtered prismatic envelope construction does
produce a filtered prism.

\begin{construction}[Filtered Koszul complexes]\label{const:filteredkoszul}
    Recall that if $J=(x_1,\ldots,x_c)$ is an ideal in $A$, then we will say that the sequence $(x_1,\ldots,x_c)$ is Koszul-regular
    if the Koszul
    complex $\Kos(x_1,\ldots,x_c)$ is a resolution of $A/J$ (so has no higher homology). If
    $\F^\star A$ is a strictly filtered commutative ring and each $x_j$ has weight $w(j)$, meaning that
    $x_j\in\F^{w(j)}A$ but $x_j\notin\F^{w(j)+1}$, then one can make the Koszul complex
    $\Kos(x_1,\ldots,x_c)$ into a complex of filtered $\F^\star A$-modules, which we will call the
    filtered Koszul complex $\F^\star\Kos(x_1,\ldots,x_c)$. For example, if $J=(x_1,x_2)$, then the
    filtered Koszul complex of the sequence $(x_1,x_2)$ will be
    $$\cdots\rightarrow 0\rightarrow\F^{\star-w(1)-w(2)}
    A\xrightarrow{\begin{pmatrix}x_2\\-x_1\end{pmatrix}}\F^{\star-w(1)} A\oplus\F^{\star-w(2)}
    A\xrightarrow{\begin{pmatrix}x_1&x_2\end{pmatrix}}\F^\star A\rightarrow 0\rightarrow\cdots.$$
    We say that $(x_1,\ldots,x_c)$ is a filtered Koszul-regular sequence if the filtered Koszul complex
    $\F^{\star}\Kos(x_1,\ldots,x_c)$ has vanishing positive homology and if $\H_0(\F^{\star}\Kos(x_1,\ldots,x_c))$
    is a strict filtration, necessarily on $A/J$. This implies that $(x_1,\ldots,x_c)$ is a regular
    sequence and much more. Equivalently, $\F^{\star}\Kos(x_1,\ldots,x_c)$ has vanishing positive homology
    (in the abelian category of (non-strict) filtered complexes)
    {\em and} the associated graded complex $\gr^\star\Kos(x_1,\ldots,x_c)$ also has vanishing
    positive homology. We let $\F^\star A/J=\H_0(\F^\star\Kos(x_1,\ldots,x_c))$ and $\F^\star
    J=\ker(\F^\star A\rightarrow\F^\star A/J)$ when $(x_1,\ldots,x_c)$ is a filtered Koszul-regular
    sequence generating $J$. In this case, $\F^\star A/\F^\star J\iso\F^\star A/J$; moreover the
    filtration on $A/J$ is the image filtration obtained by letting $\F^n A/J=\im(\F^nA\rightarrow
    A/J)$.
\end{construction}

\begin{remark}
    There is a natural extension of the definition above to infinite sequences of elements;
    we will use this extension below.
\end{remark}

We will be applying this to the case where $\F^\star A$ is a filtered prism and
$J=(d,x_1,\ldots,x_c)$. In this case, note that $d$ has weight $0$.

\begin{example}[Breuil--Kisin quotients]
    For example if $k$ is a perfect field of characteristic $p$ and $E(z_0)\in A^s=W(k)\llbracket
    z_0,\ldots,z_s\rrbracket$ is an Eisenstein polynomial corresponding to a uniformizer
    $\varpi\in\Oscr_K=W(k)\llbracket z_0\rrbracket/E(z_0)$, then we can view $A^s$ has having the
    adic filtration with respect to the ideal $(z_0,\ldots,z_s)$.
    The ideal $J=(E(z_0),z_0^n,z_1-z_0,\ldots,z_s-z_0)$ is filtered Koszul. To see this,
    use the decomposition
    $$\F^\star\Kos(E(z_0),z_0^n,z_1-z_0,\ldots,z_s-z_0)\we\F^\star\Kos(E(z_0))\otimes_{\F^\star
    A}\F^\star\Kos(z_0^n)\otimes_{\F^\star A}\F^\star\Kos(z_1-z_0)\otimes_{\F^\star
    A}\cdots\otimes_{\F^\star A}\F^\star\Kos(z_s-z_0)$$ to realize
    the filtered Koszul complex as quasi-isomorphic to
    $\varpi^\star\Oscr_K/\varpi^n$, which is
    the strict filtration on the quotient $\Oscr_K/\varpi^n$ having weight $m\geq 0$ term the ideal
    $(\varpi^m)\subseteq\Oscr_K/\varpi^n$.
\end{example}

\begin{example}[The stack of a filtered prism]\label{ex:stackoffilteredprism}
    Let $(\F^\star A,I)$ be a filtered prism. We compute the stack $\F\WCart_{\overline{A}/A}$.
    Recall that the unfiltered stack $\WCart_{\overline{A}/A}$ is equivalent to $\SSpf A$.
    We show that the filtered stack is equivalent to $\FSSpf A=\SSpf(\Rees(A)_{(p,I)}^\wedge)/\Gmhat$.
    For this, we can form the \v{C}ech complex of $\Spf(\Rees(A))\times\WCart\rightarrow\FSpf
    A\times\WCart$ and pull back $\FWCart_{\overline{A}/A}\rightarrow\FSpf A\times\WCart$ to the \v{C}ech
    complex.
    This results in a commutative diagram
    $$\xymatrix{
        \WCart_{\Rees(\overline{A})^\bullet/\Rees(A)^\bullet}\ar[d]\ar[r]&\FWCart_{\overline{A}/A}\ar[d]\\
        \Spf(\Rees(A))\times\Gmhat^{\bullet}\times\WCart\ar[r]&\FSpf A\times\WCart
    }$$
    where the geometric realization of the left vertical arrow is the right vertical arrow and where
    $\Rees(\overline{A})^\bullet$ and $\Rees(A)^\bullet$ are the global sections of
    $\Spf(\Rees(\overline{A}))\times\Gmhat^\bullet$ and $\Spf(\Rees(A))\times\Gmhat^\bullet$, respectively.
    Now, the fiber of $\Rees(A)\rightarrow\Rees(\overline{A})$ is $\Rees(I)$ by the filtered prism
    condition. Zariski locally on $\Spf(\Rees(A))$, $\Rees(I)$ is generated by a distinguished
    element $d$ and hence $\Rees(I)\we\Rees(A)$, i.e., the $\delta$-pair
    $\Rees(A)\rightarrow\Rees(\overline{A})$ determines a prism after $\Rees(I)$-adic completion
    (which is the same as $I$-adic completion using that $\Rees(A)$ is an $A$-algebra). The same is
    true of $\Rees(A)^\bullet$ in each cosimplicial degree. Thus,
    there is an identification of simplicial objects
    $$\WCart_{\Rees(\overline{A})^\bullet/\Rees(A)^\bullet}\we\SSpf(\Rees(A)_{(p,I)}^{\bullet,\wedge})$$
    by Lemma~\ref{lem:Iinsensitive}; by definition, the right-hand side has geometric realization
    $\FSSpf A$, as desired.
\end{example}

\begin{construction}[Filtered prismatic envelopes]\label{const:filteredenvelope}
    Let $(\F^\star A,(d))$ be a filtered prism and let $J=(d,x_1,\ldots,x_c)$ be a filtered
    Koszul-regular ideal, so that the image sequence $x_1,\ldots,x_c$ defines a filtered
    Koszul-regular ideal in $\F^\star\overline{A}$. Let $\F^\star
    A\{\tfrac{x_1}{d},\ldots,\tfrac{x_c}{d}\}_{(p,d)}^\wedge$ denote the $(p,d)$-completed filtered
    (derived) tensor
    product $$\left(\F^\star A\{a_1,\ldots,a_c\}_\delta\otimes_{\F^\star
    A\{r_1,\ldots,r_c\}_\delta}\F^\star A\right)_{(p,d)}^\wedge.$$
    Here $\F^\star A\{r_1,\ldots,r_c\}_\delta$ is the free filtered $\delta$-ring over $\F^\star A$
    on weight $w_u$ generators $r_u$ and similarly for $\F^\star A\{a_1,\ldots,a_c\}$. The left map
    is the filtered $\delta$-ring map sending $r_u$ to $da_u-x_u$; the right map sends $r_u$ to
    zero. Let $\F^\star R=\F^\star\overline{A}/(x_1,\ldots,x_c)$. We will show in the next
    proposition that the filtered prismatic envelope computes the filtered prismatic cohomology of
    $R$ relative to $A$.
\end{construction}

\begin{proposition}\label{prop:filteredenvelope}
    If $(\F^\star A,(d))$ is an oriented filtered prism and $J=(d,x_1,\ldots,x_c)\subseteq A$ is an
    ideal generated by a filtered Koszul-regular sequence $(d,x_1,\ldots,x_c)$, then
    \begin{enumerate}
        \item[{\em (a)}] $\F^\star A\{\tfrac{x_1}{d},\ldots,\tfrac{x_c}{d}\}_{(p,d)}^\wedge$ is a filtered prism,
        \item[{\em (b)}] $\F^\star A\{\tfrac{x_1}{d},\ldots,\tfrac{x_c}{d}\}_{(p,d)}^\wedge$ defines
            a final object of the filtered relative prismatic site $(\F^\star R/\F^\star A)_\Prism$, and
        \item[{\em (c)}] the natural maps $\FSSpf
            A\{\tfrac{x_1}{d},\ldots,\tfrac{x_c}{d}\}_{(p,d)}^\wedge\rightarrow\FWCart_{R/A}$ and
            $\F^\star\Prism_{R/A}\rightarrow\F^\star\Prism_{R/A}^\site\we\F^\star
            A\{\tfrac{x_1}{d},\ldots,\tfrac{x_c}{d}\}_{(p,d)}^\wedge$ are
            equivalences.
    \end{enumerate}
\end{proposition}

\begin{proof}
    We give the case when $c=1$ and write $x$ for $x_1$ and $w$ for $w(1)$. The general case is
    obtained by taking suitable $(p,d)$-completed tensor products. The Rees algebra
    $\Rees(\F^\star A\{\tfrac{x}{d}\}_{(p,d)}^\wedge)$ has $(p,d)$-completion the prismatic envelope
    $(\Rees(\F^\star A))\{\tfrac{x}{d}\}_{(p,d)}^\wedge$ by symmetric monoidality of the Rees construction.
    By our assumption on $x$, $(\Rees(\F^\star A))\{\tfrac{x}{d}\}_{(p,d)}^\wedge$ is
    discrete~\cite[Prop.~3.13]{prisms}.
    Thus, the Rees algebra $\Rees(\F^\star A\{\tfrac{x}{d}\}_{(p,d)}^\wedge)$ is a connective $A$-module whose
    $(p,d)$-completion is discrete. Moreover, it is a countable coproduct of $(p,d)$-complete
    $A$-modules by the countable generation of free $\delta$-rings; see~\cite[Prop.~3.8]{akn-kzpn}. This implies that $\Rees(\F^\star A\{\tfrac{x}{d}\}_{(p,d)}^\wedge)$ is already
    discrete, which shows that the filtration on the filtered prismatic envelope is filtered
    discrete. Moreover, the $(p,d)$-completion of
    $\Rees(\F^\star A\{\tfrac{x}{d}\}_{(p,d)}^\wedge)/t^{-1}$ is equivalent to
    $\Rees(\F^\star A\{\tfrac{x}{d}\}_{(p,d)}^\wedge)_{(p,d)}^\wedge/t^{-1}$. The latter is seen to
    be equivalent to the prismatic envelope $(\bigoplus_{m\in\bZ}\gr^mA)\{\tfrac{x}{d}\}_{(p,d)}^\wedge$ by symmetric
    monoidality of taking the cofiber of $t^{-1}$ (the associated graded). This is again discrete by
    the result for prismatic envelopes in Bhatt--Scholze~\cite[Prop.~3.13]{prisms}. So, the same decompletion argument shows
    that $\Rees(\F^\star A\{\tfrac{x}{d}\}_{(p,d)}^\wedge)/t^{-1}$, which is again a direct sum of
    connective $(p,d)$-complete $A$-module spectra, is discrete. This implies that the filtration on
    the filtered prismatic envelope $\F^\star A\{\tfrac{x}{d}\}_{(p,d)}^\wedge$ is strict. Hence,
    the filtered prismatic envelope is a filtered $\delta$-ring. To show it is a filtered prism, we
    have to show strictness of the induced filtration on the Hodge--Tate locus, which is to say that we want to
    show that $\Rees(\F^\star \overline{A\{\tfrac{x}{d}\}}_p^\wedge)$ and
    $\Rees(\F^\star \overline{A\{\tfrac{x}{d}\}}_p^\wedge)/t^{-1}$ are discrete. After
    $p$-completing $\Rees(\F^\star \overline{A\{\tfrac{x}{d}\}}_p^\wedge)$, we obtain an object
    equivalent to
    $\Rees(\F^\star A)\{\tfrac{x}{d}\}_p^\wedge/d$, which is discrete by the $d$-torsion freeness of
    the prismatic envelope of~\cite{prisms}. Similarly, $p$-completed and mod $t$ we obtain an
    object equivalent to $(\bigoplus_{m\in\bZ}\gr^m A)\{\tfrac{x}{d}\}_p^\wedge/d$ (using the graded
    regularity of the element $x$), which is
    discrete again. Thus, arguing as above and decompleting, we obtain discreteness, as desired.
    This proves (a).

    By part (a), the filtered prismatic envelope defines an object of the filtered prismatic site.
    Given an object $\F^\star B$ of the filtered prismatic site, there is a unique induced map
    $A\{\tfrac{x}{d}\}_{(p,d)}^\wedge\rightarrow B$ of $\delta$-rings. For this to be a map of
    filtered prisms is a property since the filtrations on filtered $\delta$-rings is assumed to be
    strict. To see that this induced map is indeed filtered it is enough
    to note that the induced map $A\{a\}_\delta\rightarrow B$ respects the filtrations by
    construction and hence the
    factorization through the $(p,d)$-completed quotient $A\{\tfrac{x}{d}\}_{(p,d)}^\wedge$ is filtered. This proves
    part (b).

    For part (c), we claim that $\F^\star\Prism_{R/A}$ is a filtered prism, which follows from the
    filtered Hodge--Tate comparison theorem. Thus, it defines an object of the filtered prismatic
    site and hence there is a map $\F^\star
    A\{\tfrac{x}{d}\}_{(p,d)}^\wedge\rightarrow\F^\star\Prism_{R/A}$. The composite
    $\F^\star\Prism_{R/A}\xrightarrow{f}
    A\{\tfrac{x}{d}\}_{(p,d)}^\wedge\xrightarrow{g}\F^\star\Prism_{R/A}$ is an idempotent map of
    filtered $\delta$-rings. By~\cite[Lem.~4.3.10]{bhatt-lurie-apc}, $f$ and $g$ are isomorphisms on
    $\F^0$ and hence the composite $g\circ f$ is the identity on $\F^0$. The only way to make the
    identity map filtered is by letting it act as the identity on each filtered piece, so $g\circ f$
    is the identity on each $\F^m\Prism_{R/A}$. It follows from strictness of the filtrations
    that $f$ and $g$ are isomorphisms on each $\F^m$,
    proving the second statement of part (c). The first statement is proved by a Rees-algebra
    argument.
\end{proof}

\begin{construction}[Filtered syntomic cohomology]\label{const:filtered_syntomic}
    Given a bounded $\delta$-pair $(A,R)$, the prismatic package $\Prismpackage_{R/A}$ includes, for
    each integer $i$, the following information:
    \begin{enumerate}
        \item[(a)] the Nygaard-filtered, Nygaard-complete, Frobenius-twisted,
            Breuil--Kisin cohomology $\N^{\geq\star}\Prismhat^{(1)}_{R/A}\{i\}$, which is a complete
            filtered $A$-module;
        \item[(b)] an $A$-linear map $\can\colon\N^{\geq
            i}\Prismhat^{(1)}_{R/A}\{i\}\rightarrow\Prismhat^{(1)}_{R/A}\{i\}$;
        \item[(c)] a $\varphi_A$-semilinear map $c\varphi\colon\N^{\geq
            i}\Prismhat_{R/A}^{(1)}\{i\}\rightarrow\Prismhat^{(1)}_{R/A}$,
    \end{enumerate}
    from which one builds
    \begin{enumerate}
        \item[(d)] relative syntomic cohomology complexes $$\bZ_p(i)(R/A)=\fib\left(\N^{\geq
            i}\Prismhat_{R/A}^{(1)}\{i\}\xrightarrow{\can-c\varphi}\Prismhat^{(1)}_{R/A}\{i\}\right).$$
    \end{enumerate}
    This data is functorial in maps of bounded $\delta$-pairs.

    Now, suppose that $(\F^\star A,\F^\star R)$ is a filtered $\delta$-pair. We can use descent
    along the formally smooth morphism $\Spf\Rees(\F^\star A)\rightarrow\FSpf A$ and the naturality of the constructions above to
    construct an $\infty$-category $\Cscr_{\FSpf A}$ by right Kan extension (thanks to
    Remark~\ref{rem_module}(iii)) and
    a filtered prismatic package $\Prismpackage_{\F^\star R/\F^\star A}\in\Cscr_{\FSpf
    A}=(\Prism^{[\star]}_{\F^\star R/\F^\star A}\{\star\},\N^{\geq\star}\Prismhat_{\F^\star R/\F^\star
    A}^{(1)}\{\star\},c,\varphi)$ using
    Theorem~\ref{thm:intro}(6, 7). To see that the underlying prismatic part
    $\Prism^{[\star]}_{\F^\star R/\F^\star A}\{\star\}$ agrees with
    $\F^\star\Prism^{[\star]}_{R/A}\{\star\}$ of Definition~\ref{def:filtered_cohomology}, one uses
    that $\Spf\Rees(\F^\star A)\times\WCart\rightarrow\FSpf A\times\WCart$ is a formally smooth
    cover together with base change (after Hodge--Tate completion) for relative prismatic cohomology
    (Theorem~\ref{thm:intro}(7)).
    The filtered prismatic package includes, for each integer $i$,
    \begin{enumerate}
        \item[(a)]
            $\F^{\geq\star}\N^{\geq\star}\Prismhat^{(1)}_{R/A}\{i\}=\N^{\geq\star}\Prismhat^{(1)}_{\F^\star
            R/\F^\star A}\{i\}$, which is a
            filtered $\F^\star A$-module equipped with a second filtration (the Nygaard filtration)
            for which it is complete;
        \item[(b)] a filtered $\F^{\geq\star} A$-linear map $\can\colon\F^{\geq\star}\N^{\geq
            i}\Prismhat^{(1)}_{R/A}\{i\}\rightarrow\F^{\geq\star}\Prismhat^{(1)}_{R/A}\{i\}$;
        \item[(c)] a $p$-filtered $\varphi_{\F^\star A}$-semilinear map $c\varphi\colon\F^{\geq\star}\N^{\geq
            i}\Prismhat_{R/A}^{(1)}\{i\}\rightarrow\F^{\geq\star}\Prismhat^{(1)}_{R/A}$.
    \end{enumerate}
    Here, a map of filtered objects $\F^\star M\rightarrow\F^\star
    N$ is $p$-filtered if it arises from a map $F^*\F^\star M\rightarrow\F^\star N$ where $F$
    is the Frobenius endomorphism on $\bA^1/\Gm$. From this, we can define, for $i\geq 0$,
    \begin{enumerate}
        \item[(d)] filtered relative syntomic cohomology complexes
            $$\F^{\geq\star}\bZ_p(i)(\F^\star R/\F^\star A)=\fib\left(\F^{\geq\star}\N^{\geq
            i}\Prismhat_{R/A}^{(1)}\{i\}\xrightarrow{\can-c\varphi}\F^{\geq\star}\Prismhat^{(1)}_{R/A}\{i\}\right)$$
            for $i\in\bZ$.
    \end{enumerate}
    In (d), we use that a $p$-filtered map of decreasing $\bN$-indexed
    filtration forgets to a filtered map since $\F^{pi}M$ maps canonically to $\F^iM$ for
    $i\geq 0$.
    As above, we often write $\F^\star\bZ_p(i)(R/\A)$ for $\F^\star\bZ_p(i)(\F^\star R/\F^\star A)$
    when it seems no confusion can arise.
\end{construction}

\begin{variant}
    If $(A^\star,R^\star)$ is a graded $\delta$-pair, then there is a graded version of the
    construction above, producing a prismatic package
    $\Prismpackage_{R^\star/A^\star}\in\Cscr_{\GrSpf A^\star}$. However, because the Frobenius
    map $c\varphi$ is $p$-graded in the sense that it takes weight $i$ into weight $pi$, the
    induced syntomic cohomology
    $$\bZ_p(i)(R^\star/A^\star)=\fib\left(\bigoplus_{\star\in\bZ}\gr^\star\N^{\geq
    i}\Prismhat_{R/A}^{(1)}\{i\}\xrightarrow{\can-c\varphi}\bigoplus_{\star\in\bZ}\gr^\star\Prismhat^{(1)}_{R/A}\{i\}\right)$$
    is filtered, not graded. The filtration is given by
    $$\F^m\bZ_p(i)(R^\star/A^\star)=\fib\left(\bigoplus_{\star\geq m}\gr^\star\N^{\geq
    i}\Prismhat_{R/A}^{(1)}\{i\}\xrightarrow{\can-c\varphi}\bigoplus_{\star\geq
    m}\gr^\star\Prismhat^{(1)}_{R/A}\{i\}\right).$$
\end{variant}

\begin{proposition}
    Let $(\F^\star A,\F^\star R)$ be a relatively quasisyntomic filtered $\delta$-pair. The pullback maps
    $1^*\colon\Cscr_{\FSpf A}\rightarrow\Cscr_{\Spf A}$ and $0^*\colon\Cscr_{\FSpf
    A}\rightarrow\Cscr_{\GrSpf\gr^\star A}$ induce natural equivalences
    \begin{enumerate}
        \item[{\em (a)}] $1^*\Prismpackage_{\F^\star R/\F^\star A}\we\Prismpackage_{R/A}$ and
        \item[{\em (b)}] $0^*\Prismpackage_{\F^\star R/\F^\star A}\we\Prismpackage_{\gr^\star
            R/\gr^\star A}$.
    \end{enumerate}
\end{proposition}

\begin{proof}
    This follows from base change for the prismatic package, Theorem~\ref{thm:intro}(7).
\end{proof}

\begin{corollary}\label{cor:underlying_syntomic}
    For each $i\in\bZ$ and all relatively quasisyntomic filtered $\delta$-pairs $A\rightarrow R$, there are natural
    equivalences
    \begin{enumerate}
        \item[{\em (i)}] $\gr^\star\bZ_p(i)(\F^\star R/\F^\star A)\we\gr^\star\bZ_p(i)(\gr^\star R/\gr^\star A)$ and
        \item[{\em (ii)}] $\F^0\bZ_p(i)(\F^\star R/\F^\star A)\we\bZ_p(i)(\F^0R/\F^0A)$.
    \end{enumerate}
\end{corollary}

\appendix
\section{Some formal stack theory}\label{app}

We give some background on quasi-coherent sheaf theory on $p$-adic formal stacks in~\ref{app_1} culminating in the
establishment of some results on base change for quasi-coherent cohomology in Section~\ref{app_2}.
We claim no originality for this material, but do not know of a suitable reference.

\subsection{Quasi-coherent sheaves on formal stacks}\label{app_1}

Our goal in this section is to compare several possible definitions of a quasi-coherent sheaf on
$\Spf R$. While in the body of this paper, all constructions are by default derived unless specified
otherwise, given an abelian group $M$, we will need to work both with the derived reduction modulo
$p$ of $M$, which we will write as
$M\doubleslash p$, and the non-derived version $M/p=\H_0(M\doubleslash p)$.

As in the rest of this paper, we work with formal prestacks, i.e., presheaves of spaces (aka anima) on
the category of spectra of (discrete) $p$-nilpotent rings. Formal stacks are sheaves of anima or
spectra with respect to the faithfully flat topology on spectra of $p$-nilpotent rings.

\begin{remark}
    Our formal stacks are formal analogues of higher stacks as opposed to derived stacks
    (see~\cite{toen-global} for an overview of the distinction and further references). Quasi-lci conditions
    guarantee that the Cartier--Witt theory developed here using formal higher stacks agrees with the
    derived analogue, as in~\cite{bhatt-lurie-prism}.
\end{remark}

\begin{example}[Formal spectra]
    Given a commutative ring $R$, $\Spf R$ denotes the formal stack given by the restriction of
    $\Spec R$ to a functor on spectra of $p$-nilpotent commutative rings. If $R$ has bounded
    $p$-power torsion, then the natural map $\Spf R_p^\wedge\rightarrow\Spf R$ is an equivalence of
    formal stacks.
\end{example}

\begin{warning}\label{warn:formalaffineline}
    We will write $\Ahat$ for $\Spf\bZ[t]$ and $\Gmhat$ for $\Spf\bZ[t^{\pm 1}]$. Note that this
    diverges from common notation in the literature, where for example $\Ahat$ denotes the
    prestack on all affine schemes, 
    which to any $\Spec R$ assigns the non-unital ring of nilpotent elements in $R$. By contrast,
    for us, formal will always mean formal with respect to the $p$-adic topology.
\end{warning}

\begin{definition}[Quasicoherent sheaves on formal stacks]
    Given any formal (pre)stack $\Xscr$, one defines $\D(\Xscr)=\lim_{\Spf R\rightarrow\Xscr}\D(R)$
    where the limit ranges over all $p$-nilpotent commutative rings $R$ and all $\Spf R$ points of
    $\Xscr$. In all cases studied in this paper, the limit above is computed as the limit over a small
    category and so we will not get into size considerations here.
\end{definition}

For an object $\Fscr\in\D(\Xscr)$ we will refer to the `value' on $x\colon \Spf R \to \Xscr$ as the pullback and write it as $x^*\Fscr \in \D(R)$.

\begin{example}
    If $p$ is nilpotent in $R$, then $\D(\Spf R)\we\D(R)$.
\end{example}

\begin{definition}
    Let $\Xscr$ be a formal prestack and let $\Fscr\in\D(\Xscr)$ be a quasi-coherent sheaf.
    \begin{enumerate}
        \item[(i)] We say $\Fscr$ is perfect if for every $p$-nilpotent commutative ring
            $R$ and every point $x\colon\Spf R\rightarrow\Xscr$ the pullback $x^*\Fscr\in\D(R)$ is
            perfect. Write $\Perf(\Xscr)\subseteq\D(\Xscr)$ for the full subcategory of
            perfect objects.
        \item[(ii)] We say $\Fscr$ has $p$-complete Tor-amplitude in $[a,b]$ for
            $a,b\in\bZ\sqcup\{\pm\infty\}$ if for every $x\colon\Spf R\rightarrow\Fscr$ the
            pullback $x^*\Fscr$ belongs to $\D(R)_{[a,b]}$, i.e., $\H_i(x^*\Fscr)=0$ for
            $i\notin[a,b]$.
    \end{enumerate}
\end{definition}

Now, we study $p$-complete $\bZ$-linear stable presentable $\infty$-categories.

\begin{definition}[$p$-completion]
    Given a $\D(\bZ)$-module $\Cscr$ in $\Pr^\L$, the $p$-completion of $\Cscr$ is
    $\Cscr_p^\wedge=\Cscr\otimes_{\D(\bZ)}\D(\bZ)_p^\wedge$.
\end{definition}

\begin{remark}
    The object $\D(\bZ)_p^\wedge$ is idempotent in the $\infty$-category of $\D(\bZ)$-modules in
    $\Pr^\L$. Indeed, consider the localization sequence
    $$\D(\bZ[1/p])\rightarrow\D(\bZ)\rightarrow\D(\bZ)_p^\wedge$$ of $\D(\bZ)$-modules.
    Tensoring with $\D(\bZ)_p^\wedge$ results in another localization sequence, for example
    by~\cite[Cor.~3.5]{agh}. However, $\D(\bZ)_{p}^\wedge$ is compactly generated by $\bF_p$ so that
    $\D(\bZ)_p^\wedge\we\D(\End_\bZ(\bF_p))$ and hence
    $\D(\bZ[1/p])\otimes_{\D(\bZ)}\D(\bZ)_p^\wedge\we\D(\bZ[1/p]\otimes_\bZ\End_\bZ(\bF_p))\we 0$.
    Thus, $\D(\bZ)_p^\wedge\rightarrow\D(\bZ)_p^\wedge\otimes_{\D(\bZ)}\D(\bZ)_p^\wedge$ is an equivalence.
    It follows that $\Mod_{\D(\bZ)_p^\wedge}(\Pr^\L)$ is a full subcategory of
    $\Mod_{\D(\bZ)}(\Pr^\L)$, that the inclusion has a left adjoint, and that it is given by
    $p$-completion in the sense above.
    
    Note also that we could have worked over the sphere spectrum as opposed to $\bZ$ here, but we will not need this generality.
\end{remark}

\begin{definition}[$p$-completeness]
    A $\D(\bZ)$-module $\Cscr$ in $\Pr^\L$ is $p$-complete if the natural map
    $\Cscr\rightarrow\Cscr_p^\wedge$ is an equivalence. By the remark above, this property is
    equivalent to saying that $\Cscr$ admits the structure of a $\D(\bZ)_p^\wedge$-module in
    $\Pr^\L$. Moreover, by idempotentness of $\D(\bZ)_p^\wedge$, the $p$-completion of any $\Cscr$
    in $\Mod_{\D(\bZ)}(\Pr^\L)$ is $p$-complete.
\end{definition}

Our philosophy is that $p$-adic formal geometry is algebraic geometry done in $\D(\bZ)_p^\wedge$-modules in
$\Pr^\L$. Recall from~\cite[Sec.~23.1.2]{sag} that a compactly assembled presentable $\infty$-category $\Cscr$ is one which is a retract,
in $\Pr^\L$, of a compactly generated $\infty$-category. This is equivalent to the dualizability of
$\Cscr$ in $\Pr^\L_\st$. If $\Cscr$ is a $\D(\bZ)$-module in $\Pr^\L_\st$, then it is dualizable as
a $\D(\bZ)$-module if and only if it is dualizable in $\Pr^\L_\st$ by~\cite[Thm.~D.7.0.7]{sag}.
Moreover, a $\D(\bZ)$-linear version of~\cite[Prop.~D.7.3.1]{sag} implies that dualizability
relative as a $\D(\bZ)$-module in $\Pr^\L_\st$ is equivalent to being a retract, in
$\D(\bZ)$-modules in $\Pr^\L_\st$, of a compactly generated $\D(\bZ)$-linear presentable $\infty$-category.

\begin{lemma}\label{lem:pcompletion}
    For any compactly assembled $\D(\bZ)$-module $\Cscr$ in $\Pr^\L$, the natural map
    $\Cscr\otimes_{\D(\bZ)}\D(\bZ)_p^\wedge\rightarrow\lim_n\Cscr\otimes_{\D(\bZ)}\D(\bZ/p^n)$ is an equivalence.
\end{lemma}

\begin{proof}
    The case of $\D(\bZ)_p^\wedge\we\lim_n\D(\bZ/p^n)$ follows from
    Proposition~\ref{prop:completeomni} below.
    Now being dualizable relative to $\D(\bZ)$ implies that the functor $\Cscr\otimes_{\D(\bZ)} - $ has a left adjoint given by tensoring with the dual $\Cscr^\vee$. Thus it preserves limits.
\end{proof}

\begin{proposition}\label{prop:completeomni}
    For a commutative ring $R$ with bounded $p$-power torsion, there are natural equivalences $$\D(\Spf
    R)\we\lim_n\D(R\doubleslash p^n)\we\D(R)_{\text{$R\doubleslash p$-cpl}}\we\D(R)^\wedge_p
    \we\Mod_R(\D(\bZ)_p^\wedge),$$
    where $\D(R)_{\text{$R\doubleslash p$-cpl}}$ is the $\infty$-category of $R\doubleslash
    p$-local objects in $\D(R)$.
\end{proposition}

\begin{proof}
    The fourth equivalence holds more generally: if $\Cscr$ is any presentable $\D(\bZ)$-linear
    stable $\infty$-category, then $\D(R)\otimes_{\D(\bZ)}\Cscr\we\Mod_R(\Cscr)$. This follows from
    the definition of the tensor product on $\Pr^\L$; see~\cite[Prop.~4.8.1.17]{ha}.
    The third equivalence follows from the equivalence between complete and torsion objects:
    $\D(R)_{\text{$R\doubleslash p$-cpl}}\we\D(R)_{\text{$p$-tors}}$, where the torsion category is the full subcategory
    of $R$-module spectra $M$ where each $x\in\pi_*M$ is annihilated by some power of $p$ (depending
    on $x$); see~\cite[Thm.~3.9]{MNN17} or~\cite[Thm.~3.3.5]{hovey-palmieri-strickland}.
    On the other hand, $\D(R)_{\text{$p$-tors}}$ is compactly generated by $R\doubleslash p$, for example
    by~\cite[Prop.~3.7]{MNN17}. It follows that
    $$\D(R)_{\text{$R\doubleslash
    p$-cpl}}\we\D(R)_{\text{$p$-tors}}\we\D(R\otimes_{\bZ}\End_{\bZ}(\bF_p))\we\D(R)\otimes_{\D(\bZ)}\D(\End_\bZ(\bF_p))\we\D(R)\otimes_{\D(\bZ)}\D(\bZ)_p^\wedge,$$
    functorially in $R$.

    Now, we prove the second and first equivalences.
    If $R$ is an animated commutative ring, then $\D(R)_{\text{$R\doubleslash p$-cpl}}\we\lim\D(R\doubleslash
    p^n)$. This is closely related to~\cite[Prop.~2.21]{MNN17}, which implies that if $S^0=R\doubleslash p$
    and if $S^\bullet$ is the descent complex associated to $R\rightarrow S^0$, then
    $\D(R)_{\text{$R\doubleslash p$-cpl}}\we\lim\Tot(\D(S^\bullet))$. One can rewrite the limit as
    $\lim_n\Tot^n(\D(S^\bullet))$, the limit of the finite
    totalizations. As each $\Tot^n$ is a finite limit, the natural map
    $\D(\Tot^n(S^\bullet))\rightarrow\Tot^n(\D(S^\bullet))$ is fully faithful and the map
    $\D(R)\rightarrow\Tot^n(\D(S^\bullet))$ canonically factors through the subcategory
    $\D(\Tot^n(S^\bullet))$ generated by the unit.  One has
    $\Tot^n(S^\bullet)\we S\doubleslash p^{n+1}$ for all $n\geq 0$ by~\cite[Prop.~2.14]{MNN17} and hence there are functors
    $$\D(R)_{\text{$R\doubleslash p$-cpl}}\rightarrow\lim_n\D(R\doubleslash p^{n+1})\rightarrow\lim_n{\Tot}^n(\D(S^\bullet)),$$
    where the composition is an equivalence and the right-hand map is fully faithful.
    It follows that the left-hand map is fully faithful and essentially surjective, i.e., an
    equivalence.

    Note that $\D(\Spf R)\we\lim_n\D(R/p^n)$ as every $R\rightarrow S$ where $p$ is nilpotent in $S$
    factors through some $R/p^n$.
    To complete the proof of the lemma, we now show that if $R$ is a commutative ring with bounded $p$-torsion,
    the natural map $\lim_n\D(R\doubleslash p^n)\rightarrow\lim_n\D(R/p^n)$ is an equivalence.
    As $R\doubleslash p$ is a compact generator of $\D(R)_p^\wedge$ and hence gives a generator of
    $\lim_n\D(R\doubleslash p^n)$ by the equivalence above, it suffices to check that it determines a
    compact generator of $\lim_n\D(R/p^n)$ and that the induced map on endomorphism ring spectra is
    an equivalence. First, we check that it is a compact generator. Suppose that $\{M_n\}$ is a tower of $\{R/p^n\}$-modules
    defining a non-zero object of the limit. We claim first that $M_1$ is non-zero. To see this, it is
    enough to assume inductively that $M_2$ is non-zero. Then, $M_1\we M_2\otimes_{R/p^2}R/p$.
    But, $R/p^2\rightarrow R/p$ is descendable by~\cite[Prop.~3.35]{mathew-galois} and hence
    $\D(R/p^2)\rightarrow\D(R/p)$ is conservative by~\cite[Prop.~3.22]{mathew-galois}.
    As $\lim\D(R/p^n)\xrightarrow{\{M_n\}\mapsto M_1}\D(R/p)$ commutes with all colimits, it follows
    that this functor reflects colimits. Now, given a tower $\{M_n\}$ as above,
    \begin{align*}
        \Map_{\lim_n\D(R/p^n)}(\{R/p^n\otimes_R R\doubleslash p\},\{M_n\})&\we\lim_n\Map_{R/p^n}(R/p^n\otimes_R R\doubleslash p,M_n)\\
        &\we\lim_n(M_n\otimes_R R\doubleslash p)[-1]\\
        &\we\lim_n(M_n\otimes_{R/p^n}R/p^n\otimes_R R\doubleslash
    p)[-1].
    \end{align*}
    Each $R/p^n\otimes_R R\doubleslash p$ fits into a canonical fiber sequence $$R/p^n[p][1]\rightarrow
    R/p^n\otimes_R R\doubleslash p\rightarrow R/p$$ where the outer terms are canonically $R/p$-modules.
    We see that each $M_n\otimes_{R/p^n}R/p^n\otimes_R R\doubleslash p$ fits into a fiber sequence
    $$M_1\otimes_{R/p}R/p^n[p][1]\rightarrow M_n\otimes_{R/p^n}R/p^n\otimes_R R\doubleslash p\rightarrow
    M_1.$$
    As $R$ has bounded $p$-power torsion, the limit of the left-hand terms vanishes because it is
    pro-zero. The limit of the right-hand terms is $M_1$. Thus, the mapping spectrum above is
    naturally equivalent to $M_1[-1]$. By what we have already said, the functor $\{M_n\}\mapsto
    M_1$ is conservative and reflects filtered colimits, which completes the proof that $R\doubleslash p$
    maps to a compact generator in $\lim_n\D(R/p^n)$. The proof that the endomorphism ring spectra
    agree is an exercise in the Mittag--Leffler condition using boundedness of $p$-power torsion; it
    is left to the reader.
\end{proof}

\begin{remark}
    This is one point where the derived approach is more natural. If one works with formal spectra
    of $p$-nilpotent animated commutative rings in the context of derived formal stacks, then for any animated commutative ring $R$ there
    is an equivalence $\D(\Spf R)\we\D(R)_p^\wedge$.
\end{remark}

Here are two examples of derived categories of formal stacks which can also be identified as
$p$-completions.

\begin{example}
    Recall that $\D(\B\Gm)\we\Gr\D(\bZ)$, the stable $\infty$-category of graded $\bZ$-module
    spectra. As $\B\Gmhat\we\colim_n\B\Gm\times_{\Spf\bZ_p}\Spf\bZ/p^n$, we have
    \begin{align*}
        \D(\B\Gmhat)&\we\lim_n\D(\B\Gm\times_{\Spf\bZ_p}\Spf\bZ/p^n)\\
        &\we\lim_n\prod_{\bZ}\D(\bZ/p^n)\\
        &\we\prod_{\bZ}\D(\bZ)_p^\wedge.
    \end{align*}
    Thus, we see that $\D(\B\Gmhat)$ is the $\infty$-category of graded $p$-complete $\bZ$-module
    spectra. It is also the $p$-completion of $\D(\B\Gm)\we\Gr\D(\bZ)$ as
    $$\left(\prod_{i\in\bZ}\D(\bZ)\right)\otimes_{\D(\bZ)}\D(\bZ)_p^\wedge\we\prod_{i\in\bZ}\D(\bZ)_p^\wedge,$$
    for example, using the fact that infinite coproducts and products agree in $\Pr^\L$ as can
    be proved using~\cite[Thm.~5.5.3.18]{htt}.
\end{example}

\begin{example}
    Similarly, $\D(\bA^1/\Gm)\we\F\D(\bZ)\we\Fun(\bZ^\op,\D(\bZ))$, the stable $\infty$-category of
    filtrations in $\bZ$-module spectra. The $p$-completion is equivalent to
    $\D(\Ahat/\Gmhat)\we\Fun(\bZ^\op,\D(\bZ)_p^\wedge)$, the $\infty$-category of filtrations in $p$-complete
    $\bZ$-module spectra.
\end{example}

\begin{remark}
    Spectral sequences can behave in strange ways in the $p$-complete world. For example,
    consider the filtration $p^\star\bQ_p$ on $\bQ_p$ where $p^n\bQ_p\subseteq\bQ_p$ consists of the
    elements of $p$-adic valuation at least $n$. As a filtration in $\D(\bZ_p)$, this is
    complete and exhaustive, i.e., $\colim_{n\rightarrow\infty}p^{-n}\bQ_p=\bQ_p$. However, as a
    filtration in $\D(\bZ_p)_p^\wedge$, valid because each $p^n\bQ_p\iso\bZ_p$, this filtration is
    complete but has colimit $0$, the $p$-completion of $\bQ_p$. Thus, this is a complete
    exhaustive filtration on $0$ where the associated graded pieces $\gr^i$ are all equivalent to
    $\bF_p$. The associated spectral sequence degenerates at the $\E_1$-page, but the
    $\E_\infty=\E_1$-page is very far from the associated graded of the induced filtration on the homotopy groups of the
    abutment.
\end{remark}

\begin{example}
    Let $B\leftarrow A\rightarrow C$ be a $p$-completely Tor-independent diagram of bounded commutative rings with pushout $B\otimes_A C$. Then
    $\D(B\otimes_AC)\we\D(B)\otimes_{\D(A)}\D(C)$. It follows that $$\D(\Spf
    (B\otimes_AC))\we\lim_n\D(B/p^n\otimes_{A/p^n}C/p^n)\we\lim_n\D(B/p^n)\otimes_{\D(A/p^n)}\D(C/p^n)\we\D(B)_p^\wedge\otimes_{\D(A)_p^\wedge}\D(C)_p^\wedge.$$
    Moreover, this is equivalent to $\D(B\otimes_AC)\otimes_{\D(\bZ)}\D(\bZ)_p^\wedge$.
\end{example}

Now, we briefly discuss the standard $t$-structure on quasi-coherent sheaves on formal stacks.

\begin{definition}
    Let $\Xscr$ be a formal stack. An object $\Fscr\in\D(\Xscr)$ is connective if
    $x^*\Fscr\in\D(R)_{\geq 0}$ for all $p$-nilpotent rings $R$ and all $R$-points $x\colon\Spf
    R\rightarrow\Xscr$. The full subcategory $\D(\Xscr)_{\geq 0}\subseteq\D(\Xscr)$ defines the
    connective part of a $t$-structure on $\D(\Xscr)$.
\end{definition}

\begin{example}
    Suppose that $R$ is a bounded $p$-complete commutative ring and consider $\D(\Spf R)_{\geq
    0}\subseteq\D(R)_p^\wedge$. Recall that $\D(\Spf R)\we\Mod_R(\D(\Spf\bZ_p))$. We claim that
    $\Fscr\in\D(\Spf R)$ is connective if and only if the underlying object of $\D(\Spf\bZ_p)$ is
    connective, which happens if and only if $\H_i(\Fscr)=0$ for $i<0$. In other words, pushforward
    along $\Spf R\rightarrow\Spf\bZ_p$ is $t$-exact. If $\H_i(\Fscr)=0$ for $i<0$, then for every
    $R\rightarrow S$ where $p$ is nilpotent in $S$, one has that $\Fscr\otimes_RS$ is connective as
    connective objects are closed under tensor products. On the other hand,
    $\Fscr\we\lim_n(\Fscr\otimes_R R\doubleslash p^n)$ and each term of $\Fscr\otimes_R
    R\doubleslash p^n$ is connective
    as it fits into an exact sequence $\Fscr\otimes_R
    R/p^n\otimes_{R/p^n}R[p^n][1]\rightarrow\Fscr\otimes_R R\doubleslash p^n\rightarrow\Fscr\otimes_R R/p^n$
    and the outer terms are connective by hypothesis. Moreover, the fiber of $\Fscr\otimes_R
    R\doubleslash p^n\rightarrow\Fscr\otimes_R R\doubleslash p^{n-1}$ is naturally equivalent to
    $\Fscr\otimes_R R\doubleslash p$ and
    is in particular connective. It follows that in the tower $\{\Fscr\otimes_R R\doubleslash p^n\}_n$ the
    induced maps on $\H_0$ are surjective and so the Mittag--Leffler condition applies and the limit
    is connective. Thus, the $t$-structure on $\D(\Spf R)$ is the
    standard $t$-structure; the coconnnective objects are precisely those derived $p$-complete
    $R$-module spectra which are coconnective. The heart of $\D(\Spf R)$ is the abelian category of
    derived $p$-complete $R$-modules.
\end{example}

\begin{warning}
    The coconnective objects in $\D(\Spf R)$ do not necessarily have the property that
    their base change to $\Spf S$ is coconnective for $S$ an $R$-algebra in which $p$ is nilpotent.
\end{warning}

\subsection{Base change for prismatic crystals}\label{app_2}

Now, we give a discussion of base change and the projection formula for prismatic cohomology of
general prismatic crystals.

\begin{definition}[Base change for quasi-coherent cohomology]
    Suppose that
    \begin{equation}\label{eq:basechange}
        \xymatrix{
            \Xscr'\ar[r]^{g'}\ar[d]^{f'}&\Xscr\ar[d]^f\\
            \Yscr'\ar[r]^g&\Yscr
        }
    \end{equation}
    is a commutative square of formal stacks. We say this square satisfies base change for
    quasi-coherent cohomology if the commutative diagram
    $$\xymatrix{
        \D(\Yscr)\ar[r]^{f^*}\ar[d]^{g^*}&\D(\Xscr)\ar[d]^{g'^*}\\
        \D(\Yscr')\ar[r]^{f'^*}&\D(\Xscr')
    }$$ of left adjoint functors is right adjointable in the sense of~\cite[Def.~4.7.4.13]{ha}.
    Using the equivalence $\alpha\colon f'^*\circ g^*\we g'^*\circ f^*$ witnessing the commutativity
    of these squares,
    this means that the natural transformation of functors $g^*\circ f_*\rightarrow f'_*\circ
    f'^*\circ g^*\circ
    f_*\we f'_*\circ g'^*\circ f^*\circ f_*\rightarrow f'_*\circ
    g'^*\colon\D(\Xscr)\rightarrow\D(\Yscr')$ is an equivalence.
\end{definition}

\begin{warning}
    Right adjointability is not a symmetric notion, and hence neither is the notion of a square
    satisfying base change for quasi-coherent cohomology. In other words, if it holds for a square
    as above, it might or might not hold for the transposed square.
\end{warning}

\begin{notation}
    Let $\D(\Xscr)^+\subseteq\D(\Yscr)$ be the subcategory of homologically bounded above objects.
\end{notation}

\begin{variant}
    \begin{enumerate}
        \item[(i)]
            We say that~\eqref{eq:basechange} satisfies base change for {\em bounded above} quasi-coherent
            cohomology if $g^*$ and $g'^*$ preserve bounded above objects
            and if the natural transformation $g^*\circ f_*\rightarrow f'_*\circ g'^*$ is an
            equivalence when evaluated on any object of $\D(\Xscr)^+$.
        \item[(ii)]
            We say that~\eqref{eq:basechange} satisfies base change for $\Oscr$-cohomology if the natural transformation $g^*\circ f_*\rightarrow f'_*\circ g'^*$ is an
            equivalence when evaluated on $\Oscr_\Xscr$.
    \end{enumerate}
\end{variant}

There is also the projection formula. See~\cite[Tag~08EU]{stacks} for the classical case of schemes.

\begin{definition}[Projection formula]
    A morphism $f\colon\Xscr\rightarrow\Yscr$ of formal stacks satisfies the projection formula
    (for quasi-coherent cohomology) if
    the induced symmetric monoidal functor $f^*\colon\D(\Yscr)\rightarrow\D(\Xscr)$ does, i.e., if
    for every $\Fscr\in\D(\Xscr)$ and $\Gscr\in\D(\Yscr)$, the natural map
    $f_*(\Fscr)\otimes\Gscr\rightarrow f_*(\Fscr\otimes f^*(\Gscr))$ is an equivalence.
\end{definition}

\begin{variant}
    We say that $f$ satisfies the projection formula for bounded above quasi-coherent cohomology with respect to
    a class of objects $\Ascr\subseteq\D(\Yscr)$ if $f_*(\Fscr)\otimes\Gscr\rightarrow f_*(\Fscr\otimes
    f^*(\Gscr))$ is an equivalence for $\Fscr\in\D(\Xscr)^+$ and $\Gscr\in\Ascr$.
    When $\Gscr=\Oscr_\Yscr$, the projection formula morphism is always an equivalence. 
    Thus, $f$ satisfies the projection formula for quasi-coherent cohomology with
    respect to objects in the thick subcategory generated by the unit $\Oscr_\Yscr$ in $\D(\Yscr)$.
\end{variant}

\begin{lemma}[Colimits and the projection formula]\label{lem:colimitprojection}
    Suppose that $f\colon\Xscr\rightarrow\Yscr$ is a colimit in formal stacks of morphisms
    $f_i\colon\Xscr_i\rightarrow\Yscr_i$. If each $f_i$ satisfies the projection formula and if each
    of the induced
    squares
    $$\xymatrix{
        \Xscr_i\ar[r]^{a_i}\ar[d]^{f_i}&\Xscr\ar[d]^f\\
        \Yscr_i\ar[r]^{b_i}&\Yscr
    }$$ satisfies base change for quasi-coherent cohomology, then $f$ satisfies the projection
    formula.
\end{lemma}

\begin{proof}
    Fix $\Fscr\in\D(\Xscr)$ and $\Gscr\in\D(\Yscr)$. As $\D(\Yscr)\we\lim_i\D(\Yscr_i)$, to prove
    that $f_*(\Fscr)\otimes\Gscr\rightarrow f_*(\Fscr\otimes f^*(\Gscr))$ is an equivalence, it is
    enough to prove that for each $i$ the natural map $b_i^*(f_*(\Fscr)\otimes\Gscr)\rightarrow
    b_i^*f_*(\Fscr\otimes f^*(\Gscr))$ is an equivalence. Using symmetric monoidality of the
    pullback functors as well as base change for quasi-coherent cohomology, this map can be
    rewritten as the natural map $f_{i*}(a_i^*(\Fscr))\otimes b_i^*(\Gscr)\rightarrow f_{i*}a_i^*(\Fscr\otimes
    f^*(\Gscr))\we f_{i*}(a_i^*(\Fscr)\otimes f_i^*(b_i^*(\Gscr)))$, which is an equivalence by the
    projection formula for $f_i$ applied to $a_i^*(\Fscr)\in\D(\Xscr_i)$ and $b_i^*(\Gscr)\in\D(\Yscr_i)$.
\end{proof}

\begin{variant}[Colimits and bounded above projection formula]\label{var:colimitprojectionbounded}
    If in the situation of Lemma~\ref{lem:colimitprojection} each $f_i$ satisfies the projection
    formula for bounded above quasi-coherent cohomology with respect to classes
    $\Ascr\in\D(\Yscr_i)$, if $\Ascr\subseteq\D(\Yscr)$ is a class of objects which is sent
    under $b_i^*$ into $\Ascr_i$, and if $a_i^*$ preserves bounded above objects, then $f$ satisfies
    base change for bounded above quasi-coherent cohomology with respect to $\Ascr$.
\end{variant}

\begin{lemma}[Colimits and base change]\label{lem:colimitbasechange}
    Suppose that a square $\square$ as in~\eqref{eq:basechange} is realized as a colimit in formal stacks
    of commutative squares
    $$\square_i=\begin{gathered}\xymatrix{
            \Xscr_i'\ar[r]^{g'_i}\ar[d]^{f'_i}&\Xscr_i\ar[d]^{f_i}\\
            \Yscr_i'\ar[r]^{g_i}&\Yscr_i
    }\end{gathered}$$
    over some indexing simplicial set $I$.
    For each edge $j\rightarrow i$ in $I$, let
    $$\square_{i,j}=\begin{gathered}\xymatrix{
            \Xscr_j\ar[r]\ar[d]^{f_j}&\Xscr_i\ar[d]^{f_i}\\
            \Yscr_j\ar[r]&\Yscr_i
    }\end{gathered}\quad\text{and}\quad
    \square_{i,j}'=\begin{gathered}\xymatrix{
            \Xscr_j'\ar[r]\ar[d]^{f_j'}&\Xscr_i'\ar[d]^{f_i'}\\
            \Yscr_j'\ar[r]&\Yscr_i'.
    }\end{gathered}
    $$
    If $\square_i$, $\square_{i,j}$, and $\square_{i,j}'$ satisfy base change for cohomology for
    each edge $j\rightarrow i$ in $I$, then so does
    $\colim_{i\in I}\square_i\we\square$.
\end{lemma}

\begin{proof}
    The lemma follows from~\cite[Cor.~4.7.4.18(2)]{ha}, to which we refer for more details.
    At the level of quasi-coherent
    sheaves, applying $\D(-)$ turns these colimit diagrams into limit diagrams in $\Pr^\L$, i.e.,
    $\D(\square)\we\lim\D(\square_i)$. As $i$ varies, the pullback functors
    $\D(\Yscr_i)\rightarrow\D(\Xscr_i)$ give a functor $\D(f^*)\colon I^\op\rightarrow\Fun(\Delta^1,\Pr^\L)$.
    By hypothesis, this functor factors to give a functor
    $\D(f^*)\colon I^\op\rightarrow\Fun^{\mathrm{RAd}}(\Delta^1,\Pr^\L)$,
    where $\Fun^{\mathrm{RAd}}(\Delta^1,\Pr^\L)$ is the (non-full) subcategory of
    $\Fun(\Delta^1,\Pr^\L)$ on all left adjoint functors $\Cscr\rightarrow\Dscr$ is $\Pr^\L$ and
    where the morphisms are natural transformations
    $(\Cscr\rightarrow\Dscr)\rightarrow(\Escr\rightarrow\Fscr)$ in $\Pr^\L$ corresponding to right
    adjointable commutative squares
    $$\xymatrix{
        \Cscr\ar[r]\ar[d]&\Dscr\ar[d]\\
        \Escr\ar[r]&\Fscr.
    }$$
    The same result for the maps induced by $\Xscr_i'\rightarrow\Yscr_i'$ gives another functor
    $\D(f'^*)\colon
    I^\op\rightarrow\Fun^{\mathrm{RAd}}(\Delta^1,\Pr^\L)$. Moreover, the condition on the squares
    $\square_i$ implies that pulling back along the maps $f_i$ and $f_i'$ gives a natural
    transformation $\D(f^*)\rightarrow\D(f'^*)$ of functors
    $I^\op\rightarrow\Fun^{\mathrm{RAd}}(\Delta^1,\Pr^\L)$. Taking limits and
    applying~\cite[Cor.~4.7.4.18(2)]{ha}, the result follows.
\end{proof}

\begin{remark}
    The proof of Lemma~\ref{lem:colimitbasechange} also shows that the commutative squares
    $$\begin{gathered}\xymatrix{
        \Xscr_i\ar[r]\ar[d]&\Xscr\ar[d]\\
        \Yscr_i\ar[r]&\Yscr
    }\end{gathered}\quad\text{and}\quad\begin{gathered}\xymatrix{
            \Xscr'_i\ar[r]\ar[d]&\Xscr'\ar[d]\\
            \Yscr'_i\ar[r]&\Yscr'
    }\end{gathered}$$
    satisfy base change for quasi-coherent cohomology for all $i\in I$ as well.
\end{remark}

\begin{remark}[Colimits are universal]
    Recall the fact that colimits are universal in an $\infty$-topos, meaning that they are
    stable under pullback. In particular, if $\colim_{i\in I}\Yscr_i\we\Yscr$ as formal
    stacks and if
    $\Xscr\rightarrow\Yscr$ is a map of formal stacks, then $\colim_{i\in
    I}\Yscr_i\times_\Yscr\Xscr\rightarrow\Xscr$ is an equivalence. See~\cite[Sec.~6.1]{htt}.
\end{remark}

\begin{variant}\label{var:colimitbasechange}
    Suppose that $\{\Xscr_i\xrightarrow{f_i}\Yscr_i\}_{i\in I}$ is an $I$-indexed diagram of morphisms of
    formal stacks. If each square
    $$\square_{i,j}=\begin{gathered}\xymatrix{
            \Xscr_j\ar[r]^{a_{ij}'}\ar[d]^{f_j}&\Xscr_i\ar[d]^{f_i}\\
            \Yscr_j\ar[r]^{a_{ij}}&\Yscr_i
    }\end{gathered}$$
    satisfies base change for quasi-coherent cohomology, then so does each square
    $$\square_i=\begin{gathered}\xymatrix{
            \Xscr_i\ar[r]^{g_i'}\ar[d]^{f_i}&\Xscr\ar[d]^f\\
        \Yscr_i\ar[r]^{g_i}&\Yscr
    }\end{gathered}$$
    where $(\Xscr\rightarrow\Yscr)\we\colim_I(\Xscr_i\rightarrow\Yscr_i)$. The proof is similar to
    that of Lemma~\ref{lem:colimitbasechange}. However, we will need a version which holds under
    weakened hypotheses in the bounded above case. Specifically, assume that each $a_{ij}^*$ and
    $a_{ij}'^*$ preserve bounded above objects and that each square $\square_{ij}$ satisfies base
    change for bounded above quasi-coherent cohomology. Then, so does each square $\square_i$.
    To prove this, consider an object $\Fscr\in\D(\Xscr)^+$ and note that $\Fscr\we\lim
    g'_{i*}g'^*_i\Fscr$. Let $\Fscr_i=g'^*_i\Fscr$. So, we can view $\Fscr$ as the compatible family
    of objects $\Fscr_i\in\D(\Xscr_i)^+$, meaning that there are compatible equivalences $\Fscr_j\we
    a'^*_{ij}\Fscr_i$ for all arrows $j\rightarrow i$ in $I$. Applying $f_*$ one finds that
    $f_*\Fscr\we\lim_I f_*g'_{i*}\Fscr_i\we\lim_I g_{i*}f_{i*}\Fscr_i$. We want to know that
    $g_i^*f_*\Fscr\we f_{i*}g'^*_i\Fscr\we f_{i*}\Fscr_i$. However, there are equivalences $a_{ij}^*f_{i*}\Fscr_i\we
    f_{j*}a'^*_{ij}\Fscr_i\we f_{j*}\Fscr_j$, compatible in $I$, by our assumption on base change for bounded above
    cohomology and the preservation of bounded above objects by $a_{ij}^*$ and $a'^*_{ij}$.
    Thus, the objects $f_{i*}\Fscr_i$ (together with the given equivalences) determine an object of
    $\lim\D(\Yscr_i)\we\D(\Yscr)$ (or a cartesian section of an appropriate fibration corresponding
    to $i\mapsto\D(\Yscr_i)$). This object is $f_*\Fscr$ and the proof is complete.
\end{variant}

\begin{remark}
    Let $\D(\Xscr_\star)\rightarrow I^\op$ be the cartesian fibration corresponding to
    $i\mapsto\D(\Xscr_i)$ together with the pullback maps. The limit $\D(\Xscr)$ is the
    $\infty$-category of cartesian sections of the fibration. There is by hypothesis a natural
    transformation $\D(\Yscr_\star)\rightarrow\D(\Xscr_\star)$ of cartesian fibrations over $I^\op$
    and the pullback functor $f^*$ takes cartesian sections to cartesian sections. The functor
    $\D(\Xscr_\star)\rightarrow\D(\Yscr_*)$, corresponding to the right adjoints $f_*$, need not
    take cartesian fibrations to cartesian fibrations. However, this is exactly what is guaranteed
    by our assumption on the squares $\square_{i,j}$. 
\end{remark}

\begin{lemma}\label{lem:formalbasechange}
    Let $B\leftarrow A\rightarrow C$ be a diagram of bounded $p$-complete commutative rings. If the
    diagram is $p$-completely Tor-independent (for example if $B$ or $C$ is $p$-completely flat over
    $A$), then
    $$\xymatrix{
        \Spf(B\tensorhat_AC)\ar[r]\ar[d]&\Spf C\ar[d]\\
        \Spf B\ar[r]&\Spf A
    }$$
    satisfies base change for quasi-coherent cohomology. In addition, $\Spf C\rightarrow\Spf A$
    satisfies the projection formula.
\end{lemma}

\begin{proof}
    The assumptions imply that this square is the colimit of the squares
    $$\xymatrix{
        \Spf(B/p^n\otimes_{A/p^n}C/p^n)\ar[r]\ar[d]&\Spf(C/p^n)\ar[d]\\
        \Spf(B/p^n)\ar[r]&\Spf(A/p^n)
    }$$
    and each of these squares satisfies the projection formula and base change for quasi-coherent
    cohomology by a standard argument; see for example~\cite[Tags~08EU, 08IB]{stacks}. The statement now
    follows from Lemmas~\ref{lem:colimitprojection}
    and~\ref{lem:colimitbasechange}.
\end{proof}

\begin{remark}
    This simple case of base change can be proved directly. Consider the commutative square
    $$\xymatrix{\D(A)_p^\wedge\ar[r]\ar[d]&\D(C)_p^\wedge\ar[d]\\
    \D(B)_p^\wedge\ar[r]&\D(B\tensorhat_AC)_p^\wedge.
    }$$
    For a derived $p$-complete $C$-module spectrum $M$, the base change transformation is a morphism
    $B\tensorhat_A M\rightarrow (B\tensorhat_A C)\tensorhat_C M$, which is an equivalence.
    However, the lemma gives an illustration of how we will use colimits to produce new cases of
    base change.
\end{remark}

\begin{lemma}
    Let $B\leftarrow A\rightarrow C$ be a diagram of bounded prisms. If the
    diagram is $(p,I)$-completely Tor-independent (for example if $B$ or $C$ is $(p,I)$-completely flat over
    $A$), then the pullback square
    $$\xymatrix{
        \SSpf(B\tensorhat_AC)\ar[r]\ar[d]&\SSpf C\ar[d]\\
        \SSpf B\ar[r]&\SSpf A
    }$$
    satisfies base change for quasi-coherent cohomology. In addition, $\SSpf C\rightarrow\SSpf A$
    satisfies the projection formula.
\end{lemma}

\begin{proof}
    Follow the proof of Lemma~\ref{lem:formalbasechange} but with $\Spf(-/I^n)$.
\end{proof}

In order to prove more general cases of base change for prismatic crystals, we need a base
case provided by the next proposition. Let $\Tscr_A\subseteq\D(\SSpf A)$ be the full subcategory of
those $(p,I)$-complete $M$ such that $M/(p,I)$ is equivalent to a filtered colimit of uniformly
bounded above perfect $A/(p,I)$-modules.

\begin{proposition}\label{prop:basicsquare}
    Suppose that $A$ is a bounded prism and that $X$ is a quasicompact separated
    formal scheme
    over $\Spf\overline{A}$. If $X$ is bounded and $\L_{X/\overline{A}}$ has $p$-complete
    Tor-amplitude in $[0,1]$, then for every map of prisms $A\rightarrow B$ of bounded
    $(p,I)$-complete Tor-amplitude,
    the pullback square
    $$\xymatrix{
        \WCart_{X_{\overline{B}}/B}\ar[r]\ar[d]&\WCart_{X/A}\ar[d]\\
        \SSpf B\ar[r]&\SSpf A
    }$$
    satisfies base change for bounded above quasi-coherent cohomology. Moreover,
    $\WCart_{X/A}\rightarrow\SSpf A$ satisfies the projection formula for
    $\Fscr\in\D(\WCart_{X/A})^+$ and $\Gscr\in\Tscr_A$.
\end{proposition}

\begin{proof}
    We have the following lemma, which is not a special case of
    Lemma~\ref{lem:colimitbasechange}.

    \begin{lemma}\label{lem:finitecolimitbasechange}
        Suppose that
        $$\xymatrix{
            \Xscr_i'\ar[r]^{g_i'}\ar[d]^{f_i'}&\Xscr_i\ar[d]^{f_i}\\
            \Yscr'\ar[r]^g&\Yscr
        }$$
        is a finite diagram of commutative squares each satisfying base change for (bounded above)
        quasi-coherent cohomology. Then, the colimit diagram
        $$\xymatrix{
            \Xscr'\ar[r]^{g'}\ar[d]^{f'}&\Xscr\ar[d]^f\\
            \Yscr'\ar[r]^g&\Yscr
        }$$ satisfies base change for (bounded above) quasi-coherent cohomology.
        Moreover, if each $f_i$ satisfies the projection formula, then so does $f$.
    \end{lemma}

    \begin{proof}
        Write $h_i\colon\Xscr_i\rightarrow\Xscr$. Fix $\Fscr\in\D(\Xscr)$ which by definition is
        equivalent to $\lim h_{i*}h_i^*\Fscr$. Let $\Fscr_i=h_i^*\Fscr$. Thus,
        \begin{align*}
            g^*f_*(\Fscr)&\we g^*f_*\lim h_{i*}\Fscr_i\\
            &\we g^*\lim f_*h_{i*}\Fscr_i\\
            &\we \lim g^*f_* h_{i*}\Fscr_i\\
            &\we \lim g^*f_{i*}\Fscr_i\\
            &\we \lim f'_{i*}g_i'^*\Fscr_i\\
            &\we \lim f_{i*}'h_i^*g'^*\Fscr\\
            &\we f_*'g'^*\Fscr,
        \end{align*}
        where the third equivalence is because $g^*$ commutes with finite limits and the fifth is
        by the assumption of base change. The proof of the claim about the projection formula is
        similar using that left adjoint functors between stable $\infty$-categories commute with
        finite limits.
    \end{proof}

    Now, one reduces the proof of the proposition using Lemma~\ref{lem:finitecolimitbasechange} to
    the case where $X=\Spf R$ by induction on the number of affines needed to cover $X$.
    Indeed, assume the proposition is true for formal schemes which can be covered by at most $n$
    affines, and assume that $X$ can be covered by $n+1$ affines $\{\Spf R_i\}_{i=1}^{n+1}$.
    Then, let $Y=\cup_{i=1}^{n}\Spf R_i$ and $Z=\Spf R_{n+1}$. The intersection $W=Y\cap Z$ satisfies
    the hypotheses of the lemma, and is covered by $n$ affines. Moreover,
    $$\xymatrix{
        \WCart_{W/A}\ar[r]\ar[d]&\WCart_{Z/A}\ar[d]\\
        \WCart_{Y/A}\ar[r]&\WCart_{X/A}
    }$$ is a pushout in formal stacks and similarly for the base change to $B$. Now, apply the
    lemma.

    Now, assume that $X=\Spf R$ and let $R\rightarrow R^0$ be a quasisyntomic cover where $R^0$ is
    relatively quasiregular semiperfectoid over $\overline{A}$. Let $R^\bullet$ be the resulting
    \v{C}ech complex. By~\cite[Thm.~7.17]{bhatt-lurie-prism}, there are natural equivalences
    $$\WCart_{R\tensorhat_{\overline{A}}\overline{B}/B}\we\SSpf\Prism_{R\tensorhat_{\overline{A}}\overline{B}/B}\quad\text{and}\quad\WCart_{R/A}\we\SSpf\Prism_{R/A},$$
    where $\SSpf\Prism_{R/A}$ is the functor $\Map_{\DAlg_A}(\Prism_{R/A},-)$ restricted to
    $(p,I)$-nilpotent $A$-algebras and similarly for
    $\SSpf\Prism_{R\tensorhat_{\overline{A}}\overline{B}/B}$.
    Throughout, we use that $\Prism_{R/A}\we\Tot\Prism_{R^\bullet/A}$ and that
    $|\SSpf\Prism_{R^\bullet/A}|\we|\SSpf\Prism_{R/A}|$ and similarly for the base change to
    $B$. The first equivalence is by descent for prismatic cohomology; the second is because
    $\WCart_{R^0/A}\rightarrow\WCart_{R/A}$ is a surjection of flat sheaves
    (by~\cite[Lem.~6.3]{bhatt-lurie-prism}) and by compatibility of $\WCart_{-/A}$
    with pullbacks (which follows from~\cite[Rem.~3.5]{bhatt-lurie-prism}).

    \begin{warning}
        Note that despite the apparent affine behavior of $\WCart_{R/A}$, it is not typically the
        case that the global sections functor
        $\D(\SSpf\Prism_{R/A})\rightarrow\D(\Prism_{R/A})_{p,I}^\wedge$ is an equivalence. For example $\WCart = \WCart_{\bZ_p/\bZ_p}$ itself is not affine.
        This complicates the
        proof of base change. However, when $\Prism_{R/A}$ is discrete, taking global sections
        is an equivalence.
    \end{warning}

    Thus, the pullback square of the statement of the proposition is equivalent in the affine case to a
    pullback square
    $$\xymatrix{
        \SSpf\Prism_{R\tensorhat_{\overline{A}}\overline{B}/B}\ar[r]^{g'}\ar[d]^{f'}&\SSpf\Prism_{R/A}\ar[d]^f\\
        \SSpf B\ar[r]^g&\SSpf A.
    }$$
    We prove a lemma giving some properties of this square.

    \begin{lemma}\label{lem:random}
        \begin{enumerate}
            \item[{\em (a)}] The map $g'$ satisfies the projection formula.
            \item[{\em (b)}] The map $f$ satisfies the projection formula for bounded above objects
                with respect to $\Tscr_A$.
            \item[{\em (c)}] The natural map $f^*g_* B\rightarrow g'_*f'^*B$ is an
                equivalence. That is, the {\em transposed square} satisfies base change for
                $\Oscr$-cohomology.
        \end{enumerate}
    \end{lemma}

    \begin{proof}
        For parts (a) and (c) we use the commutative diagram
        $$\xymatrix{
            \SSpf\Prism_{R^\bullet\tensorhat_{\overline{A}}\overline{B}/B}\ar[r]\ar[d]^{g'_\bullet}&\SSpf\Prism_{R\tensorhat_{\overline{A}}\overline{B}/B}\ar[d]^{g'}\ar[r]&\SSpf
            B\ar[d]^g\\
            \SSpf\Prism_{R^\bullet/A}\ar[r]&\SSpf_{R/A}\ar[r]&\SSpf A.
        }$$
        The left vertical arrow represents a simplicial diagram in morphisms of stacks satisfying
        the projection formula. The induced pullback functors on quasi-coherent sheaves gives a
        cosimplicial diagram in $\Fun^{\mathrm{RAd}}(\Delta^1,\Pr^\L)$ with limit the pullback map
        on quasi-coherent sheaves associated to the middle vertical arrow. It follows from
        Variant~\ref{var:colimitbasechange}
        that the left-hand square satisfies base change for quasi-coherent sheaves in each
        simplicial degree. Thus, by Lemma~\ref{lem:colimitprojection}, the map $g'$ satisfies the projection
        formula. As the exterior square satisfies base change for quasi-coherent cohomology, so does
        the right-hand square by taking a limit in $\Fun^{\mathrm{RAd}}(\Delta^1,\Pr^\L)$.

        For part (b), suppose that $\Fscr\in\D(\SSpf\Prism_{R/A})$ and $\Gscr\in\D(\SSpf A)$ are bounded
        above and assume that $\Gscr/(p,I)$ can be written as a filtered colimit of uniformly
        bounded above perfect complexes of $A/(p,I)$-modules. The natural map
        $$f_*(\Fscr)\tensorhat_A\Gscr\rightarrow f_*(\Fscr\tensor_{\Oscr_{\SSpf\Prism_{R/A}}}
        f^*(\Gscr))$$ can be written as
        \begin{equation}\label{eq:tot}
            \Tot(\Fscr^\bullet)\tensorhat_A\Gscr\rightarrow\Tot(\Fscr^\bullet\tensorhat_A\Gscr),
        \end{equation}
        where $\Fscr^\bullet$ denotes the pullback of $\Fscr$ to $\SSpf\Prism_{R^\bullet/A}$.
        Since $\Gscr\mapsto\Gscr/(p,I)$ preserves all limits, we can test whether~\eqref{eq:tot} is
        an equivalence after modding out by $(p,I)$. But, $$\Tot(\Fscr^\bullet)\otimes_A
        A/(p,I)\otimes_{A/(p,I)}\Gscr'\rightarrow\Tot(\Fscr^\bullet\otimes_A
        A/(p,I)\otimes_{A/(p,I)}\Gscr')$$ is an equivalence for any perfect complex $\Gscr'$ over
        $A/(p,I)$. As $\Gscr/(p,I)$ is a filtered colimit of perfect complexes which is uniformly
        bounded above, and as uniformly bounded above filtered colimits commute with totalizations,
        we see that~\eqref{eq:tot} holds for $\Gscr\in\Tscr_A$.
    \end{proof}

    Returning to the proof of the proposition in the case when $X=\Spf R$ is affine, consider the
    natural map $g^*f_*\Fscr\rightarrow f'_*g'^*\Fscr$ for $\Fscr\in\D(\SSpf\Prism_{R/A})$ bounded
    above. As $g_*$ is conservative, to test if this map is an equivalence, we can apply $g_*$ to
    obtain $g_*g^*f_*\Fscr\rightarrow g_*f'_*g'^*\Fscr$, which can be rewritten as a map
    $$f_*\Fscr\tensorhat_AB\rightarrow f_*g'_*g'^*\Fscr\we
    f_*(g'_*\Oscr_{\SSpf\Prism_{R\tensorhat_{\overline{A}}\overline{B}/B}}\otimes_{\Oscr_{\SSpf\Prism_{R/A}}}\Fscr)\we
    f_*(f^*g_*B\otimes_{\Oscr_{\SSpf\Prism_{R/A}}}\Fscr)\we g_*B\tensorhat_A f_*\Fscr,$$
    where the first equivalence is by part (a) of Lemma~\ref{lem:random}, the second by part (c), and
    third by part (b), which uses that $B\in\Tscr_A$ by assumption. One checks that the map is indeed the natural equivalence between the left
    and right-hand sides, which completes the proof. Finally, the claim about the projection formula
    is Lemma~\ref{lem:random}(b).
\end{proof}

Let $B^0$ be a transversal prism and let $\rho\colon\SSpf B^0\rightarrow\WCart$ be the prismatic structure
map. Let $\Tscr_{\WCart}\subseteq\D(\WCart)$ be the full subcategory of those $\Gscr$ such that
$\rho^*\Gscr\in\Tscr_{B^0}$. Using~\cite[Prop.~2.4.9]{bhatt-lurie-apc}, one can show that $\Tscr_{\WCart}$ is
independent of the transversal prism $B^0$.

\begin{corollary}\label{cor:1}
    Suppose that $A$ is a $\delta$-ring and $X$ is a quasisyntomic $p$-adic formal scheme over $\Spf
    A$. If $X$ is quasicompact and separated and $\L_{X/A}$ has $p$-complete Tor-amplitude in $[0,1]$, then for every transversal prism
    $B^0$ the pullback square
    $$\xymatrix{
        \WCart_{X_{\overline{B^0}}/A\tensorhat B^0}\ar[r]\ar[d]&\WCart_{X/A}\ar[d]\\
        \SSpf B^0\ar[r]&\WCart
    }$$ satisfies base change for bounded above quasi-coherent cohomology. Moreover,
    $\WCart_{X/A}\rightarrow\WCart$ satisfies the bounded above projection formula with respect to
    $\Tscr_{\WCart}$.
\end{corollary}

\begin{proof}
    Let $\SSpf B^\bullet\rightarrow\WCart$ be the \v{C}ech complex of $\SSpf B^0\rightarrow\WCart$.
    Pulling back $\WCart_{X/A}\rightarrow\WCart$ along $\SSpf B^\bullet\rightarrow\WCart$, we obtain
    a commutative diagram
    $$\xymatrix{
        \WCart_{X_{\overline{B^\bullet}}/A\tensorhat B^\bullet}\ar[r]\ar[d]&\WCart_{X/A}\ar[d]\\
        \SSpf B^\bullet\ar[r]&\WCart
    }$$
    where the vertical arrow on the left defines a simplicial object in morphisms of formal stacks.
    For each $[m]\rightarrow[n]$ in $\Delta^\op$ the pullback square
    $$\xymatrix{
        \WCart_{X_{\overline{B^m}}/A\tensorhat B^m}\ar[r]\ar[d]&\WCart_{X_{\overline{B^n}}/A\tensorhat B^n}\ar[d]\\
        \SSpf B^m\ar[r]&\SSpf B^n
    }$$
    satisfies base change for bounded above quasi-coherent sheaves by
    Proposition~\ref{prop:basicsquare} and the right-hand morphism satisfies the projection formula
    for bounded above objects. The desired base change and projection formula result by
    Variant~\ref{var:colimitbasechange}.
\end{proof}

Fix a transversal prism $C^0$ and
let $\Tscr_{\Spf A\times\WCart}\subseteq\D(\Spf A\times\WCart)$
be the full subcategory of objects $\Gscr$ whose pullback to $\Spf A\times\SSpf C^0$ is in
$\Tscr_{(A\otimes C^0)_{(p,I)}^\wedge}$.

\begin{corollary}\label{cor:2}
    Let $A\rightarrow R$ be a relatively quasisyntomic $\delta$-pair and let $A\rightarrow B$ be a
    map of $\delta$-rings which has $p$-complete bounded Tor-amplitude.
    The pullback square
    $$\xymatrix{
        \WCart_{R\tensorhat_AB/B}\ar[r]^{g'}\ar[d]^{f'}&\WCart_{R/A}\ar[d]^f\\
        \Spf B\times\WCart\ar[r]^g&\Spf A\times\WCart
    }$$
    satisfies base change for quasi-coherent cohomology and $f$ satisfies the projection formula for
    bounded above quasi-coherent cohomology with respect to $\Tscr_{\Spf A\times\WCart}$.
\end{corollary}

\begin{proof}
    Fix a transversal prism $C^0$ and consider the \v{C}ech complex $\SSpf C^\bullet\rightarrow\WCart$
    of $\rho_{C^0}\colon\SSpf C^0\rightarrow\WCart$.
    Pulling back along $\Spf A\times\SSpf C^\bullet\rightarrow\Spf A\times\WCart$ for a
    transversal prism point $\rho_C\colon\SSpf C\rightarrow\WCart$, the resulting cosimplicial pullback square is equivalent to
    $$\xymatrix{
        \WCart_{R\tensorhat_AB\tensorhat\overline{C^\bullet}/B\tensorhat
        C^\bullet}\ar[r]^{g'^\bullet}\ar[d]^{f'^\bullet}&\WCart_{R\tensorhat\overline{C^\bullet}/A\tensorhat
        C^\bullet}\ar[d]^{f^\bullet}\\
        \SSpf(B\tensorhat C^\bullet)\ar[r]^{g^\bullet}&\SSpf(A\tensorhat C^\bullet).
    }$$
    Each of these squares satisfies the hypothesis of Proposition~\ref{prop:basicsquare} by our assumption of
    the $p$-complete bounded Tor-amplitude of $B$ as an $A$-module and hence it satisfies base
    change for bounded above quasi-coherent cohomology. We would like to now apply
    Lemma~\ref{lem:colimitbasechange}, but we need a variant which applies in the bounded above
    setting, along the lines of Variant~\ref{var:colimitbasechange}.
    Fix
    \begin{align*}
        a^\bullet&\colon\WCart_{R\tensorhat\overline{C}^\bullet/A\tensorhat C}\rightarrow\WCart_{R/A},\\
        b^\bullet&\colon\SSpf(A\tensorhat C^\bullet)\rightarrow\Spf A\times\WCart,\\
        a'^\bullet&\colon\WCart_{R\tensorhat_AB\tensorhat\overline{C}^\bullet/B\tensorhat C}\rightarrow\WCart_{R\tensorhat_A B/B},\\
        b'^\bullet&\colon\SSpf(B\tensorhat C^\bullet)\rightarrow\Spf B\times\WCart,
    \end{align*}
    Consider $\Fscr\in\D(\WCart_{R/A})$, which is equivalent to $\lim_{[i]\in\Delta}a^i_*
    a^{i*}\Fscr$. Let $\Fscr_i=a^{i*}\Fscr$. By Variant~\ref{var:colimitbasechange},
    $b^{i*}f_*\Fscr\we f_{i*}\Fscr_i$ for all $i$, naturally in $i$. Thus, the pullback
    $g^*f_*\Fscr$ is determined by the compatible family $g^{\bullet*}f_{\bullet*}\Fscr_\bullet$ (as an
    appropriate cartesian section corresponding to an object of $\D(\Spf
    B\times\WCart)\we\lim\SSpf(B\tensorhat C)$. However, by base change for bounded above
    quasi-coherent cohomology in each (co)simplicial degree, $g^{\bullet*}f_{\bullet*}\Fscr_\bullet\we
    f'_{\bullet*}g'^{\bullet*}\Fscr_\bullet$, which is checked in the same way to yield $f'_*
    g'^*\Fscr$ upon passage to limits using Variant~\ref{var:colimitbasechange} again.
\end{proof}

\small
\bibliographystyle{amsplain}
\bibliography{kzpn}

\medskip
\noindent
\textsc{Department of Mathematics, Northwestern University}\\
{\ttfamily antieau@northwestern.edu}

\medskip
\noindent
\textsc{FB Mathematik und Informatik, Universit\"at M\"unster}\\
{\ttfamily krauseac@uni-muenster.de}

\medskip
\noindent
\textsc{FB Mathematik und Informatik, Universit\"at M\"unster}\\
{\ttfamily nikolaus@uni-muenster.de}

\end{document}